\documentclass[11pt]{article}
\def\x{2.4}
\usepackage[lmargin=\x cm,rmargin=\x cm,tmargin=\x cm,bmargin=\x cm,bindingoffset=0cm]{geometry}

\usepackage{comment}
\usepackage{pgfplots}
\pgfplotsset{compat=1.18}
\usepackage{graphicx,subcaption,float}
\usepackage[]{hyperref}
\hypersetup{colorlinks=true}

\usepackage{amsthm,mathdots}
\usepackage{nicefrac}
\usepackage{amsmath}
\usepackage{amssymb}
\usepackage{amsfonts}
\usepackage{mathrsfs}
\makeatletter
\renewcommand{\@fnsymbol}[1]{\@arabic{#1}}
\makeatother

\usepackage{stmaryrd}
\usepackage{graphicx} 
\usepackage{xcolor}
\usepackage{fancyhdr}
\usepackage[ansinew]{inputenc}

\usepackage{tikz}
\usepackage{tikz-3dplot}
\usepackage{tikz-cd}
\usetikzlibrary{patterns}
\usetikzlibrary{matrix} 
\usepackage{multicol} 
\usepackage{float} 
\usepackage{makeidx} 
\usepackage{hyperref}
\usepackage[english]{babel}

\usepackage[nomessages]{fp}
\usepackage{mathtools}

\newtheorem{Theorem}{Theorem}[section]
\newtheorem{Proposition}[Theorem]{Proposition}
\newtheorem{Lemma}[Theorem]{Lemma}
\newtheorem{Corollary}[Theorem]{Corollary}

\newtheorem{Definition}[Theorem]{Definition}

\newtheorem{Remark}[Theorem]{Remark}

\newcounter{IssueCounter}
\newtheorem{Issue}[IssueCounter]{Issue}

\newcounter{QuestionCounter}
\newtheorem{Question}[QuestionCounter]{Question}

\definecolor{color1}{rgb}{0.1,0.3,0.0}

\newcommand{\changes}[1]{{\color{black}#1}}

\newcounter{ReminderCounter}
\newtheorem{Reminder}[ReminderCounter]{Reminder}

\makeatletter
\newcommand{\xMapsto}[2][]{\ext@arrow 0599{\Mapstofill@}{#1}{#2}}
\def\Mapstofill@{\arrowfill@{\Mapstochar\Relbar}\Relbar\Rightarrow}
\makeatother

\newcommand{\mb}[1]{\mathbb{#1}}
\newcommand{\mc}[1]{\mathcal{#1}}
\newcommand{\mf}[1]{\mathfrak{#1}}

\newcommand{\Ad}{\mathrm{Ad}}

\newcommand{\diag}{\mathrm{diag}}

\definecolor{color2}{rgb}{0.30,0.0,0.9}

\newcommand{\A}{\mathrm{Aut}}
\newcommand{\g}{{\mathfrak{g}}}
\newcommand{\ox}{\mathcal{O}_{\mathbb{X}}}
\newcommand{\CC}{{\mathbb{C}}}
\newcommand{\RR}{{\mathbb{R}}}
\newcommand{\LL}{{\Lambda}}
\newcommand{\ZZ}{{\mathbb{Z}}}

\newcommand{\twp}{\tilde{\wp}}
\newcommand{\NN}{\mathbb{N}}

\newcommand{\ot}{\mathcal{O}_{\mathbb{T}}}
\newcommand{\TT}{{\mathbb{T}}}

\begin{document}
\title{Normal forms of elliptic automorphic Lie algebras and Landau--Lifshitz type of equations}
\author{
Sara Lombardo\thanks{Maxwell Institute for Mathematical Sciences and Department of Mathematics, Heriot-Watt University, Edinburgh, EH14 4AS, UK.}
\and
Casper Oelen\footnotemark[1] \textsuperscript{,}\thanks{Corresponding author. Email: \texttt{c.oelen@hw.ac.uk}}
}

\maketitle

\begin{abstract}
We present normal forms of elliptic automorphic Lie algebras with dihedral symmetry of order 4, which arise naturally in the context of Landau--Lifshitz type of equations. These normal forms provide a transparent description and allow a classification of such Lie algebras over $\CC$. Using this perspective, we show that a Lie algebra introduced by Uglov, as well as the hidden symmetry algebra of the Landau--Lifshitz equation introduced by Holod, can each be realised as an elliptic $\mf{sl}(2,\CC)$-current algebra. 
Furthermore, we realise the Wahlquist--Estabrook algebra of the Landau--Lifshitz equation in terms of elliptic automorphic Lie algebras. This construction reveals that, as a complex Lie algebra, it is isomorphic to the direct sum of an $\mf{sl}(2,\CC)$-current algebra and the two-dimensional abelian Lie algebra $\CC^2$. Finally, we apply the automorphic Lie algebra framework to an $n$-component generalisation of the Landau--Lifshitz equation due to Golubchik and Sokolov in the case $n=3$.
\end{abstract}

\section{Introduction}
Automorphic Lie algebras are Lie algebras of invariants that first emerged in the context of integrable systems. They originated in the study of algebraic reductions of Lax pairs by Lombardo and Mikhailov \cite{MR2090010,MR2166845},  related to the notion of reduction groups, proposed by Mikhailov in \cite{mikhailov1979integrability} and \cite{mikhailov1980reduction}.  While a precise definition will be given in Section \ref{sec:Basics}, for now the reader may think of automorphic Lie algebras as Lie algebras of meromorphic maps (usually with prescribed poles) from a compact Riemann surface $X$ into a finite-dimensional Lie algebra $\g$, which are equivariant with respect to a finite group $\Gamma$  acting on $X$ and on $\g$, both by automorphisms. The group $\Gamma$ plays the role of the reduction group in the original context of integrable systems.

\changes{While we emphasise here the equivariant (automorphic) property of these algebras, it is helpful to set them in the context of other infinite dimensional Lie algebras which are important in mathematical physics; automorphic Lie algebras are obtained by imposing a discrete group symmetry on a current algebra of Krichever--Novikov (KN) type. KN algebras \cite{MR902293, MR1082031, 10.1063/1.530700}  provide a general geometric framework for infinite-dimensional Lie algebras built from meromorphic data on curves, while automorphic Lie algebras are obtained by taking invariants of KN-type current algebras under finite reduction groups. They share the same geometric origin, almost-grading, and representation-theoretic features, but automorphic Lie algebras encode additional discrete symmetries. Although automorphic Lie algebras are sometimes described as `generalised Kac-Moody (KM) algebras' (replacing the cyclic group of the KM algebra with the reduction group $\Gamma$), this is not accurate in Kac's sense, in that KM algebras are defined by generalised Cartan matrices while automorphic Lie algebras are defined by invariant-theoretic data and root systems, Weyl groups, and reflection geometry are no longer intrinsic. In other words, automorphic Lie algebras generalise loop-type realisations, not Cartan-matrix theory.}

\changes{
Automorphic Lie algebras
where} $X$ has genus 0 have been extensively studied in the past two decades by Knibbeler, Lombardo and Sanders \cite{MR2718933,MR3254844,MR3682219,MR4072405,MR4142334}. More recently, hyperbolic automorphic Lie algebras have been introduced by Knibbeler, Lombardo and Veselov \cite{MR4565687}. The study of automorphic Lie algebras on genus 1 Riemann surfaces, also known as elliptic automorphic Lie algebras, has been initiated in a PhD thesis by Oelen \cite{oelen2022automorphic}, which resulted in a classification by Knibbeler, Lombardo and Oelen in \cite{MR4808734}. One of the main objectives of the theory is to obtain certain normal forms of these algebras, which make it possible to identify their Lie-isomorphism types. In this context, the recent work by Knibbeler \cite{knibbeler2025uniform} contains a uniform construction of normal forms in the case of genus 0.

Prior to the systematic study of elliptic automorphic Lie algebras, examples have appeared in the literature, particularly in the context of integrable systems\changes{; this is not entirely surprising, given their relation to KN algebras illustrated above.} For example, Reiman and Semenov-Tyan-Shanskii \cite{MR861265} introduced Lie algebras of automorphic, meromorphic $\mf{sl}(n,\CC)$-valued functions on a complex torus in relation to Lax equations with
spectral parameter on an elliptic curve. Furthermore, Uglov \cite{MR1276451} studied the $\mf{sl}(2,\CC)$ case, and addressed applications to quantisation and  quantum groups related to elliptic R-matrices \changes{\cite{sklyanin1979complete}} \changes{(see Section \ref{sec:NormalForms} for more reference and historical notes).} 
More recently, elliptic automorphic Lie algebras with $\g=\mf{sl}(n,\CC)$ were investigated by Skrypnyk in \cite{MR2920470}, in which various algebraic constructions related to these Lie algebras were carried out.

In the current paper, we present a new method for obtaining \changes{certain normal forms} and for classifying automorphic Lie algebras that have appeared frequently in the context of integrable systems. Namely, those algebras with $\g=\mf{sl}(2,\CC)$, $X$ a complex torus, and symmetry group $\Gamma =D_2$, where $D_2$ is the dihedral group of order 4 -- also known as Klein's four group. We shall assume throughout the current paper that $D_2$ is embedded as translations in $\A(X)$. In Section \ref{sec:Basics} we describe this new method, which employs theta functions, and which has the benefit of being more transparent from an algebraic point of view, whilst having the potential of being generalisable to higher rank Lie algebras and different symmetry groups. \changes{The normal forms show that the automorphic Lie algebras under consideration are KN-type current algebras on an elliptic curve.} 

We extend $D_2$-automorphic Lie algebras with one orbit of poles to those with finite unions of orbits of poles. 
Moreover, we consider a broader class of automorphic Lie algebras with the same geometric and group-theoretic setup but with $\g$ an arbitrary complex reductive Lie algebra. Finally, we briefly touch upon the case of $\g=\mf{sl}(2,\RR)$, where we consider a type of fixed-point Lie subalgebra in the above context. 

In Section \ref{sec:NormalForms}, we discuss the relation between a $D_2$-automorphic Lie algebra based on $\mf{sl}(2,\CC)$ and a Lie  algebra called $\mc{E}_{k,\nu^{\pm}}$ introduced by Uglov in \cite{MR1276451}. We will show that the geometric realisation of $\mc{E}_{k,\nu^{\pm}}$ as an elliptic automorphic Lie algebra ($\nu^{\pm}\in \CC/\LL$, and $k$ is the elliptic modulus corresponding to $\LL$) is isomorphic to an elliptic $\mf{sl}(2,\CC)$-current algebra. 
More specifically, we establish an isomorphism of Lie algebras
 $$\mc{E}_{k,\nu^{\pm}}\cong \mf{sl}(2,\CC)\otimes_{\CC}\CC[\wp_{\frac{1}{2}\LL},\wp_{\frac{1}{2}\LL}', \xi],$$ where $\wp_{\frac{1}{2}\LL}$ is the Weierstrass $\wp$-function associated to the lattice $\frac{1}{2}\LL$ and $\xi$ is a function on $\CC/\frac{1}{2}\LL$ with two simple poles depending on $\nu^-,\nu^{+}$.
 The Lie (bi)algebra $\mc{E}_{k,\nu^{\pm}}$ can be quantised and the corresponding quantum bialgebra is related to the eight-vertex model R-matrix \cite{MR1229896}. 
 
  We also find a new basis of the hidden symmetry algebra of the Landau--Lifshitz equation by Holod \cite{holod1987hidden}, denoted $\mc{H}_{r_1,r_2,r_3}$,  defined on the elliptic curve $$E_{r_1,r_2,r_3}: \lambda_i^2-\lambda_j^2=r_j-r_i, \quad i,j=1,2, 3,$$ in $\CC^3$, where $r_1,r_2,r_3$ are pairwise distinct constants. In this basis, $\mc{H}_{r_1,r_2,r_3}$ is revealed to be isomorphic to an elliptic current algebra:  $$\mc{H}_{r_1,r_2,r_3}\cong \mf{sl}(2,\CC)\otimes_{\CC}\CC[x,x^{-1},y]/(y^2-(x-A_1)(x-A_2)(x-A_3)),$$
where the constants $A_i$ are given by $A_i=r_i-\frac{1}{3}(r_1+r_2+r_3)$. Finally, we prove that Holod's algebra $\mc{H}_{r_1,r_2,r_3}$ is isomorphic to Uglov's algebra $\mc{E}_{\frac{1}{\sqrt{2}},0,\frac{1+\mathrm{i}}{4}}$, whenever the underlying curve $E_{r_1,r_2,r_3}$ is isomorphic to a complex torus corresponding to the square lattice $\ZZ+\ZZ\mathrm{i}$.

In Section \ref{sec:LL}, we study a relation between integrable partial differential equations (PDEs) and (elliptic) automorphic Lie algebras. In \cite{MR3035112}, the authors consider a PDE which can be viewed as a multicomponent generalisation of the Landau--Lifshitz equation, originally introduced in \cite{MR1821313}. To present this equation, let us first introduce some notation. Let $\mathbb{K}$ be either $\CC$ or $\RR$ and let $n\geq 2$ be an integer.
We write $S=(s^1(x,t),\ldots, s^n(x,t))^T\in \mathbb{K}^n$, and fix pairwise distinct constants $r_1,\ldots, r_n\in \mathbb{K}$. For vectors $V=(v_1,\ldots, v_n)^T$, $W=(w_1,\ldots,w_n)^T\in \mathbb{K}^n$, let $\langle V,W \rangle=\sum_{i=1}^nv_iw_i$. The PDE under consideration is given by
 \begin{equation}\label{eq:PDE}
S_t=\left(S_{xx}+\frac{3}{2}\langle S_x,S_x\rangle S\right)_x+\frac{3}{2}\langle S,RS\rangle S_x,\quad \langle S,S\rangle=1,
\end{equation}
where $R=\diag(r_1,\ldots,r_n).$ In this context, we mention the fully anisotropic Landau--Lifshitz equation \cite{MR1821313} 
\begin{equation}\label{eq:LL}
u_t=u\times u_{xx}+Ru\times u, \quad \langle u,u\rangle=1,
\end{equation}
where $u=u(x,t)\in \mathbb{K}^3$, the symbol $\times$ denotes the cross product and $R=\diag(r_1,r_2,r_3)$ encodes the anisotropy.
For $n=3$, system \eqref{eq:PDE} coincides with the higher symmetry of third order for the Landau--Lifshitz equation.
A zero-curvature representation (ZCR) for \eqref{eq:PDE} for $n\geq 3$ is constructed with spectral parameter belonging to the algebraic curve 
\begin{equation}\label{eq:algcurve}
E_{r_1,\ldots, r_n}:  \lambda_i^2-\lambda_j^2=r_j-r_i,\quad i,j=1,\ldots, n,
\end{equation}
in $\mathbb{K}^n$.

The PDE \eqref{eq:PDE} has an infinite number of symmetries and conservation laws \cite{MR1821313}. An auto-B\"acklund transformation and soliton-like solutions were obtained in  \cite{MR2189890}. This PDE and its symmetries have been obtained in \cite{MR2105209} by means of the AKS (Adler--Kostant--Symes) scheme. 

In this paper, we show that for $\mathbb{K}=\CC$ and in the elliptic case ($n=3$), the ZCR for  \eqref{eq:PDE} admits a natural interpretation in terms of elliptic automorphic Lie algebras. This can be understood as follows.
On curves of positive genus, constructing ZCRs presents a difficulty as a consequence of the Riemann--Roch theorem, cf. \cite{MR725411}. 
One way around this difficulty is based on an ansatz of the form of the Lax pair related to Mikhailov's reduction group. More precisely, one imposes a reduction on the Lax pair or ZCR by means of the action of a (finite) group, see also \cite[Remark 3.3, 3.4]{babelon2003introduction} in this context. 
This is where (elliptic) automorphic Lie algebras naturally appear -- for example in the Lax pair of the Landau--Lifshitz equation and in various systems described in  \cite{MR861265}. 

In Sections \ref{sec:NormalForms} and \ref{sec:LL}, we also consider a Lie algebra that arises in connection with the prolongation algebras of the Landau--Lifshitz equation and Krichever--Novikov equation. We show that this algebra, denoted by $\mf{R}_{r_1,r_2,r_3}$ (see e.g. \cite{MR4056196}, and denoted by $R$ in \cite{MR1218995, MR1947240}), for pairwise distinct complex constants $r_i$, can be realised as an automorphic Lie algebra. Using this realisation, we establish an isomorphism of Lie algebras 
$$\mf{R}_{r_1,r_2,r_3}\cong  \mf{sl}(2,\mathbb{C})\otimes_{\mathbb{C}}R,$$
where $R$ is the ring $R=\mathbb{C}[x,y]/(y^2-(x-r_1)(x-r_2)(x-r_3))$.

This isomorphism shows that the Wahlquist--Estabrook (WE) algebra \cite{MR358111} of the Landau--Lifshitz equation, which is isomorphic to the direct sum of the $\mf{R}_{r_1,r_2,r_3}$ and the abelian Lie algebra $\CC^2$ \cite{MR1218995}, has a particularly simple structure. Consequently, a similar observation applies to the prolongation algebra (generalised WE algebra) of the non-singular Krichever--Novikov equation \cite{MR1947240}. 

The WE algebras of the multicomponent Landau--Lifshitz equations \eqref{eq:PDE} are closely related to certain complex Lie algebras $\g(n)$ \cite{MR3035112}.  The algebra $\g(3)$ is generated over $\mathbb{K}$ by elements $p_1,p_2,p_3$ subject to the relations 
$$
[p_i,[p_j,p_k]]=0,\quad 
[p_i,[p_i,p_k]]-[p_j,[p_j,p_k]]=(r_j-r_i)p_k, 
$$
where $(i,j,k)$ is a cyclic permutation of $(1,2,3)$.  It is known that $\g(3)\cong \mf{R}_{r_1,r_2,r_3}$.  

Although we focus on the case $\mathbb{K}=\CC$ and $n=3$ in this paper, we keep the discussion more general in some places to indicate that the methodology can, in principle, be applied to general $n$.
Understanding the structure of the algebras $\g(n)$ is important because of their connection to WE algebras, which is relevant in the classification of the multicomponent Landau--Lifshitz equations \eqref{eq:PDE}, cf. \cite{MR3035112}. Normal forms of the elliptic automorphic Lie algebras studied in \cite{MR4808734} and in the present paper reveal fundamental properties of the Lie algebra $\g(3)$, highlighting the use of the Lie algebras and normal forms introduced here.

\section{Basic notions and constructions}\label{sec:Basics}
In this section, we first briefly introduce the concept of automorphic Lie algebras after which we discuss the necessary functional aspects required for the construction of an intertwining operator. We apply this operator to obtain certain normal forms analogous to the Chevalley normal form of finite-dimensional simple Lie algebras, and establish isomorphisms between a class of automorphic Lie algebras and elliptic current algebras. Finally, we touch upon a generalisation of our constructions to real Lie algebras of invariants.

Let $\g$ be a complex finite-dimensional Lie algebra and let $X$ be a compact Riemann surface. Let $\Gamma$ be a finite group.
Suppose that $\sigma:\Gamma\rightarrow \A(X)$ is an injective homomorphism, where $\A(X)$ denotes the automorphism group of $X$ (i.e., the group of biholomorphic maps $X\rightarrow X$). Assume that $\rho:\Gamma \rightarrow \A(\g)$ is a representation, where $\A(\g)$ is the group of Lie algebra automorphisms of $\g$. \changes{Fix $n\in \mathbb{N}$, $p_1,\ldots,p_n \in X$, and let $\mathbb{X}=X\setminus \bigcup_{i=1}^n\sigma(\Gamma)\cdot \{p_i\}$}. Thus $\mb{X}$ is the Riemann surface obtained from $X$ by removing a finite union of $\Gamma$-orbits. We denote by $\ox$ the ring of regular functions on $\mathbb{X}$; that is, the ring of meromorphic functions on $X$ that are holomorphic on $\mathbb{X}$. Let $\tilde{\sigma}$ be the induced homomorphism on $\ox$, defined by $\tilde{\sigma}(\gamma) f=f\circ \sigma(\gamma^{-1})$, for $f\in \ox$. 

We consider the current algebra $\g\otimes_{\CC}\ox$, equipped with the Lie bracket $$[A\otimes f, B\otimes g]=[A,B]\otimes fg,$$
for $A,B\in \g$ and $f,g\in \ox$, extended $\CC$-linearly.  The group $\Gamma$ acts diagonally on $\g\otimes_{\CC}\ox$ via
\begin{equation}\label{ALiaAction}
\rho(\gamma)\otimes \tilde{\sigma}(\gamma)(A\otimes f)= \rho(\gamma)A\otimes \tilde{\sigma}(\gamma)f,
\end{equation}
 for all $A\in \g, f\in \ox$, and $\gamma\in \Gamma$.
The automorphic Lie algebra $(\g\otimes_{\CC}\ox)^{\rho\otimes \tilde{\sigma}(\Gamma)}$ is the fixed-point Lie subalgebra of $\g\otimes_{\CC}\ox$ with respect to the action of $\Gamma $ as defined in \eqref{ALiaAction}:
\begin{align*}
(\g\otimes_{\CC}\ox)^{\rho\otimes \tilde{\sigma}(\Gamma)}=\{a\in \g\otimes_{\CC}\ox: \rho(\gamma)\otimes \tilde{\sigma}(\gamma)a =a,\,\, \forall \gamma\in \Gamma\}.
\end{align*}
Equivalently, we can view $(\g\otimes_{\CC}\ox)^{\rho\otimes \tilde{\sigma}(\Gamma)}$ as the Lie algebra of holomorphic maps $\varphi:\mathbb{X}\rightarrow \g$, meromorphic at $X\setminus \mathbb{X}$, that are $\Gamma $-equivariant in the sense that $$\varphi(\sigma(\gamma)z)=\rho(\gamma)\varphi(z)$$
for all $\gamma \in \Gamma$ and $z\in \mb{X}$, with the pointwise bracket inherited from $\g$.

A classification of automorphic Lie algebras with $\g=\mf{sl}(2,\CC)$, $X$ a complex torus, and $\Gamma$ embedding simultaneously in $\A(\g)$ and $\A(X)$, such that $X$ is punctured at precisely one orbit of $\Gamma$, can be found in \cite{MR4808734}.
In this section, we will be concerned with the case where $\g$ is a finite-dimensional complex reductive Lie algebra and where $X$ is a complex torus $\CC/\LL$. The group $\Gamma$ will be the dihedral group $D_2$ of order 4. The full setup is described below.

Next, we construct certain normal forms of a class of automorphic Lie algebras and establish their Lie algebra isomorphism classes over $\CC$. Let us first introduce the objects we will be working with. Let $X$ be the complex torus $X=\CC/(\ZZ+\ZZ\tau)$, with $\tau\in \mathbb{H}=\{z\in \CC: \mathrm{Im}(z)>0\}$, and consider the punctured torus $$\mb{X}=X\setminus \bigcup_{i=1}^n\sigma(D_2)\cdot\{p_i\},$$ where $p_i\in X$ and $n\in \NN$. The group $D_2$ acts on $X$, via the homomorphism $\sigma$, by translations over two distinct half periods of $X$. We consider representations $\rho:D_2\rightarrow \mathrm{Inn}(\g)$ ($\mathrm{Inn}(\g)$ is the group of inner automorphisms of $\g$) which factor through $PGL(2,\CC)=GL(2,\CC)/\{k\mathrm{Id}: k\in \CC^*\}$ where $\CC^*=\CC\setminus\{0\}$. The automorphic Lie algebras we consider are defined by
$$(\g\otimes_{\CC}\ox)^{\rho\otimes \tilde{\sigma}(D_2)}.$$ 
For simple $\g$, we construct a normal form of these automorphic Lie algebras, showing that they are, in essence, the same as $\g$, except that the base field $\CC$ is replaced by a ring of automorphic functions.  The normal form sheds light on a number of fundamental properties, and will allow us to determine the $\CC$-isomorphism classes of these Lie algebras. It will be obtained through the construction of a suitable $D_2$-equivariant map of $\g\otimes_{\CC} \mathcal{O}_{X\setminus D_2\cdot S}$,  which we refer to as `the intertwiner', defined in terms of theta functions. From this point onwards, we denote $X$ by $T$ and $\mb{X}$ by $\TT$. 

We now describe the ingredients needed for the construction of the intertwiner. 
We consider the theta functions 
\begin{align*}
\theta_{a,b}(z|\tau)=\sum_{k\in \ZZ}\exp\Big\{\pi \mathrm{i} \tau(k+a)^2+2\pi \mathrm{i} (k+a)(z+b)\Big\},
\end{align*}
where $a,b\in \mathbb{R}$. The zeros of $\theta_{a,b}(z|\tau)$ lie precisely in the set $(a+\tfrac{1}{2})\tau+(b+\tfrac{1}{2})+\ZZ+\ZZ\tau.$  The Jacobi Theta functions are defined by
\begin{align*}
\theta_1(z|\tau)=-\theta_{\frac{1}{2},\frac{1}{2}}(z|\tau),\quad
\theta_2(z|\tau)=\theta_{\frac{1}{2},0}(z|\tau), \quad
\theta_3(z|\tau)=\theta_{0,0}(z|\tau), \quad
\theta_4(z|\tau)=\theta_{0,\frac{1}{2}}(z|\tau).
\end{align*}
There exist numerous identities satisfied by theta functions, and for the reader's convenience we list those used in the present paper. Proofs can be found in \cite{MR3350266}.

\begin{subequations}\label{eq:1}
\begin{align}
\theta_1(z+1|\tau)&=-\theta_1(z|\tau), & \theta_1(z+\tfrac{\tau}{2}|\tau)&=\mathrm{i}e^{-\pi \mathrm{i}(z+\tfrac{\tau}{4})}\theta_4(z|\tau),\\ 
\theta_2(z+1|\tau)&=-\theta_2(z|\tau), & \theta_2(z+\tfrac{\tau}{2}|\tau)&=e^{-\pi \mathrm{i}(z+\tfrac{\tau}{4})}\theta_3(z|\tau),  \\
\theta_3(z+1|\tau)&=\theta_3(z|\tau), & \theta_3(z+\tfrac{\tau}{2}|\tau)&=e^{-\pi \mathrm{i}(z+\tfrac{\tau}{4})}\theta_2(z|\tau),\\
\theta_4(z+1|\tau)&=\theta_4(z|\tau), & \theta_4(z+\tfrac{\tau}{2}|\tau)&=\mathrm{i}e^{-\pi \mathrm{i}(z+\tfrac{\tau}{4})}\theta_1(z|\tau).
\end{align}
\end{subequations}
\begin{align}
\theta_3(0|\tau)\theta_4^2(0|\tau)\theta_3(2z|\tau)-\theta_4(0|\tau)\theta_3^2(0|\tau)\theta_4(2z|\tau)&=-2\theta_1^2(z|\tau)\theta_2^2(z|\tau),\label{Identity1}\\
\theta_3(0|\tau)\theta_4^2(0|\tau)\theta_3(2z|\tau)+\theta_4(0|\tau)\theta_3^2(0|\tau)\theta_4(2z|\tau)&=2\theta_3^2(z|\tau)\theta_4^2(z|\tau) \label{Identity2}.
\end{align}
Identities \eqref{Identity1} and \eqref{Identity2} yield
\begin{align}
\theta_3^2(0|\tau)\theta_4^4(0|\tau)\theta_3^2(2z|\tau)-\theta_4^2(0|\tau)\theta_3^4(0|\tau)\theta_4^2(2z|\tau)&=-4\theta_1^2(z|\tau)\theta_2^2(z|\tau)\theta_3^2(z|\tau)\theta_4^2(z|\tau). \label{identity3}
\end{align}
Moreover, we have 
\begin{align}
\theta_2(0|\tau)\theta_3(0|\tau)\theta_4(0|\tau)\theta_1(2z|\tau)=2\theta_1(z|\tau)\theta_2(z|\tau)\theta_3(z|\tau)\theta_4(z|\tau).\label{Identity5}
\end{align}

\noindent Furthermore,
\begin{align}
2\theta_2^2(z|2\tau)&=\theta_3(z|\tau)\theta_3(0|\tau)-\theta_4(z|\tau)\theta_4(0|\tau), \label{Identity6}\\  
2\theta_3^2(z|2\tau)&=\theta_3(z|\tau)\theta_3(0|\tau)+\theta_4(z|\tau)\theta_4(0|\tau), \label{Identity7}
\end{align}
and 
\begin{align}
2\theta_2(z|2\tau)\theta_3(z|2\tau)&=\theta_2(z|\tau)\theta_2(0|\tau), \label{Identity8}\\
2\theta_1(z|2\tau)\theta_4(z|2\tau)&=\theta_1(z|\tau)\theta_2(0|\tau). \label{Identity9}
\end{align}

\noindent The derivative of $\theta_1$ with respect to $z$ evaluated at $z=0$, satisfies
\begin{align}
\theta_1'(0|\tau)=\pi \theta_2(0|\tau)\theta_3(0|\tau)\theta_4(0|\tau). \label{IdentityThetaPrime}
\end{align}
\noindent Another identity related to theta functions evaluated at $z=0$ is 
\begin{align}\label{IdentityThetaConstants}
\theta_2(0|\tau)^4+\theta_4(0|\tau)^4=\theta_3(0|\tau)^4.
\end{align}

\noindent Finally, we require the identities
\begin{align}
\theta_4^2(0|\tau)\theta_3^2(2z|\tau)-\theta_3^2(0|\tau)\theta_4^2(2z|\tau)+\theta_2^2(0|\tau)\theta_1^2(2z|\tau)=0, \label{Identity4}
\end{align}
which can be proved using \eqref{Identity1}--\eqref{Identity5}, and
\begin{align}
\theta_1(u+v|\tau)\theta_1(u-v|\tau)\theta_3^2(0|\tau) &= \theta_4^2(u|\tau)\theta_2^2(v|\tau) - \theta_2^2(u|\tau)\theta_4^2(v|\tau), \label{uv-identities1}\\
\theta_1(u+v|\tau)\theta_1(u-v|\tau)\theta_4^2(0|\tau) &= \theta_3^2(u|\tau)\theta_2^2(v|\tau) - \theta_2^2(u|\tau)\theta_3^2(v|\tau),\label{uv-identities2}
\end{align}
which hold for all $u,v\in \CC$.

We now turn to the construction of a matrix-valued function $\Omega$ on $\CC$, which will play a central role in defining the intertwining operator. This operator will be used to obtain normal forms and to determine the $\CC$-isomorphism classes of $D_2$-automorphic Lie algebras.

 Consider $\mf{sl}(2,\CC)$ with standard basis $$h=\begin{pmatrix}
1&0\\0&-1
\end{pmatrix},\quad e=\begin{pmatrix}
0&1\\0&0
\end{pmatrix},\quad f=\begin{pmatrix}
0&0\\1&0
\end{pmatrix}.$$
To define an automorphic Lie algebra, we need to choose a representation $\rho: D_2\rightarrow \A(\mf{sl}(2,\CC))$ and a homomorphism $\sigma:D_2\rightarrow \A(T)$. Let us first discuss representations of $D_2$ on $\mf{sl}(2,\CC).$ 

Consider the Heisenberg group $\mathrm{He}_n$ for $n=2$, which is a central extension of $D_2\cong C_2\times C_2$ by $C_2$ ($C_2$ denotes the cyclic group of order 2). The relevance of the Heisenberg group will become clear below. It has order 8 and can be presented by  
\begin{align*}
\mathrm{He}_2=\langle t_1,t_2,\epsilon:  t_1^2=t_2^2=\epsilon^2=1, \,\, [t_1,t_2]=\epsilon, \,\, [t_1,\epsilon]=1,\,\, [t_2,\epsilon]=1 \rangle,
\end{align*}
where $[\cdot , \cdot]$ denotes the group commutator.
The group $\mathrm{He}_2$ is isomorphic to the dihedral group $D_4$. 

Consider the faithful, irreducible linear representation $\rho': \mathrm{He}_2\rightarrow GL(2,\CC)$ defined by 
\begin{equation}\label{def:rho'}
\rho'(t_1)=T_1:=\begin{pmatrix}
1&0\\0&-1
\end{pmatrix},\quad \rho'(t_2)=T_2:=\begin{pmatrix}
0&1\\1&0
\end{pmatrix}.
\end{equation}
Note that $\rho'(\epsilon)=\rho'([t_1,t_2])=-\mathrm{Id}$.
Each linear representation of $\mathrm{He}_2$ gives rise to a projective representation $\mathrm{He}_2/\langle \epsilon\rangle \cong D_2 \rightarrow PGL(2,\CC)$.
We shall slightly abuse notation and use the symbols $t_1,t_2$ to generate the group $D_2$ from now on. The representation $\rho'$ induces a faithful irreducible representation $\rho:D_2\rightarrow \A(\mf{sl}(2,\CC))$ via
\begin{align}\label{rhoprime}
\rho(t_1)=\Ad(T_1),\quad \rho(t_2)=\Ad(T_2),
\end{align}
where $\Ad: GL(2,\CC)\rightarrow \A(\mf{sl}(2,\CC))$ is the adjoint map $\Ad(g)X=gXg^{-1}$. Note that $PGL(2,\CC)\cong \A(\mf{sl}(2,\CC))$.
The representation $\rho$ is the unique faithful representation of $D_2$ on $\mf{sl}(2,\CC)$, up to conjugation in $\A(\mf{sl}(2,\CC))$, since isomorphic finite subgroups of $\A(\mf{sl}(2,\CC))$ are conjugate; see, for example, \cite[Lemma 4.2]{MR4808734}.
 
Having specified the action of $D_2$ on $\mf{sl}(2,\CC)$, we now consider the action of $D_2$ on a complex torus $T=\CC/(\ZZ+\ZZ\tau)$.  
Define the homomorphism $\sigma: D_2 \rightarrow \A(T)$ by
\begin{equation}\label{def:Gamma-action}
\sigma(t_1)z=z+\tfrac{1}{2}, \quad \sigma(t_2)z=z+\tfrac{\tau}{2}.
\end{equation}
Thus, $D_2$ acts on $T$ by translations, and in particular this action is fixed-point free.

We now introduce a meromorphic matrix-valued map $\Omega$ on $\CC$, which will be used later in the construction of normal forms. Let 

\begin{align}\label{def:Omega}
\Omega(z)=\begin{pmatrix}
\psi_-(z)\theta_{2}(2z|2\tau) & \theta_{3}(2z|2\tau)  \\
\psi_+(z)\theta_{3}(2z|2\tau)  & \theta_{2}(2z|2\tau)
\end{pmatrix},
\end{align}
where 
\begin{equation}\label{psi}
\psi_{\pm}(z)=\pm\frac{\theta_4^2(0|\tau)}{\theta_3(0|\tau)}\frac{\theta_3(2z|\tau)}{\theta_1(2z|\tau)}-\frac{\theta_3^2(0|\tau)}{\theta_4(0|\tau)}\frac{\theta_4(2z|\tau)}{\theta_1(2z|\tau)}.
\end{equation}
The functions $\psi_{\pm}$ are meromorphic on $T$ with simple poles precisely at $\{0,\tfrac{1}{2},\tfrac{\tau}{2}, \tfrac{1+\tau}{2}\}$, and they transform under the $D_2$-action as
\begin{align*}
\psi_{\pm}(z+\tfrac{1}{2})=-\psi_{\pm}(z),\quad \psi_-(z+\tfrac{\tau}{2})=\psi_+(z).
\end{align*}

\noindent The matrix $\Omega(z)$ satisfies the following transformation rules:
\begin{subequations}\label{eq:subeqns}
\begin{align}  
\Omega(z+\tfrac{1}{2})&=\rho'(t_1)\Omega(z),    \label{eq:Trans1} \\
\Omega(z+\tfrac{\tau}{2})&=e^{-\pi i\left(2z+\tfrac{\tau}{2}\right)}\rho'(t_2)\Omega(z),  \label{eq:Trans2} \\
\Omega(-z)&=-\Omega(z)\rho'(t_1).  \label{eq:Trans3} 
\end{align}
\end{subequations}
\begin{Lemma}\label{det}
The determinant of $\Omega(z)$ is $$\det\Omega(z)=\theta_2^2(0|\tau)\theta_1(2z|\tau).$$
\end{Lemma}
\begin{proof}
Using the theta function identities stated in this section, we compute
\begin{align*}
\det \Omega(z)&=\theta_{2}^2(2z|2\tau)\psi_-(z)- \theta_{3}^2(2z|2\tau)\psi_+(z)\\
&=\frac{1}{2\theta_1(2z|\tau)}\Bigg[(\theta_3(2z|\tau)\theta_3(0|\tau)-\theta_4(2z|\tau)\theta_4(0|\tau))\left(-\frac{\theta_4^2(0|\tau)}{\theta_3(0|\tau)}\theta_3(2z|\tau)-\frac{\theta_3^2(0|\tau)}{\theta_4(0|\tau)}\theta_4(2z|\tau)\right) - \\ & \quad (\theta_3(2z|\tau)\theta_3(0|\tau)+\theta_4(2z|\tau)\theta_4(0|\tau))\left(\frac{\theta_4^2(0|\tau)}{\theta_3(0|\tau)}\theta_3(2z|\tau)-\frac{\theta_3^2(0|\tau)}{\theta_4(0|\tau)}\theta_4(2z|\tau)\right)\Bigg]\\
&=\frac{1}{2\theta_1(2z|\tau)} \Bigg[2\theta_3^2(0|\tau)\theta_4^2(2z|\tau) -2\theta_4^2(0|\tau)\theta_3^2(2z|\tau) \Bigg]\\
&=\frac{\theta_2^2(0|\tau)\theta_1^2(2z|\tau)}{\theta_1(2z|\tau)}\\
&=\theta_2^2(0|\tau)\theta_1(2z|\tau),
\end{align*}
where we have made use of \eqref{Identity4} in the second-last step.
\end{proof}

\changes{\begin{Remark}
The matrix $\Omega(z)$ first appeared, to the best of our knowledge, in \cite{oelen2022automorphic} in a different form, and subsequently in \cite{MR4808734}, where related equivariance properties were discussed.  Interestingly, $\Omega(z)$ is closely related to a matrix $g(u)$ that appeared independently in \cite{SKRYPNYK2023116102} in the context of quantum integrable systems, more precisely in connection to the algebraic Bethe ansatz. 
The matrices $\Omega(z)$ and $g(u)^{-1}$ coincide up to renaming $z$ to $u$, multiplication by a scalar function (since $\Omega(z)$ is not elliptic, whereas $g(u)$ is), and multiplication by a constant diagonal matrix.
\end{Remark}}

Throughout the paper we assume that $D_2$ is embedded in $\A(T)$ via $\sigma$ as defined in \eqref{def:Gamma-action}, and we usually write $D_2$ instead of $\sigma(D_2)$.

The map $\Omega$ is holomorphic on $\CC\setminus \frac{1}{2}\LL$. The function $\tau\mapsto \theta_2(0|\tau)$ is a nowhere vanishing (holomorphic) function on $\mathbb{H}$, which follows from \cite[Chapter V, \S 8, Corollary 3]{chandrasekharan2012elliptic}. Together with the fact that $\theta_1(2z|\tau)=0$ precisely when $z\in \frac{1}{2}\LL$,  Lemma \ref{det} shows that, for every $\tau\in \mathbb{H}$, it is a holomorphic map $\CC\setminus \frac{1}{2}\LL\rightarrow GL(2,\CC)$ which is meromorphic at $\frac{1}{2}\LL$.

 We now consider the ring of regular functions $\mathcal{O}_{T\setminus D_2\cdot \{0\}}$ on the punctured torus $T\setminus D_2\cdot \{0\}$. We begin by studying the affine algebraic curves \eqref{eq:algcurve} for $\mathbb{K}=\CC$ and in the case it is elliptic, that is, for $n=3$. The material treated here is well-known, but we include it for completeness; see for example \cite{MR1220945, chandrasekharan2012elliptic} for related discussions. 
 \subsection{\changes{Preliminary results on elliptic curves}\label{subsec:1}}
  
 \changes{Fix $(r_1,r_2,r_3) \in \CC^3$ pairwise distinct, and consider the curve $E_{r_1,r_2,r_3}\subset \CC^3$ which is defined by the vanishing of the polynomials $\lambda_1^2-\lambda_3^2- r_3+r_1$ and $\lambda_2^2-\lambda_3^2-r_3+r_2$.} Its projective closure   
\begin{equation*}
\overline{E}_{r_1,r_2,r_3}: \begin{cases} \lambda_1^2-\lambda_3^2-(r_3-r_1)\lambda_0^2=0, \\ \lambda_3^2-\lambda_2^2-(r_2-r_3)\lambda_0^2=0,
\end{cases}
\end{equation*}
in homogeneous coordinates $[\lambda_0:\lambda_1:\lambda_2:\lambda_3]\in \mathbb{P}^3(\CC)$ defines a smooth projective curve of genus 1, and hence a compact Riemann surface. We may identify it with the complex torus $T=\CC/(\ZZ+\ZZ\tau)$ for a suitable $\tau\in \mathbb{H}$ determined by the parameters $r_1,r_2,r_3$, as we will describe below. 

Let $C_{r_1,r_2,r_3}$ be the Weierstrass curve 
\begin{equation}\label{WeierstrassForm}
C_{r_1,r_2,r_3}: y^2=(x-r_1)(x-r_2)(x-r_3)
\end{equation} in $\CC^2$ and denote by $\overline{C}_{r_1,r_2,r_3}$ the projective closure of $C_{r_1,r_2,r_3}$ given by $y^2z=(x-r_1z)(x-r_2z)(x-r_3z)$ in $\mathbb{P}^2(\CC)$.  
The curve $\overline{E}_{r_1,r_2,r_3}$  is a degree-4 unramified cover of $\overline{C}_{r_1,r_2,r_3}$. This can be seen as follows. From the defining equations of $\overline{E}_{r_1,r_2,r_3}$, we have 
$$\lambda_1^2+r_1\lambda_0^2=\lambda_2^2+r_2\lambda_0^2=\lambda_3^2+r_3\lambda_0^2,$$ so that the expression $\lambda_0(\lambda_j^2+r_j\lambda_0^2)$ is independent of $j$. Let $$\tilde{x}=\lambda_0(\lambda_j^2+r_j\lambda_0^2), \quad \tilde{y}=\lambda_1\lambda_2\lambda_3, \quad \tilde{z}=\lambda_0^3,$$ and consider the morphism
\begin{equation}\label{eq:pi}\pi: \overline{E}_{r_1,r_2,r_3}\rightarrow \overline{C}_{r_1,r_2,r_3}, \quad \pi([\lambda_0:\lambda_1:\lambda_2:\lambda_3])=[\tilde{x}:\tilde{y}:\tilde{z}],
\end{equation} cf. \cite{MR1852644}. One can verify that we indeed have $[\tilde{x}:\tilde{y}:\tilde{z}]\in \overline{C}_{r_1,r_2,r_3}$.

We let the group $D_2=\langle t_1,t_2\rangle$ act on $\overline{E}_{r_1,r_2,r_3}$ via $$t_1\cdot [\lambda_0:\lambda_1:\lambda_2:\lambda_3]= [\lambda_0:\lambda_1:-\lambda_2:-\lambda_3],\quad t_2\cdot [\lambda_0:\lambda_1:\lambda_2:\lambda_3]= [\lambda_0:-\lambda_1:-\lambda_2:\lambda_3],$$
and note that this action is fixed-point free. The map $\pi$ is constant on orbits of $D_2$, and each point of $\overline{C}_{r_1,r_2,r_3}$ has exactly four distinct pre-images. Hence $\pi$ is a degree-4 unramified cover. It follows that the quotient curve satisfies $\overline{E}_{r_1,r_2,r_3}/D_2\cong \overline{C}_{r_1,r_2,r_3}$, and the analogous statement for the corresponding affine curves holds as well. 

We shall mostly work with the affine curve $E_{r_1,r_2,r_3}$ (thus in the chart $\lambda_0\neq 0$). For later reference, we record that $D_2$ acts via $\sigma$, interpreted as automorphism of $E_{r_1,r_2,r_3}$, on a point $(\lambda_1,\lambda_2,\lambda_3)\in E_{r_1,r_2,r_3}$ as 
\begin{equation}\label{ActionC2xC2}
\sigma(t_1)(\lambda_1,\lambda_2,\lambda_3)=(\lambda_1,-\lambda_2,-\lambda_3),\quad \sigma(t_2) (\lambda_1,\lambda_2,\lambda_3)=(-\lambda_1,-\lambda_2,\lambda_3).
\end{equation}

Consider again the curve  $C_{r_1,r_2,r_3}$ defined in \eqref{WeierstrassForm}.
It is well known that for any pairwise distinct complex numbers $r_1,r_2,r_3$, the curve $C_{r_1,r_2,r_3}$ admits a uniformisation in terms of the Weierstrass-$\wp$ function. We remind the reader that this is the meromorphic function $\wp_{\LL}$ on $\CC$ which is periodic with respect to a lattice $\LL=\ZZ+\ZZ\tau$ and defined by $$\wp_{\LL}(z)=\frac{1}{z^2}+\sum_{\omega\in \LL\setminus\{0\}}\left(\frac{1}{(z-\omega)^2}-\frac{1}{\omega^2}\right).$$ It has poles of order 2 precisely at the lattice $\LL$. Due to $\LL$-periodicity, $\wp_{\LL}$ descends to a meromorphic function on $\CC/\LL$. It satisfies 
\begin{equation}\label{eq:wp_prime}
(\wp_{\LL}')^2=4\wp_{\LL}^3-g_2(\LL)\wp_{\LL}-g_3(\LL)
\end{equation}
where $\wp_{\LL}'$ denotes the derivative and $g_2(\LL),g_3(\LL)\in \CC$ are the elliptic invariants. For later reference we mention that the elliptic invariants transform as
\begin{equation}\label{eq:modularity}
g_2(\alpha\LL)=\alpha^{-4}g_2(\LL),\quad g_3(\alpha\LL)=\alpha^{-6}g_3(\LL)
\end{equation}
under scaling of the lattice by $\alpha\in \CC^*$. It follows from the definition of $\wp_{\LL}$ that $$\wp_{\alpha\LL}(z)=\alpha^{-2}\wp_{\LL}(\alpha^{-1}z)$$ for any $\alpha\in \CC^*.$
The relation between the numbers $r_i$ in the definition of $C_{r_1,r_2,r_3}$ and the lattice $\LL=\ZZ+\ZZ\tau$ is given by 
\begin{equation}\label{eq:mod_lambda}
\lambda(\tau)=\frac{r_3-r_2}{r_3-r_1},
\end{equation} where $\lambda(\tau)=\frac{\theta_2^4(0|\tau)}{\theta_3^4(0|\tau)}$ is the modular lambda function \cite[Chapter VII]{chandrasekharan2012elliptic}. 

The right-hand side of \eqref{eq:wp_prime} factors as $4(\wp_{\LL}-e_1)(\wp_{\LL}-e_2)(\wp_{\LL}-e_3)$, where $e_1=\wp_{\LL}(1/2)$, $e_2=\changes{\wp_{\LL}((1+\tau)/2)}$ and $e_3=\changes{\wp_{\LL}(\tau/2)}$. The numbers $e_1,e_2,e_3$ satisfy $$e_1+e_2+e_3=0,\quad e_1e_2e_3=\frac{g_3}{4},\quad e_2e_3+e_3e_1+e_1e_2=-\frac{g_2}{4}.$$ It follows that for distinct numbers $a_1,a_2,a_3$ such that $a_1+a_2+a_3=0$, the curve $C_{a_1,a_2,a_3}$ can be uniformised by $x=\wp_{\LL}(z)$ and $y=\frac{1}{2}\wp_{\LL}'(z)$ for a $\tau$ related to the $a_i$ via \eqref{eq:mod_lambda}. Moreover, the sets of numbers coincide: $\{a_1,a_2,a_3\}=\{e_1,e_2,e_3\}$.
 
The quantity $\lambda(\tau)=\frac{e_3-e_2}{e_3-e_1}$ determines the elliptic curve $C_{e_1,e_2,e_3}$ up to isomorphism through the corresponding value of the $j$-invariant \cite[Chapter VII, \S 8]{chandrasekharan2012elliptic}: $$j(\tau)=\frac{4}{27}\frac{(\lambda(\tau)^2-\lambda(\tau)+1)^3}{\lambda(\tau)^2(\lambda(\tau)-1)^2}.$$ 
In particular, two elliptic curves $C_{e_1,e_2,e_3}$ and $C_{e_1',e_2',e_3'}$ are isomorphic if and only if $j(\tau)=j(\tau')$, where $\tau$ and $\tau'$ are chosen so that $\lambda(\tau)$ and $\lambda(\tau')$ correspond to the triples $(e_1,e_2,e_3)$ and $(e_1',e_2',e_3')$, respectively. Two complex tori $T_i$ with modular parameters $\tau_i$, for $i=1,2$, are isomorphic if and only if $[\tau_1]=[\tau_2],$ where $[\tau]$ denotes the $SL(2,\ZZ)$-orbit of $\tau$, $$[\tau]=SL(2,\ZZ)\cdot \tau =\left\{\frac{a \tau+b}{c\tau+d}:\,a,b,c,d\in \ZZ,  \, ad-bc=1\right\}.$$
Here $A=\begin{psmallmatrix}a&b\\c&d \end{psmallmatrix}\in SL(2,\ZZ)$ acts on $\tau$ via $A\cdot \tau= \frac{a \tau+b}{c\tau+d}$, see for example \cite[Theorem 11.1.4]{husemoller2004elliptic}.

Let us now describe how to uniformise the curve $E_{r_1,r_2,r_3}$.
Consider the following meromorphic functions on $\CC$:
  \begin{equation}\label{mu}
\mu_i(z)=\frac{\theta_1'(0|\tau)}{\theta_{i+1}(0|\tau)}\frac{\theta_{i+1}(2z|\tau)}{\theta_1(2z|\tau)} , \quad i=1,2,3.
\end{equation}  
By the quasi-periodicity of the theta functions, each $\mu_i$ descends to a meromorphic function on the torus $T=\CC/\LL$, holomorphic on $T\setminus  D_2\cdot \{0\}$, since $\theta_1(2z)$ vanishes exactly on  $\frac{1}{2}\LL$. The poles are simple and for the residue at $z=0$ (the coefficient of the $z^{-1}$ term in the Laurent expansion of $\mu_i(z)$ at $z=0$, denoted by $\mathrm{res}_{z=0}\mu_i(z)$), it holds that $\mathrm{res}_{z=0}\mu_i(z)=\frac{1}{2}$ for $i=1,2,3$. It is well known (see \cite[Section 3.3]{MR4866321}) that 
 \begin{equation}\label{eq:rel-wp-mu}
 \wp_{\LL}(2z)= \mu_i(z)^2 +e_i, \quad i=1,2,3.
 \end{equation}
These identities, and the fact that the map $\pi$ \eqref{eq:pi} induces a covering $E_{r_1,r_2,r_3}\rightarrow C_{r_1,r_2,r_3}$, together with the uniformisation of $C_{r_1,r_2,r_3}$ in terms of $\wp_{\LL}$, show that the functions $\mu_i$ uniformise the curve $E_{r_1,r_2,r_3}$. Indeed, setting  $$b_i=r_i-\frac{1}{3}(r_1+r_2+r_3),\quad i=1,2,3,$$ we have $b_1+b_2+b_3=0$. Define the affine coordinate 
$$x=\wp_{\LL}(2z)+\frac{1}{3}(r_1+r_2+r_3).$$
We may identify $e_i=b_i$ for a suitable choice of $\tau$, so that \eqref{eq:rel-wp-mu} gives $$\mu_i(z)^2=\wp_{\LL}(2z)-e_i=x-r_i.$$
Recall that the map $\pi$ was defined using the coordinate $\tilde{x}=\lambda_i^2+r_i$ (in the chart $\lambda_0=1$).
Hence we may take $\lambda_i=\mu_i(z)$ for $i=1,2,3$, and choose a $\tau$ such that \eqref{eq:mod_lambda} holds. 

It follows that the $\mu_i$ define a holomorphic embedding of the punctured complex torus $T\setminus  D_2\cdot \{0\}$ into $\mathbb{C}^3$, via $$z\mapsto (\mu_1(z),\mu_2(z),\mu_3(z)).$$ 
This embedding extends to a projective map $\phi:\CC\rightarrow \mathbb{P}^3(\CC)$ given by $$\phi: z\mapsto \left[\frac{\theta_1(2z|\tau)}{\theta_1'(0|\tau)}:\frac{\theta_2(2z|\tau)}{\theta_2(0|\tau)}:\frac{\theta_3(2z|\tau)}{\theta_3(0|\tau)}:\frac{\theta_4(2z|\tau)}{\theta_4(0|\tau)}\right],$$
which is $\LL$-periodic by the quasi-periodicity of the theta functions and therefore descends to a holomorphic map $\tilde{\phi}: T\rightarrow \mathbb{P}^3(\CC)$. The image of $\tilde{\phi}$ coincides with the smooth projective curve $\overline{E}_{r_1,r_2,r_3}$ so that $\tilde{\phi}$ gives a complex analytic isomorphism $T\cong \overline{E}_{r_1,r_2,r_3}$. 

The curve $C_{r_1,r_2,r_3}$ has exactly one point at infinity, namely $[x:y:z]=[0:1:0]$ on $\overline{C}_{r_1,r_2,r_3}$. The curve $E_{r_1,r_2,r_3}$ has exactly four points at infinity, namely $[\lambda_0:\lambda_1:\lambda_2:\lambda_3]=[0:1:\pm 1:\pm 1]$ on $\overline{E}_{r_1,r_2,r_3}$. 
It follows that $E_{r_1,r_2,r_3}\cong \overline{E}_{r_1,r_2,r_3}\setminus D_2\cdot\{ \infty\}$, where $\infty=[0:1:1:1]$. 
 The points at infinity of $E_{r_1,r_2,r_3}$ correspond to $z=0,\frac{1}{2}, \frac{\tau}{2},\frac{1+\tau}{2}\in T$ in the uniformisation by the $\mu_i$ \eqref{mu}. Thus, we obtain a complex analytic isomorphism $$E_{r_1,r_2,r_3}\cong T\setminus D_2\cdot\{0\}.$$

We now discuss some functional aspects. Denote by $\CC[\lambda_1,\lambda_2,\lambda_3]$ the ring of polynomials in the $\lambda_i$. Let $I_{r_1,r_2,r_3}\subset \CC[\lambda_1,\lambda_2,\lambda_3]$ be the ideal generated by the polynomials $\lambda_i^2-\lambda_j^2- r_j+r_i$, where $i,j=1,2,3.$ The coordinate ring of $E_{r_1,r_2,r_3}$, i.e., the ring of regular functions on $E_{r_1,r_2,r_3}$,  denoted by $\CC[E_{r_1,r_2,r_3}]$, is  $$\CC[E_{r_1,r_2,r_3}]=\CC[\lambda_1,\lambda_2,\lambda_3]/I_{r_1,r_2,r_3}.$$ Since the $\mu_i$ uniformise $E_{r_1,r_2,r_3}$, analytically the ring $\CC[E_{r_1,r_2,r_3}]$ can be realised as
 $$\mathcal{O}_{T\setminus D_2\cdot\{0\}}=\CC[\mu_1,\mu_2,\mu_3].$$
 
 We will now study the action of $D_2$ on $\mathcal{O}_{T\setminus D_2\cdot\{p\}}$, the $\CC$-algebra of meromorphic functions on $T=\CC/\LL$ that are holomorphic outside $D_2\cdot \{p\}$, for some $p\in T$. Recall that $\gamma\in D_2$ acts on $f\in \mathcal{O}_{T\setminus D_2\cdot\{p\}}$ as $\gamma\cdot f = f\circ \sigma(\gamma^{-1})$, where $\sigma$ is defined in \eqref{def:Gamma-action}.
Let $\{\alpha_{00},\alpha_{01},\alpha_{10},\alpha_{11}\}$ denote the set of characters of $D_2=C_2\times C_2$, defined by $\alpha_{ij}(t_1,t_2)= \chi_i(t_1)\chi_j(t_2)$, where $\chi_0, \chi_1$ are the characters of $C_2$. (Recall that these are defined as $\chi_i(r^j)=(-1)^{ij}$ for $i,j=0,1$ and where $r$ generates $C_2$.) The next lemma describes the isotypical components $$ \mathcal{O}_{T\setminus D_2\cdot\{0\}}^{\alpha_{ij}} = \{f\in \mathcal{O}_{T\setminus D_2\cdot\{0\}}: \gamma\cdot f = \alpha_{ij}(\gamma)f, \,\,\forall \gamma\in D_2\}$$ (without loss of generality, we take $p=0$) in terms of  the functions $\mu_i$ and the Weierstrass $\wp$-function associated to the lattice $\LL$.  \changes{In other words, $\mathcal{O}_{T\setminus D_2\cdot\{0\}}^{\alpha_{ij}} $ is the (simultaneous) $\alpha_{ij}$-eigenspace for the action of $C_2\times C_2$.}

\begin{Lemma}\label{IsotypicalComponents}
The isotypical components of the action of $D_2$ on $\mathcal{O}_{T\setminus D_2\cdot\{0\}}$ are given by
\begin{align*}
 \mathcal{O}_{T\setminus D_2\cdot\{0\}}^{D_2}:=\mathcal{O}_{T\setminus D_2\cdot\{0\}}^{\alpha_{00}}&=\CC[\wp_{\frac{1}{2}\LL},\wp_{\frac{1}{2}\LL}'], &  \mathcal{O}_{T\setminus D_2\cdot\{0\}}^{\alpha_{10}}&=\CC[\wp_{\frac{1}{2}\LL}]\mu_3\oplus \CC[\wp_{\frac{1}{2}\LL}]\mu_1\mu_2,\\
 \mathcal{O}_{T\setminus D_2\cdot\{0\}}^{\alpha_{01}}&=\CC[\wp_{\frac{1}{2}\LL}]\mu_1\oplus \CC[\wp_{\frac{1}{2}\LL}]\mu_2\mu_3, &  \mathcal{O}_{T\setminus D_2\cdot\{0\}}^{\alpha_{11}}&=\CC[\wp_{\frac{1}{2}\LL}]\mu_2\oplus \CC[\wp_{\frac{1}{2}\LL}]\mu_1\mu_3.
 \end{align*}
 \end{Lemma}

\begin{proof}
Recall that $\mathcal{O}_{T\setminus D_2\cdot\{0\}}=\CC[\mu_1,\mu_2,\mu_3]$, where the $\mu_i$ are as defined in \eqref{mu}. Using \eqref{eq:1}, one verifies that $D_2$ acts on the $\mu_i$ as follows: 
\begin{align*}
\mu_1(z+\tfrac{1}{2})&=\mu_1(z), & \mu_1(z+\tfrac{\tau}{2})&=-\mu_1(z),\\
\mu_2(z+\tfrac{1}{2})&=-\mu_2(z), & \mu_2(z+\tfrac{\tau}{2})&=-\mu_2(z),\\
\mu_3(z+\tfrac{1}{2})&=-\mu_3(z), & \mu_3(z+\tfrac{\tau}{2})&=\mu_3(z).
\end{align*}
It follows that the right-hand sides of the equalities in the statement of the lemma are contained in the corresponding isotypical components. We shall prove the identities directly.  

Introduce the divisor $D=(0)+(\tfrac{1}{2})+ (\tfrac{\tau}{2})+ (\tfrac{1+\tau}{2})$ on $T=\CC/(\ZZ+\ZZ\tau)$, and consider the vector subspace of the field $\mc{M}(T)$ of meromorphic functions on $T$ which have poles that are bounded by $D$:  $$L(D)=\{f\in \mc{M}(T): \mathrm{div}(f)+D\geq 0\},$$
where  $\mathrm{div}(f)$ is the divisor of the function $f$. 
 Let us write $\wp=\wp_{\frac{1}{2}\LL}$. The functions $\mu_i^2$ satisfy $\mu_i^2=\frac{1}{4}\wp -e_i$, for $i=1,2,3$ by \eqref{eq:rel-wp-mu}, and it holds that $\mu_1\mu_2\mu_3=-\frac{1}{16}\wp'$. The latter equality can be deduced from the fact that $\wp'\in L(3D)^{\alpha_{00}}$ and using that $\wp'$ and the $\mu_i$ are odd functions, together with the Laurent expansion of $\wp'$ and $\mu_i$ about $z=0$.
 
 The claim follows from the Riemann--Roch Theorem: for any divisor $D'$ on a complex torus with $\deg(D')>0$, one has $\dim L(D')=\deg(D')$ \cite[Proposition V.3.14]{miranda1995algebraic}, where $\deg(D')$ denotes the degree of $D'$. Using that the $\mu_i$ are linearly independent over $\CC$, we now show that for the divisor $D$ defined above, $\dim L(nD)^{\alpha_{ij}}=n$ for any $n\in \ZZ_{\geq 1}$ and $i,j=0,1$. 
 
Note that $L(D)=\CC\langle1, \mu_1,\mu_2,\mu_3 \rangle$, i.e., the complex vector space spanned by $1$ and the $
\mu_i$. For any $n\geq 2$, the quotient vector space $L(nD)/L((n-1)D)$ has dimension $4n-(4(n-1))=4$ since $\deg(D)=4$. In each case, we can form products $p_{ij}$ of the $\mu_k$ such that 
$p_{ij} \in L(nD)^{\alpha_{ij}}$ and $p_{ij}\notin L((n-1)D)^{\alpha_{ij}}$, for all choices of $i,j\in \{0,1\}$. More explicitly, for $k\geq 1$ we have 
\begin{align*}
L(2kD)&=L((2k-1)D)\oplus \CC\langle \wp^k,\wp^{k-1}\mu_i\mu_j: i,j=1,2,3\rangle,\\
L((2k+1)D)&=L(2kD)\oplus \CC\langle \wp^{k-1}\wp', \wp^{k}\mu_i: i=1,2,3\rangle.
\end{align*}
From these decompositions, explicit bases for the isotypical components $L(nD)^{\alpha_{ij}}$ can be obtained directly. 
Finally, since $\mathcal{O}_{T\setminus D_2\cdot\{0\}}^{\alpha_{ij}}=\bigcup_{n\in \ZZ_{\geq 0}}L(nD)^{\alpha_{ij}}$, the claim follows.
\end{proof}
\subsection{\changes{Normal forms of automorphic Lie algebras}}

We are interested in finding normal forms of the automorphic Lie algebras 
\begin{equation}\label{def:D2-aLia}
\mf{A}(\g, \tau, S, \rho):=(\g\otimes_{\CC}\mathcal{O}_{T\setminus  D_2\cdot S})^{\rho\otimes\tilde{\sigma}(D_2)}
\end{equation} 
which are analogous to the Chevalley normal forms \changes{of} complex simple Lie algebras. Recall that throughout the current paper, we assume that $D_2$ is the group of translations over half-periods of $T=\CC/(\ZZ+\ZZ\tau)$, as in \eqref{def:Gamma-action}, so we omit $\sigma$ in the notation. We also assume that $S$ is a nonempty finite subset of $T$ and that the representation $\rho:D_2\rightarrow \A(\g)$ factors through $PGL(2,\CC)$.

Any complex simple Lie algebra $\g$ of rank $\ell$ with root system $\Phi$ has a basis, known as a Chevalley basis, given by $\{h_i, a_{\alpha}: i=1,\ldots, \ell, \text{and } \alpha\in \Phi\}$ such that the brackets are given by 
\begin{alignat*}{10}
[h_i,h_j]&=0,\\
[h_i,a_{\alpha}]&=\alpha(h_i)a_{\alpha},\\
[a_{\alpha},a_{-\alpha}]&=h_{\alpha},\\
[a_{\alpha},a_{\beta}]&=\pm(r+1)a_{\alpha+\beta}, \quad && \alpha+\beta\in \Phi,\\
[a_{\alpha},a_{\beta}]&=0, && \alpha+\beta \notin \Phi\cup \{0\},
\end{alignat*}
where $\alpha,\beta\in \Phi$, and $r$ is the greatest positive integer such that $\beta-r\alpha$ is a root, and where $h_{\alpha}$ is a $\ZZ$-linear combination of the $h_i$ (dual to $\frac{\alpha}{(\alpha,\alpha)}$), see \cite[Section 25]{humphreys1972introduction}.

Note that $\mf{A}(\g, \tau, S, \rho)$ carries a natural module structure over $\mc{O}_{T\setminus D_2\cdot S}^{D_2}$, defined by 
\begin{equation}\label{module}
g\cdot \left(\sum_iA_i\otimes f_i\right)=\sum_iA_i \otimes (gf_i),
\end{equation}
where $g\in \mc{O}_{T\setminus D_2\cdot S}^{D_2}$ and $\sum_iA_i\otimes f_i\in \mf{A}(\g, \tau, S, \rho)$. 
We will use the notation $$\mc{O}_{T\setminus D_2\cdot S}^{D_2}\langle B \rangle := \bigoplus_{i=1}^n\mc{O}_{T\setminus D_2\cdot S}^{D_2}b_i$$ to denote the free $\mc{O}_{T\setminus D_2\cdot S}^{D_2}$-module with basis $B=\{b_1,\ldots, b_n\}$.

\begin{Definition}[Normal form of $ \mf{A}(\g, \tau, S, \rho)$]\label{def:chevalley}
Let $\g$ be a finite-dimensional complex simple Lie algebra with Chevalley basis $\{h_i, a_{\alpha}: i=1,\ldots, \ell, \text{and } \alpha\in \Phi\}$. A collection of elements $\{H_i,A_{\alpha}: i=1,\ldots, \ell, \text{and }\alpha\in \Phi\}\subset \mf{A}(\g, \tau, S, \rho)$ defines a normal form of $\mf{A}(\g, \tau, S, \rho)$ if
\begin{equation}\label{def:normform}
\mf{A}(\g, \tau, S, \rho)=\mathcal{O}_{T\setminus  D_2\cdot S}^{ D_2}\left\langle \{H_i,A_{\alpha}\} \right\rangle,
\end{equation}
and the map $$\g \rightarrow \mf{A}(\g, \tau, S, \rho), \quad h_i\mapsto H_i, \,\,a_{\alpha}\mapsto A_{\alpha}$$ is a Lie algebra embedding. 
We call the right-hand side of \eqref{def:normform} a normal form of $\mf{A}(\g, \tau, S, \rho).$ 
\end{Definition}

The elements $H_i,A_{\alpha}\in \mf{A}(\g, \tau, S, \rho)$ satisfy the same Lie bracket relations as the $h_i, a_{\alpha}$. Thus a normal form of the $D_2$-automorphic Lie algebras \eqref{def:D2-aLia} is directly analogous to a Chevalley basis of $\g$. 

In general, it holds that $(\g\otimes_{\CC}\ox)^{\Gamma}\not\cong \g\otimes_{\CC}\ox^{\Gamma}$ whenever there exists some  $x\in \mathbb{X}$ such that $\Gamma_x:=\{\gamma\in \Gamma: \gamma\cdot x=x\}$ is nontrivial, assuming that $\Gamma_{x}$ acts nontrivially on $\g$, cf. \cite{MR4565687,MR4712421}. In the present setup, the group $\Gamma =D_2$ acts fixed-point freely on $T$, so $\Gamma_x=\{1\}$ for all $x\in T$.

For $\mf{A}(\g, \tau, S, \rho)$ to admit a normal form, it is necessary that $\g$ can be realised $D_2$-equivariantly inside $\g\otimes_{\CC}\mathcal{O}_{T\setminus  D_2\cdot S}$. Our strategy for constructing normal forms is to exhibit such an embedding $\varphi:\g\rightarrow \mf{A}(\g, \tau, S, \rho)$, where the matrix $\Omega(z)$ defined in \eqref{def:Omega} plays a central role.

If a normal form exists, then $\mf{A}(\g, \tau, S, \rho)$ is a free module of rank $\dim(\g)$ over the ring of invariants $\mathcal{O}_{T\setminus  D_2\cdot S}^{ D_2}$ with basis $\{H_i,A_{\alpha}\}$. Moreover, the induced map 
 $$\widehat{\varphi}:\g\otimes_{\CC}\mathcal{O}_{T\setminus  D_2\cdot S}^{D_2} \rightarrow  \mf{A}(\g, \tau, S, \rho), \quad \widehat{\varphi}(A\otimes f)=f\cdot \varphi(A)$$
is an isomorphism of Lie algebras and $\mathcal{O}_{T\setminus  D_2\cdot S}^{D_2}$-modules.

Thus, the existence of a normal form is equivalent to realising $\mf{A}(\g, \tau, S, \rho)$ as a current algebra over $\mathcal{O}_{T\setminus  D_2\cdot S}^{D_2}$ with underlying Lie algebra $\g$. Normal forms provide a powerful tool for classifying $\mf{A}(\g,\tau,S,\rho)$ up to $\CC$-Lie algebra isomorphism, as we will see in Corollary \ref{cor:isomorphism_classes}.

We aim to find normal forms of $\mathfrak{A}(\mathfrak{g}, \tau, S,\rho)$, in the sense of Definition \ref{def:chevalley},  where $\mathfrak{g}=\mathfrak{sl}(2,\mathbb{C})$ and $\rho$ is defined in \eqref{rhoprime}; that is, a basis $(H,E,F)$ over $\mathcal{O}_{T\setminus D_2\cdot S}^{ D_2}$ such that $$\mf{A}(\mf{sl}(2,\CC), \tau, S, \rho)=\mathcal{O}_{T\setminus D_2\cdot S}^{ D_2}\left\langle H,E,F\right\rangle\cong \mf{sl}(2,\CC)\otimes_{\CC}\mathcal{O}_{T\setminus D_2\cdot S}^{ D_2},$$
where the $H,E,F$ satisfy the relations $[H,E]=2E$, $[H,F]=-2F$ and $[E,F]=H$, and such that the bracket is linearly extended over $\mathcal{O}_{T\setminus D_2\cdot S}^{ D_2}$.  The techniques used to obtain this normal form can, under certain circumstances, be extended to more general base Lie algebras $\g$. A normal form of $\mf{A}(\mf{sl}(2,\CC), \tau, \{0\}, \rho)$ was obtained in \cite{MR4808734} using a different language.

 The next lemma is the first step towards obtaining normal forms. It is a slight extension of Lemmas 6.17 and 6.18 in \cite{MR4808734}, which construct a $D_2$-equivariant automorphism of $\mf{sl}(2,\CC)\otimes_{\CC}\mc{O}_{T\setminus D_2\cdot \{0\}}$. Our formulation uses a different approach: instead of the intertwiner built from square roots of elliptic functions in \cite{MR4808734}, we work with the matrix $\Omega(z)$ expressed in theta functions. 
 
 We consider the case of a (nonempty) finite union of orbits of punctures, $$S =\bigcup_{i=0}^{n-1} D_2\cdot \{p_i\} = D_2\cdot \{p_0,\ldots, p_{n-1}\}.$$ Without loss of generality, we may assume $p_0=0$ by translating the set by $-p_0$.
 
Recall from Lemma \ref{det} that $\Omega$ is a holomorphic $GL(2,\CC)$-valued map on $\CC\setminus \frac{1}{2}\LL$, meromorphic at $\frac{1}{2}\LL$. Precomposing the adjoint map $\Ad: GL(2,\CC)\rightarrow \A(\mf{sl}(2,\CC))$ with $\Omega$ yields a map $$\omega: \CC\setminus \tfrac{1}{2}\LL\rightarrow \A(\mf{sl}(2,\CC)),\quad \omega(z)=\Ad(\Omega(z)).$$ Below we show that this map descends to a holomorphic map on $T\setminus D_2\cdot\{0\}$, meromorphic at $D_2\cdot \{0\}$. We identify $\omega$ with its descended map on $T\setminus D_2\cdot\{0\}$ and write simply $\omega$. We may also view $\omega=\Ad(\Omega)$ as a map $\mf{sl}(2,\CC)\rightarrow \mf{sl}(2,\CC)\otimes_{\CC}\mc{O}_{T\setminus D_2\cdot \{0\}}$.

The current algebra $\mf{sl}(2,\CC)\otimes_{\CC}\mc{O}_{T\setminus D_2\cdot S}$ carries a natural $\mc{O}_{T\setminus D_2\cdot S}$-module structure defined by $g\cdot (A\otimes f)=A\otimes (gf)$, where $f,g\in \mc{O}_{T\setminus D_2\cdot S}$ and $A\in \mf{sl}(2,\CC)$. Define the induced endomorphism $\tilde{\omega}$ on $\mf{sl}(2,\CC)\otimes_{\CC}\mc{O}_{T\setminus D_2\cdot \{0\}}$ by $$\tilde{\omega}(A\otimes f)=f\cdot \Ad(\Omega)A.$$ Pointwise, this is given by $\tilde{\omega}(A\otimes f)(z)=f(z)\Ad(\Omega(z))A$ for $z\in T\setminus D_2\cdot\{0\}$. This map will serve as the intertwining operator discussed earlier. 

\begin{Lemma}\label{Lemma3}
Let $n\in \NN$ and set $p_0=0\in T$. Let $S=\{p_0,p_1,\ldots, p_{n-1}\}\subset T$ and $\TT=T\setminus D_2\cdot S$. Then the endomorphism $$\tilde{\omega}(A\otimes f)=f\cdot \Ad(\Omega)A$$ defines an $\ot$-linear automorphism of $\mf{sl}(2,\CC)\otimes_{\CC}\ot$.
 Moreover, the associated map $\omega(z)=\Ad(\Omega(z))$ is a $D_2$-equivariant map on $T$ in the sense that $$\omega(\sigma(t)z)=\rho(t)\omega(z),\quad t\in D_2,$$
where $z\in T\setminus D_2\cdot \{0\}$, and $\rho$ is defined in \eqref{rhoprime} and $\sigma$ in \eqref{def:Gamma-action}. Consequently, $\omega$ defines an embedding of Lie algebras  $\mf{sl}(2,\CC)\rightarrow \mf{A}(\mf{sl}(2,\CC), \tau, S, \rho)$.
\end{Lemma}

\begin{proof}
We will first prove the case of a single orbit of punctures, and then extend to the general case of $n$ orbits. 

Using that $\rho(t_i)=\Ad(T_i)$ together with the transformation rules \eqref{eq:Trans1} and \eqref{eq:Trans2}, we obtain the $D_2$-equivariance
\begin{equation}\label{PsiEquiv}
\omega(\sigma(t)z)=\rho(t)\omega(z),\quad t\in D_2.
\end{equation}
In particular, $\omega(z+1)=\omega(z)$ and  $\omega(z+\tau)=\omega(z)$ so that $\omega$ descends to a well-defined map on $T$.

The map $\tilde{\omega}$ is clearly $\mc{O}_{T\setminus D_2\cdot \{0\}}$-linear. We define an $\mc{O}_{T\setminus D_2\cdot \{0\}}$-linear endomorphism $$\eta: \mf{sl}(2,\CC)\otimes_{\CC}\mc{O}_{T\setminus D_2\cdot \{0\}}\rightarrow \mf{sl}(2,\CC)\otimes_{\CC}\mc{O}_{T\setminus D_2\cdot \{0\}}, \quad \eta(A\otimes f):=f\cdot \Ad(\Omega^{-1})A.$$ One checks that for all $A\in \mf{sl}(2,\CC)$ and $f\in \mc{O}_{T\setminus D_2\cdot \{0\}}$, we have  
$$(\tilde{\omega}\circ \eta)(A\otimes f)=\tilde{\omega}(f\cdot \Ad(\Omega^{-1})A)=f\cdot \Ad(\Omega)(\Ad(\Omega^{-1})A)=A\otimes f,$$
and similarly $\eta\circ \tilde{\omega} =\mathrm{Id}$. So $\eta=\tilde{\omega}^{-1}$ and hence $\tilde{\omega}$ is an $\mc{O}_{T\setminus D_2\cdot \{0\}}$-linear automorphism. This completes the proof in the one-orbit case.

Now let $S\subset T$ be any finite set containing $0$, and set $\TT=T\setminus D_2\cdot S$.  Let $a=\sum_{i}A_i\otimes f_i\in \mf{sl}(2,\CC)\otimes_{\CC}\ot$. For $z\in \TT$, $$\changes{\tilde{\omega}(a)(z)}=\sum_if_i(z)\Omega(z)A_i\Omega(z)^{-1}.$$ Since $0\in S$, it follows that $\tilde{\omega}(a)$ has no poles outside $D_2\cdot S$ (because conjugating with $\Omega(z)$ only introduces poles in $D_2\cdot \{0\}$). Hence $\tilde{\omega}$ defines an automorphism of $\mf{sl}(2,\CC)\otimes_{\CC}\ot$, and the equivariance property \eqref{PsiEquiv} for the associated map $\omega$ still holds.

The final claim follows from the equivariance property of $\omega$.
\end{proof}

\begin{Remark}\label{rem:intertwiner}
The operator $\tilde{\omega}$ maps the subalgebra $\mf{sl}(2,\CC)\otimes_{\CC}\ot^{\tilde{\sigma}(D_2)}$ onto the automorphic Lie algebra $\mf{A}(\mf{sl}(2,\CC),\tau, S,\rho)$. In particular, it establishes an isomorphism of Lie algebras and $\ot^{\tilde{\sigma}(D_2)}$-modules $$\mf{sl}(2,\CC)\otimes_{\CC}\ot^{\tilde{\sigma}(D_2)} \xrightarrow[]{\,\tilde{\omega}\,} \mf{A}(\mf{sl}(2,\CC),\tau, S,\rho).$$ 

Letting $E:=\tilde{\omega}(e\otimes 1)$, $F:=\tilde{\omega}(f\otimes 1)$ and $H:=\tilde{\omega}(h\otimes 1)$, we see that $$E(\sigma(t)z)=\rho(t)E(z),\quad F(\sigma(t)z)=\rho(t)F(z), \quad H(\sigma(t)z)=\rho(t)H(z),\quad t\in D_2,$$
and these satisfy the usual $\mf{sl}(2,\CC)$-relations $$[H(z),E(z)]=2E(z),\quad [H(z),F(z)]=-2F(z),\quad [E(z),F(z)]=H(z),$$
for all $z\in T\setminus D_2\cdot \{0\}$. 

Thus, conjugation by $\Omega(z)$ embeds $\mf{sl}(2,\CC)$ $D_2$-equivariantly into the current algebra $\mf{sl}(2,\CC)\otimes_{\CC}\ot$.
\end{Remark}

\begin{Remark}\label{SignsOmega}
The matrix $\Omega(z)$ is not unique among matrices whose adjoint action satisfies \eqref{PsiEquiv}. For example, for any choice of $a_i,b_i\in \{0,1\}$, we may replace $\Omega(z)$ by
$$\tilde{\Omega}(z):=T_1^{a_1}T_2^{a_2}\Omega(z) T_1^{b_1}T_2^{b_2},$$ where $T_1,T_2$ are defined in \eqref{def:rho'}. Since $T_1T_2=-T_2T_1$, the equivariance property of $\Ad(\tilde{\Omega})$ \eqref{PsiEquiv} remains unchanged. This fact will be used later in this section.
\end{Remark}

We have seen that $\Ad(\Omega): \CC\setminus \frac{1}{2}\LL\rightarrow \A(\mf{sl}(2,\CC))$ is a holomorphic map, meromorphic at $\frac{1}{2}\LL$, which descends to a holomorphic map on $T\setminus D_2\cdot\{0\}$. This map serves in the construction of an operator that intertwined the actions defined by $\sigma$ and $\rho$ as in \eqref{def:Gamma-action} and \eqref{rhoprime}, respectively. 
\subsection{\changes{Normal forms of $\mf{A}(\g,\tau, S,\rho)$}}

 We now describe a procedure, inspired by \cite[Theorem 3.3]{knibbeler2025uniform}, to generalise this construction and obtain a map $T\setminus D_2\cdot\{0\}\rightarrow \A(\g)$, where $\g$ is a complex simple Lie algebra. This procedure establishes an explicit isomorphism between $\mf{A}(\g,\tau, S,\rho)$, for suitably chosen data, and a current algebra, in a manner analogous to that described in Remark \ref{rem:intertwiner}.

Given such a $\g$, one canonically associates to it the adjoint group $G_{\mathrm{ad}}$; see, for example, \cite[Section 1.1]{collingwood1993nilpotent}.
The group $G_{\mathrm{ad}}$ may be defined as the unique connected Lie group with Lie algebra $\g$ such that the adjoint representation $\Ad:G_{\mathrm{ad}}\rightarrow \A(\g)$ is faithful. \changes{Since the kernel of $\Ad$ is the centre $Z(G_{\mathrm{ad}})$, it follows that $Z(G_{\mathrm{ad}})$ is trivial.}
The group of inner automorphisms of $\g$, denoted by $\mathrm{Inn}(\g)$, is defined to be $\Ad(G_{\mathrm{ad}})$. 
It is known that $\Ad(G_{\mathrm{ad}})=\A(\g)^0 \cong G_\mathrm{ad}$, where $\A(\g)^0$ denotes the connected component of $\A(\g)$ containing the identity.

\changes{We now consider homomorphisms $PGL(2,\CC)\rightarrow G_{\mathrm{ad}}$. Note that $G_{\mathrm{ad}}$ is a complex connected, centreless simple Lie group. The existence of an $\mf{sl}(2,\CC)$-subalgebra of $\g$ is equivalent to the existence of a homomorphism $\mf{sl}(2,\CC)\to \g$, which integrates to a homomorphism $SL(2,\CC)\rightarrow G_{\mathrm{ad}}$. 
Homomorphisms $SL(2,\CC)\to G_{\mathrm{ad}}$, up to conjugation by $G_{\mathrm{ad}}$, are classified by nilpotent orbits in $\g$, and hence by weighted Dynkin diagrams; see \cite[Section 3.5]{collingwood1993nilpotent}.} 

\changes{Such homomorphisms factor through $PGL(2,\CC)$ precisely when $-\mathrm{Id}\in SL(2,\CC)$ is mapped to the identity in $G_{\mathrm{ad}}$, which holds precisely when the associated weighted Dynkin diagram has labels in $\{0,2\}$. Precomposing with $\Ad$ yields a homomorphism $PGL(2,\CC)\rightarrow \mathrm{Inn}(\g)$. In particular, this construction applies to any simple $\g$.}

 A homomorphism $\overline{\rho}: PGL(2,\CC)\rightarrow \mathrm{Inn}(\g)$ factors uniquely as $\overline{\rho}=\Ad\circ \Psi$, where $\Psi:PGL(2,\CC)\rightarrow G_{\mathrm{ad}}$ is a homomorphism, since $\Ad:G_{\mathrm{ad}}\rightarrow \mathrm{Inn}(\g)$ is an isomorphism. We may summarise the idea of generalising $\Omega(z)$ in the following way: $$\Omega(z) \xmapsto[]{\Pi}[\Omega(z)] \xmapsto[]{\Ad\circ \Psi}\Ad(\Psi([\Omega(z)])),$$
where $\Pi:GL(2,\CC)\rightarrow PGL(2,\CC)$ is the canonical projection. Note that $z\mapsto [\Omega(z)]$ defines a map on $T\setminus D_2\cdot\{0\}$ by \eqref{eq:Trans1} and \eqref{eq:Trans2}.

The first step toward the explicit construction of a map $T\setminus D_2\cdot\{0\}\rightarrow \A(\g)$ is to factorise the matrix $\Omega(z)$ as
$$\begin{pmatrix}
\psi_-(z)\theta_{2}(2z|2\tau)  & \theta_{3}(2z|2\tau) \\
\psi_+(z)\theta_{3}(2z|2\tau) & \theta_{2}(2z|2\tau)  
\end{pmatrix}=\begin{pmatrix}
1 &\frac{\theta_3(2z|2\tau)}{\theta_2(2z|2\tau)}\\
 0& 1
\end{pmatrix}\begin{pmatrix}
1 &0\\
\frac{\psi_+(z)}{2\theta_2(0|\tau)}\frac{\theta_2(2z|\tau)}{\theta_1(2z|\tau)}& 1
\end{pmatrix}\begin{pmatrix}
 \frac{\det(\Omega(z))}{\theta_2(2z|2\tau)} &0\\
0&\theta_2(2z|2\tau)
\end{pmatrix}.
$$
This can be written as 
\begin{equation}\label{exp:factorisation}
\Omega(z)=\Delta(z)\exp\left(\frac{\theta_3(2z|2\tau)}{\theta_2(2z|2\tau)}e\right)\exp\left(\frac{\psi_+(z)}{2\theta_2(0|\tau)}\frac{\theta_2(2z|\tau)}{\theta_1(2z|\tau)}f\right)\exp\left(\log\left(\frac{\Delta(z)}{\theta_2(2z|2\tau)}\right)h\right),
\end{equation}
where $\Delta(z)=\sqrt{\det(\Omega(z))}$, $\psi_+$ is a meromorphic function on $T$ defined in \eqref{psi}, and $h,e,f$ are the standard basis elements of $\mf{sl}(2,\CC)$.

Suppose that $\Psi$ is a homomorphism $PGL(2,\CC)\rightarrow G_{\mathrm{ad}}$, and let $\mathrm{d}\Psi:\mf{pgl}(2,\CC)\rightarrow \g$ be the derivative at the identity of $PGL(2,\CC)$. The exponential intertwines $\Psi$ and $\mathrm{d}\Psi$ so that $\Psi(\exp(A))=\exp(\mathrm{d}\Psi(A))$ for all $A\in \mf{pgl}(2,\CC)\cong \mf{sl}(2,\CC)$. We let $$H=\mathrm{d}\Psi\left(\left[\begin{pmatrix}
1 &0\\
0& -1
\end{pmatrix}\right]\right), \quad E=\mathrm{d}\Psi\left(\left[\begin{pmatrix}
0 &1\\
0& 0
\end{pmatrix}\right]\right),\quad F=\mathrm{d}\Psi\left(\left[\begin{pmatrix}
0 &0\\
1& 0
\end{pmatrix}\right]\right),$$
where $[A']$ stands for the equivalence class of $A'\in \mf{gl}(2,\CC)$ in $\mf{pgl}(2,\CC)$. 
We use the factorisation of $\Omega(z)$ \eqref{exp:factorisation} to obtain
\begin{equation}\label{def:GeneralOmega}
\Psi([\Omega(z)])=\exp\left(\frac{\theta_3(2z|2\tau)}{\theta_2(2z|2\tau)}E\right)\exp\left(\frac{\psi_+(z)}{2\theta_2(0|\tau)}\frac{\theta_2(2z|\tau)}{\theta_1(2z|\tau)}F\right)\exp\left(\log\left(\frac{\Delta(z)}{\theta_2(2z|2\tau)}\right)H\right),
\end{equation}
and applying $\Ad$ yields a map $T\setminus D_2\cdot\{0\}\rightarrow \A(\g)$.

The next result generalises Lemma \ref{Lemma3} to complex reductive Lie algebras $\g$. Recall that reductiveness means  $\g=[\g,\g]\oplus Z(\g)$, where $[\g,\g]$ is the semisimple derived subalgebra and $Z(\g)$ is the centre of $\g$. For reductive $\g$ it need not hold that $\A(\g)^0=\mathrm{Inn}(\g)$. For example, if $\mf{a}$ is abelian, then $\A(\mf{a})^0=GL(\mf{a})$ and $G_{\mathrm{ad}}=\{1\}$, and thus $\mathrm{Inn}(\mf{a})=\{\mathrm{Id}\}$, see  \cite[Section 1.2]{collingwood1993nilpotent}. In our context of representations $D_2\rightarrow \mathrm{Inn}(\g)$, this implies that $D_2$ acts trivially on $Z(\g)$ (since $Z(\g)$ is abelian), and we may identify $\mathrm{Inn}(\g)=\mathrm{Inn}([\g,\g])\times \{\mathrm{Id}_{Z(\g)}\}$. 

\begin{Proposition}\label{Prop4}
Let $\g$ be a finite-dimensional complex reductive Lie algebra and suppose that $\rho:D_2\rightarrow \mathrm{Inn}(\g)$ is a  representation that factors through a representation $\overline{\rho}:PGL(2,\CC)\rightarrow \mathrm{Inn}(\g)$. Let $\sigma:D_2\rightarrow \A(T)$ be defined by $\sigma(t_1)z=z+\tfrac{1}{2}$ and $\sigma(t_2)z=z+\tfrac{\tau}{2}$.
Then there exists an $\Omega_{\overline{\rho}}\in \A(\g\otimes_{\CC}\ot)$ such that $$\Omega_{\overline{\rho}}(\sigma(t) z)=\rho(t)\Omega_{\overline{\rho}}(z)$$
for all $t\in D_2$ and $z\in \TT=T\setminus D_2\cdot S$, where $S$ is a nonempty, finite set containing $0$.
\end{Proposition}

\begin{proof}

\textbf{Step 1: $\g$ simple.}\\

Assume first that $\g$ is a finite-dimensional complex simple Lie algebra.  
By assumption,  $\rho$ factors as $\rho=\overline{\rho}\circ \delta$, where $\overline{\rho}: PGL(2,\CC)\rightarrow \mathrm{Inn}(\g)$ and $\delta: D_2\rightarrow PGL(2,\CC)$ are homomorphisms. We may assume without loss of generality that $\delta$ is injective. 

Let
$\tilde{\delta}:D_2\rightarrow PGL(2,\CC)$ be defined by  
$$\tilde{\delta}(t_1)=\left[\begin{pmatrix}
1&0\\
0&-1
\end{pmatrix}\right],\quad \tilde{\delta}(t_2)=\left[\begin{pmatrix}
0&1\\
1&0
\end{pmatrix}\right].$$ 
Recall that $\tilde{\delta}$ is the unique injective homomorphism up to conjugation in $PGL(2,\CC)$. Let $M\in PGL(2,\CC)$ be such that $\delta(t)=M\tilde{\delta}(t)M^{-1}$ for all $t\in D_2$. Then $\rho$ is conjugate to $\overline{\rho}\circ \tilde{\delta}$ in $\mathrm{Inn}(\g)$, i.e.,  $$\overline{\rho}(\tilde{\delta}(t))=\overline{\rho}(M)^{-1}\rho(t)\overline{\rho}(M).$$ 
Moreover, $\overline{\rho}=\Ad\circ \Psi$, for a unique homomorphism $\Psi: PGL(2,\CC)\rightarrow G_{\mathrm{ad}}$,  where $G_{\mathrm{ad}}$ denotes the adjoint group associated with $\g$. 

The holomorphic map $\Omega':\CC\setminus \frac{1}{2}\LL \rightarrow PGL(2,\CC)$ given by $\Omega'(z)= [\Omega(z)]$ descends to a holomorphic map on $T\setminus D_2\cdot\{0\}$. We introduce $\Omega_{\overline{\rho}}$, viewed as an $\mathrm{Inn}(\g)$-valued map on $\CC$, by $$\Omega_{\overline{\rho}}(z):=\overline{\rho}(M)\overline{\rho}(\Omega'(z)).$$
By Lemma \ref{Lemma3},
\begin{align*}
\Omega_{\overline{\rho}}(\sigma(t) z)&=\overline{\rho}(M)\overline{\rho}(\tilde{\delta}(t))\overline{\rho}(\Omega'(z))\\
&=\overline{\rho}(M)\overline{\rho}(M)^{-1}\rho(t)\overline{\rho}(M)\overline{\rho}(\Omega'(z))\\
&=\rho(t)\Omega_{\overline{\rho}}(z)
\end{align*}
for all $t\in D_2$ and $z\in T\setminus D_2\cdot\{0\}$. The explicit form of $\overline{\rho}(\Omega'(z))$ is given by the adjoint map $\Ad$ applied to \eqref{def:GeneralOmega} for the relevant $\Psi$. 
By construction, $\Omega_{\overline{\rho}}(z)$ is an inner automorphism of $\g$ for every $z$ in its domain of definition. 

Note that $\Omega_{\overline{\rho}}$ plays an analogous role to $\omega=\Ad(\Omega)$ in Lemma \ref{Lemma3}. To avoid overloading the notation, we do not introduce the separate symbol $\tilde{\Omega}_{\overline{\rho}}$ for the induced map on $\g\otimes_{\CC}\ot$, but the construction of Lemma \ref{Lemma3} should be kept in mind when interpreting  $\Omega_{\overline{\rho}}$.

We will show that $\Omega_{\overline{\rho}}$ preserves the location of the poles, i.e., that $\Omega_{\overline{\rho}}$ induces an automorphism of $\g\otimes_{\CC}\mc{O}_{T\setminus D_2\cdot\{0\}}$. From this, it will follow that it is an automorphism of $\g\otimes_{\CC}\ot$. 

For $n\in \ZZ_{\geq 0}$, let $\mathrm{Sym}^{n}(\CC^2):=(\CC^2)^{\otimes n}/S_{n}$ where $S_{n}$ is the symmetric group on $n$ symbols which acts on  $(\CC^2)^{\otimes n}$ by permutations. Let $\mathrm{Det}: GL(2,\CC)\rightarrow \CC^*$ be the determinant representation, which is given by $\mathrm{Det}(g)=\det(g)$.
 The finite-dimensional irreducible representations of $GL(2,\CC)$ are given by $$V_{\lambda}=\mathrm{Sym}^{\lambda_1-\lambda_2}(\CC^2)\otimes \mathrm{Det}^{\lambda_2}, \quad \text{where} \quad \lambda=(\lambda_1,\lambda_2)\in \ZZ^2, \,\, \lambda_1\geq \lambda_2,$$ cf. \cite[Proposition 15.47]{fulton2013representation}. Among these, those that factor through $PGL(2,\CC)$ correspond to $\lambda_1=-\lambda_2=n$, i.e., $V_{(n,-n)}$.
Concretely, $V_{\lambda}$ may be identified with the space of homogeneous polynomials in $x,y$ of degree $\lambda_1-\lambda_2$. Here,  $g\in GL(2,\CC)$ acts on a polynomial $P(x,y)$ as $g\cdot P(x,y)=P(g^{-1}x,g^{-1}y)$, where we identify $x$ with $(1,0)^T\in \CC^2$ and $y$ with $(0,1)^T\in\CC^2$. Thus the representations that factor through $PGL(2,\CC)$ correspond to spaces of even degree homogeneous polynomials. 

To show that, for fixed $x\in \g$, the map $z\mapsto \Omega_{\overline{\rho}}(z)x$ defines a $\g$-valued meromorphic map on $\CC/\LL$, we use two facts. First, by Lemma \ref{det}, $\det\Omega$ vanishes exactly on $\frac{1}{2}\LL$. Second, the matrix entries of $\Omega$ are meromorphic functions on $\CC$ whose possible poles lie in $\frac{1}{2}\LL$. Since $V_{(n,-n)}=\mathrm{Sym}^{2n}(V)\otimes \mathrm{Det}^{-n}$ and every finite-dimensional representation of $PGL(2,\CC)$ is completely reducible (hence in particular $\overline{\rho}$), it follows that $\overline{\rho}(\Omega'(z))x=\Omega_{\overline{\rho}}(z)x$ is meromorphic on $\CC$ with at most poles only in the same set, namely those in $\frac{1}{2}\LL$. By the equivariance properties we derived earlier in the proof, we have $$\Omega_{\overline{\rho}}(z+1)=\Omega_{\overline{\rho}}(z),\quad \Omega_{\overline{\rho}}(z+\tau)=\Omega_{\overline{\rho}}(z),$$ so that we in fact have meromorphic functions on $\CC/\LL$, holomorphic outside $D_2\cdot \{0\}$. Indeed, if $x\in \g$ and $X(z)=\Omega_{\overline{\rho}}(z)x$, then $X(z+1)=X(z)$ and $X(z+\tau)=X(z),$  for all $z\in T\setminus D_2\cdot \{0\}$. Thus,  $\Omega_{\overline{\rho}}$ preserves $\g\otimes_{\CC}\mathcal{O}_{T\setminus D_2\cdot \{0\}}$, and it is an automorphism of $\g\otimes_{\CC}\ot$ by the same argument as in the proof of Lemma \ref{Lemma3}. \\

\textbf{Step 2: $\g$ semisimple.}\\

 Suppose that $\g=\g_1\oplus \dots \oplus \g_m$, with each $\g_k$ simple. Since $\rho(D_2)\subset \mathrm{Inn}(\g)$, the action preserves the simple summands.  Thus, $\rho$ decomposes correspondingly as $\rho=\rho_1\oplus \dots \oplus\rho_m$, where $\rho_k:D_2\rightarrow \mathrm{Inn}(\g_k)$ is a representation which factors through $\overline{\rho}_k:PGL(2,\CC)\rightarrow \mathrm{Inn}(\g_k)$. Define $$\Omega_{\overline{\rho}}=\Omega_{\overline{\rho}_1}\oplus \dots \oplus \Omega_{\overline{\rho}_m}.$$ It is clear that we still have $\Omega_{\overline{\rho}}\in \A(\g\otimes_{\CC}\ot)$, as well as the stated equivariance properties.\\

\textbf{Step 3: $\g$ reductive.}\\

Write $\g=[\g,\g]\oplus Z(\g)$, where $Z(\g)$ is the centre of $\g$. The semisimple part decomposes as $[\g,\g]=\g_1\oplus \dots \oplus \g_m$ where all $\g_k$ are simple. Since $\rho$ preserves this decomposition of $\g$, it splits as $\rho=\rho_1\oplus \dots \oplus\rho_m\oplus \rho_{Z(\g)}$, where $\rho_{Z(\g)}:D_2\rightarrow \mathrm{Inn}(Z(\g))$ is the trivial representation since  $\mathrm{Inn}(Z(\g))=\{\mathrm{Id}_{Z(\g)}\}$, where $\mathrm{Id}_{Z(\g)}$ is the identity map on $Z(\g)$.
Define $$\Omega_{\overline{\rho}}=\Omega_{\overline{\rho}_1}\oplus \dots \oplus \Omega_{\overline{\rho}_m}\oplus \mathrm{Id}_{Z(\g)},$$
where $\mathrm{Id}_{Z(\g)}:\TT\rightarrow  \mathrm{Inn}(Z(\g))$ is the identity map on $\TT$. Then  $\Omega_{\overline{\rho}}(\sigma(t)z)=\rho(t)\Omega_{\overline{\rho}}(z)$.
Finally, it is clear that $\Omega_{\overline{\rho}}$ is an automorphism of $\g\otimes_{\CC}\ot$. This completes the proof.
\end{proof}

The ring of invariants of $\mathcal{O}_{T\setminus D_2\cdot\{0\}}$ with respect to the action of $D_2$ is given by  $\mathcal{O}_{T\setminus  D_2\cdot\{0\}}^{D_2}=\CC[\wp_{\frac{1}{2}\LL},\wp_{\frac{1}{2}\LL}'],$
cf. Lemma \ref{IsotypicalComponents}. Note that $$\CC[\wp_{\frac{1}{2}\LL},\wp_{\frac{1}{2}\LL}']\cong \CC[\wp_{\LL},\wp_{\LL}']\cong \CC[x,y]/(y^2-4x^3+g_2(\LL)x+g_3(\LL))=:R_{\LL}$$ as $\CC$-algebras (recall the differential equation satisfied by $\wp_{\LL}$ \eqref{eq:wp_prime}). A direct way of seeing this, is by noting that $p(x,y)\mapsto p(\alpha^2x,\alpha^3y)$ provides an isomorphism $R_{\LL}\rightarrow  R_{\alpha\LL}$, making use of the modularity properties of $g_2$ and $g_3$ \eqref{eq:modularity}. 

The next proposition describes the ring of invariants when we puncture $T$ at the set $D_2\cdot S$, where $S\subset T$ is a nonempty finite set. As before, we may assume without loss of generality that $p_0=0\in S$. The result states that $\ot^{D_2}$ is generated over $\CC$ by $\wp_{\frac{1}{2}\LL}$, $\wp_{\frac{1}{2}\LL}'$, and certain functions $\xi_p$ on $T$, which account for the punctures other than those in $D_2\cdot\{0\}=\{0,\tfrac{1}{2},\tfrac{\tau}{2}, \tfrac{1+\tau}{2}\}$. 
\begin{Proposition}\label{ot2} 
Let $\TT=T\setminus D_2\cdot \{p_0=0,p_1,\ldots, p_{n-1}\}$. The ring of invariants is given by 
$$\ot^{D_2}=\CC[\wp_{\frac{1}{2}\LL},\wp_{\frac{1}{2}\LL}',\xi_{p_1},\dots, \xi_{p_{n-1}}],$$
where $\xi_{p}(z)=\mu_1(z)\mu_1(z-p)$ and $\mu_1(z)=\frac{\theta_1'(0|\tau)}{\theta_{2}(0|\tau)}\frac{\theta_{2}(2z|\tau)}{\theta_1(2z|\tau)}.$ In other words, $\ot^{D_2}$ is obtained from $\mathcal{O}_{T\setminus D_2\cdot\{0\}}^{D_2}=\CC[\wp_{\frac{1}{2}\LL},\wp_{\frac{1}{2}\LL}']$ by adjoining $\xi_p$ for every $p\in S\setminus\{0\}$.
\end{Proposition}

\begin{proof}
Consider the divisor  $D=\sum_{p\in  D_2\cdot S}(p)$ on $T$, where we may assume without loss of generality that $S=\{p_0=0,p_1,\ldots, p_{n-1}\}\subset T$ is chosen so that $p_i\notin D_2\cdot\{p_j\}$ for all $i\neq j$. \changes{That is, the $p_i$ lie in distinct $D_2$-orbits.} Recall that, by the Riemann--Roch theorem, $\dim L(D)=\deg(D)$. Here, the degree of $D$ is $\deg(D)=\#(D_2\cdot S)=4n$. 
Introduce for $p\in S\setminus\{0\} $ the function $\xi_p$ on $T$ defined by
\begin{equation}\label{xip}
\xi_p(z)=\mu_1(z)\mu_1(z-p).
\end{equation}
Note that $\xi_p$ is $ D_2$-invariant by Lemma \ref{IsotypicalComponents} and hence defines a meromorphic function on $\CC/\frac{1}{2}\LL$ with order 1 poles precisely in $\{0,p\}$.
Define $\xi_0:=1$. 
Let $t_q:T\rightarrow T$ be the translation $t_q(z)=z-q$,  where $q\in T$.
Consider the set $$J:=\bigcup_{p\in S}\left\{\xi_p,\mu_1\circ t_p, \mu_2\circ t_p, \mu_3\circ t_p\right\}\subset L(D).$$ The functions in $J$ are linearly independent over $\CC$ because the $\mu_i\circ t_p$ lie in distinct isotypical components, and the pole sets $D_2\cdot \{p\}$ are disjoint for $p\neq q$. Moreover, the functions $\xi_p$ have poles only at $\{0,p\}$, ensuring linear independence. Therefore, $\dim \CC\langle J\rangle \geq 4n$, where $\CC\langle J\rangle$ is the linear space spanned over $\CC$ by the set $J$. 
Since $\CC\langle J\rangle \subset L(D)$, we have $\CC\langle J\rangle =L(D)$.

For $k\geq 1$, we have $\dim L((k+1)D)/L(kD)=4n(k+1)-4nk=4n$.  For $k=1$, a basis of the quotient vector space is $$\bigcup_{p\in S}\left\{\mu_1^2\circ t_p,(\mu_1\mu_2)\circ t_p, (\mu_1\mu_3)\circ t_p, (\mu_2\mu_3)\circ t_p\right\}+L(D).$$
It follows that
\begin{align*}
L(2D)&=L(D)\oplus \CC\big\langle \bigcup_{p\in S}\left\{\mu_1^2\circ t_p,(\mu_1\mu_2)\circ t_p, (\mu_1\mu_3)\circ t_p, (\mu_2\mu_3)\circ t_p\right\}\big\rangle,
\end{align*}
and inductively, 
\begin{equation*}
L(kD)=L((k-1)D)\oplus  \CC\big\langle \bigcup_{p\in S}\left\{\prod(\mu_{i_1}\mu_{i_2}\cdots \mu_{i_k})\circ t_p\right\}\big\rangle,
\end{equation*} 
where the product is over $(i_1,\ldots,i_k)\in \{0,1,2,3\}^k$ (this can be made more explicit using the same idea as in the proof of Lemma \ref{IsotypicalComponents}). 

Next, we compute the invariant subspaces of $L(kD)$ for all $k$.
Clearly, $L(0)^{D_2}=L(0)=\CC$, and it follows that $L(D)^{D_2}=\CC\langle \bigcup_{p\in S}\{\xi_p\} \rangle$ since $S$ has size $n$ and the $\mu_i\circ t_p$ account for $3n$ $\CC$-linearly independent functions which are not invariant.  
To obtain $L(2D)^{D_2}$ and $L(3D)^{D_2}$, we use that $\mu_i^2=\frac{1}{4}\wp_{\frac{1}{2}\LL}+c_i$ for some $c_i\in \CC$ and $\mu_1\mu_2\mu_3=-\frac{1}{16}\wp_{\frac{1}{2}\LL}'$ (see the proof of Lemma \ref{IsotypicalComponents}). We deduce 
\begin{equation}\label{LD-2} 
L(2D)^{D_2}=L(D)^{D_2}\oplus \CC\big\langle  \bigcup_{p\in S}\{\wp_{\frac{1}{2}\LL}\circ t_p\}\big\rangle, \quad L(3D)^{D_2}=L(2D)^{D_2}\oplus \CC\big\langle \bigcup_{p\in S}\{\wp_{\frac{1}{2}\LL}'\circ t_p\}\big\rangle.
\end{equation}
 Inductively, $$\bigcup_{k\in \ZZ_{\geq 0}}L(kD)^{D_2}=\bigoplus_{p\in S\setminus \{0\}}\CC\xi_{p}\oplus \CC\big[ \bigcup_{p'\in S}\{\wp_{\frac{1}{2}\LL}\circ t_{p'},\wp_{\frac{1}{2}\LL}'\circ t_{p'}\}\big].$$ 
Making use of \eqref{LD-2}, we infer $$\xi_p^2\in \CC\langle 1,\xi_p,\wp_{\frac{1}{2}\LL}\circ t_p, \wp_{\frac{1}{2}\LL}\rangle, \quad \xi_p^3\in \CC\langle 1,\xi_p,\wp_{\frac{1}{2}\LL}\circ t_p, \wp_{\frac{1}{2}\LL}'\circ t_p, \wp_{\frac{1}{2}\LL}, \wp_{\frac{1}{2}\LL}'\rangle.$$ This implies that we can write $\wp_{\frac{1}{2}\LL}\circ t_p$ and $\wp_{\frac{1}{2}\LL}'\circ t_p$ in terms of algebraic expressions of $\xi_p$, $\wp_{\frac{1}{2}\LL}$ and $\wp_{\frac{1}{2}\LL}'$, and thus $\CC[\wp_{\frac{1}{2}\LL}\circ t_p,\wp_{\frac{1}{2}\LL}'\circ t_p]\subset \CC[\wp_{\frac{1}{2}\LL}, \wp_{\frac{1}{2}\LL}', \xi_p]$, for all $p\in S$.
We may conclude that $$\ot^{D_2}=\bigcup_{k\in \ZZ_{\geq 0}}L(kD)^{D_2}=\CC[\wp_{\frac{1}{2}\LL},\wp_{\frac{1}{2}\LL}'][\xi_{p_1}]\ldots [\xi_{p_{n-1}}]=\CC[\wp_{\frac{1}{2}\LL},\wp_{\frac{1}{2}\LL}',\xi_{p_1},\dots, \xi_{p_{n-1}}].$$
Finally, we observe that if $S=\{0\}$, we recover the case of one orbit of punctures: $\mathcal{O}_{T\setminus  D_2\cdot\{0\}}^{D_2}=\CC[\wp_{\frac{1}{2}\LL},\wp_{\frac{1}{2}\LL}']$. 
\end{proof}

 The next result appears in Theorem 6 \cite{MR4808734}  for $\g=\mf{sl}(2,\CC)$, but here it is extended to arbitrary finite unions of orbits of $D_2$ and to representations $ D_2\rightarrow \mathrm{Inn}(\g)$ that factor through $PGL(2,\CC)$ for more general $\g$. We remind the reader that there is precisely one faithful homomorphism $D_2\rightarrow PGL(2,\CC)$, up to conjugation. For clarity, we use the full notation instead of the shorthand $\mf{A}(\g, \tau, S, \rho)$ in the next two results.
  
\begin{Theorem}\label{Thm2}
Let $\g$ be a finite-dimensional complex reductive Lie algebra. Let $\rho:D_2\rightarrow \mathrm{Inn}(\g)$ be a representation that factors through a representation $\overline{\rho}:PGL(2,\CC)\rightarrow \mathrm{Inn}(\g)$ and let $\sigma:D_2\rightarrow \A(T)$ be a homomorphism that embeds $D_2$ as translations of $T$. Let $S= \{p_0=0,p_1,\ldots, p_{n-1}\}$ and $\TT=T\setminus D_2\cdot S$.  Then there is an isomorphism of Lie algebras and $\ot^{\tilde{\sigma}(D_2)}$-modules: $$(\g\otimes_{\CC}\ot)^{\rho\otimes \tilde{\sigma}(D_2)}\cong \g\otimes_{\CC} \CC[\wp_{\frac{1}{2}\LL},\wp_{\frac{1}{2}\LL}',\xi_{p_1},\ldots, \xi_{p_{n-1}}].$$
When $\g$ is simple, the automorphic Lie algebra $(\g\otimes_{\CC}\ot)^{\rho\otimes \tilde{\sigma}(D_2)}$ has a normal form $$(\g\otimes_{\CC}\ot)^{\rho\otimes \tilde{\sigma}(D_2)}=\ot^{\tilde{\sigma}(D_2)}\big\langle \{\Omega_{\overline{\rho}}(h_i\otimes 1), \Omega_{\overline{\rho}}(a_{\alpha}\otimes 1)\}\big\rangle,$$
where  $\{h_i,a_{\alpha}: i=1,\ldots, \ell, \text{ and } \alpha\in \Phi\}$ is a Chevalley basis for $\g$, and $\Omega_{\overline{\rho}}$ is the intertwiner constructed in Proposition \ref{Prop4}.
\end{Theorem}

\begin{proof}
Without loss of generality, we may assume that $\sigma$ is given by \eqref{def:Gamma-action}. By Proposition \ref{Prop4}, the intertwiner $\Omega_{\overline{\rho}} \in \A(\g\otimes_{\CC}\ot)$ realises an isomorphism $\g\otimes_{\CC}\ot^{\tilde{\sigma}(D_2)}\rightarrow (\g\otimes_{\CC}\ot)^{\rho\otimes \tilde{\sigma}(D_2)}$, in a manner similar to Remark \ref{rem:intertwiner}.

Using Proposition \ref{ot2}, we therefore obtain $$(\g\otimes_{\CC}\ot)^{\rho\otimes\tilde{\sigma}(D_2)}\cong \g\otimes_{\CC}\CC[\wp_{\frac{1}{2}\LL},\wp_{\frac{1}{2}\LL}',\xi_{p_1},\ldots, \xi_{p_{n-1}}].$$
The claim about the normal form follows immediately. Via $\Omega_{\overline{\rho}}$, the Chevalley basis elements $h_i\otimes 1$ and $a_{\alpha}\otimes 1$ are mapped into $(\g\otimes_{\CC}\ot)^{\rho\otimes\tilde{\sigma}(D_2)}$, providing a free $\ot^{\tilde{\sigma}(D_2)}$-module basis and realising the desired Lie algebra isomorphism.
\end{proof}

\changes{\begin{Remark} 
The ring of regular functions on the punctured complex torus $\TT/D_2$, denoted by $\mc{O}_{\TT/D_2}$, may be identified with the ring $\mc{O}_{\TT}^{D_2} = \CC[\wp_{\frac{1}{2}\LL}, \wp_{\frac{1}{2}\LL}', \xi_{p_1}, \ldots, \xi_{p_{n-1}}]$. The isomorphism of Theorem~\ref{Thm2} thus implies that the automorphic Lie algebras $\mf{A}(\g,\tau,S,\rho)$ are current algebras of Krichever--Novikov type. The elements $\wp_{\frac{1}{2}\LL}, \wp_{\frac{1}{2}\LL}', \xi_{p_i}$ provide an explicit generating set of this function algebra, whereas the Krichever--Novikov formalism describes the same algebra via a basis that makes an almost-graded structure explicit. This structure is not manifest in the above presentation. \end{Remark}}

\begin{Remark}
For $\g$ abelian, we immediately have that $(\g\otimes_{\CC}\ot)^{D_2}=\g\otimes_{\CC}\ot^{D_2}$, since the action of $D_2$ on $\g$ is trivial. 
\end{Remark}

Theorem \ref{Thm2} implies that the $\CC$-Lie algebra isomorphism classes of $(\g\otimes_{\CC}\ot)^{\rho\otimes \tilde{\sigma}(D_2)}$ are independent of the choice of a homomorphism $\rho$. The next corollary describes how the isomorphism classes depend on the complex structure of $T$ and the set of punctures. \changes{We will use the notation $2S:=\{2p: p\in S\}$, where $S$ is a subset of a complex torus. }

\begin{Corollary}\label{cor:isomorphism_classes}
Suppose $\g$ is a complex nonabelian reductive Lie algebra.  For $i=1,2$, let $\rho_i:D_2\rightarrow \mathrm{Inn}(\g)$ and $\sigma_i: D_2\rightarrow \A(T_i)$ be as in Theorem \ref{Thm2}, where $T_i$ is a complex torus, and let $S_i\subset T_i$ be nonempty finite subsets. 
Set $\TT_i=T_i\setminus \sigma_i(D_2)S_i$.  Then $$(\g\otimes_{\CC}\mathcal{O}_{\mathbb{T}_1})^{\rho_1\otimes \tilde{\sigma}_1(D_2)}\cong (\g\otimes_{\CC}\mathcal{O}_{\mathbb{T}_2})^{\rho_2\otimes \tilde{\sigma}_2(D_2)}$$
as Lie algebras if and only if there exists an isomorphism $\phi: T_1\rightarrow T_2$ such that \changes{$\phi(2S_1)=2S_2$}. 
\end{Corollary}

\begin{proof}
We suppress the notation of $\rho_j\otimes \tilde{\sigma}_j$ and write $D_2\cdot S_j$  for the union of orbits without reference to $\sigma_j$. The proof relies on the equivalence $$\g\otimes_{\CC}\mc{A}\cong \g\otimes_{\CC}\mc{B}\iff \mc{A}\cong \mc{B},$$ where $\mc{A}$ and $\mc{B}$ are associative, commutative algebras with unit and $\g$ is a complex nonabelian reductive Lie algebra. This result follows from the fact that $\mc{A}\cong \mc{B}$ as associative algebras implies that $\g\otimes_{\CC}\mc{A}\cong \g\otimes_{\CC}\mc{B}$ as Lie algebras, together with \cite[Proposition 4.7]{MR2183958}, which establishes the other direction of the equivalence.  

We set $\mc{A}=\mathcal{O}_{\mathbb{T}_1}^{D_2}\cong \mathcal{O}_{\mathbb{T}_1/D_2}$ and $\mc{B}=\mathcal{O}_{\mathbb{T}_2}^{D_2}\cong \mathcal{O}_{\mathbb{T}_2/D_2}$. It is well known that compact Riemann surfaces correspond to smooth complex projective curves, and removing finitely many points gives affine curves. In particular, punctured compact Riemann surfaces correspond to affine algebraic curves. By \cite[Corollary 3.7]{hartshorne1977algebraic}, affine varieties $X$ and $Y$ over an algebraically closed field $k$ are isomorphic if and only if their respective coordinate rings are isomorphic as $k$-algebras. In our context, this implies that $$\mathcal{O}_{\mathbb{T}_1/D_2}\cong \mathcal{O}_{\mathbb{T}_2/D_2} \iff \mathbb{T}_1/D_2\cong \mathbb{T}_2/D_2.$$ Therefore, we have  
$$(\g\otimes_{\CC}\mathcal{O}_{\mathbb{T}_1})^{D_2}\cong (\g\otimes_{\CC}\mathcal{O}_{\mathbb{T}_2})^{D_2}\iff \g\otimes_{\CC}\mathcal{O}_{\mathbb{T}_1}^{D_2} \cong \g\otimes_{\CC}\mathcal{O}_{\mathbb{T}_2}^{D_2}\iff \mathbb{T}_1/D_2\cong \mathbb{T}_2/D_2,$$
where we have used Theorem \ref{Thm2} for the first equivalence. 
 
Recall $T=\CC/\LL$ and $T/D_2\cong \CC/\frac{1}{2}\LL\cong T$. \changes{Upon removing a $D_2$-orbit $D_2\cdot\{p\}$ from $T$, for some $p\in T$, we obtain $$(T\setminus D_2\cdot \{p\})/D_2\cong \left(\CC/\tfrac{1}{2}\LL\right)\setminus \{p'+\tfrac{1}{2}\LL\},$$ 
where $p'\in \CC$ is a lift of $p$, and where we write $p'+\tfrac{1}{2}\LL$ to denote an element of $\CC/\tfrac{1}{2}\LL$ for clarity. Under the isomorphism $\CC/\tfrac{1}{2}\LL\to \CC/\LL$, given by $z+\tfrac{1}{2}\LL\mapsto 2z+\LL$, we obtain $(T\setminus D_2\cdot \{p\})/D_2\cong T\setminus \{2p\}$.  
In the general setup, a similar argument shows that $\TT_i/D_2\cong \TT'_i,$ where $\TT'_i=T_i\setminus 2S_i$.} 
Therefore, $$\mathbb{T}_1/D_2\cong \mathbb{T}_2/D_2\iff \changes{\TT_1'\cong \TT_2'}.$$

Finally, an isomorphism \changes{$\phi':\TT_1'\rightarrow \TT_2'$} extends to an isomorphism $\phi: T_1\rightarrow T_2$, see \cite[Problem 4.8]{MR3186310} for a more general statement. In particular, \changes{$\phi$ satisfies $\phi(2S_1)=2S_2$.} 
Conversely, any isomorphism $\phi: T_1\rightarrow T_2$ such that $\phi(2S_1)=2S_2$ restricts to an isomorphism $\TT_1'\rightarrow \TT_2'$. 
\end{proof}

The isomorphism from Theorem \ref{Thm2} for one orbit of punctures $D_2\cdot\{p\}\subset T=\CC/\LL$ and any representation $\rho$, can be expressed as $$\mf{A}(\g, \tau, \{p\}, \rho)\cong\g\otimes_{\CC} \CC[x,y]/(y^2-4x^3+g_2(\LL)x+g_3(\LL)),$$
where  $g _2(\LL),g_3(\LL)$ are the elliptic invariants of $T$. Recall that two complex tori $T_i=\CC/(\ZZ+\ZZ\tau_i)$ are isomorphic if and only if their moduli lie in the same $SL(2,\ZZ)$-orbit. That is, if and only if $[\tau_1]=[\tau_2]$. If $\g$ is nonabelian and $S_i=\{p_i\}\subset  T_i$, for $i=1,2$, and $\rho_1,\rho_2$ are any representations, then 
\begin{equation}\label{eq:isom-one-orbit}
\mf{A}(\g,\tau_1,\{p_1\},\rho_1)\cong \mf{A}(\g,\tau_2,\{p_2\},\rho_2) \iff [\tau_1]=[\tau_2].
\end{equation} 
This follows because $T_1\cong T_2$ if and only if $T_1\setminus \{p_1\}\cong T_2\setminus \{p_2\}$ for any choice of $p_i\in T_i$. In particular, for non-isomorphic complex tori we get non-isomorphic automorphic Lie algebras. Note that for fixed $\tau\in \mathbb{H}$, the automorphic Lie algebras $\mf{A}(\g,\tau,S_1)$ and $\mf{A}(\g,\tau,S_2)$, for arbitrary finite subsets $S_1,S_2\subset T$ are typically not isomorphic due to the strong condition on the relation between $S_1$ and $S_2$.

Theorem \ref{Thm2} shows in particular that for $\rho$ as defined in \eqref{rhoprime}, $\mf{A}(\mf{sl}(2,\CC), \tau, \{0\}, \rho)$ is a free $\mathcal{O}_{T\setminus  D_2\cdot \{0\}}^{D_2}$-module of rank 3 with basis $H=\Ad(\Omega)h,E=\Ad(\Omega)e$ and $F=\Ad(\Omega)f$. 
We now explicitly state the basis elements $H$, $E$ and $F$, where we make use of identities \eqref{Identity6}--\eqref{Identity8} to rewrite various expressions involving theta functions. 

Denote the theta zero values $\theta_j(0|\tau)$ by $\theta_j$ and $\theta_1'(0|\tau)$ by $\theta_1'$. For any matrix $M=\begin{psmallmatrix}a&b\\c&d \end{psmallmatrix}\in GL(2,\CC)$, the matrix of $\Ad(M)$ with respect to the basis $\mc{B}=(h,e,f)$, is given by $$[\mathrm{Ad}(M)]_{\mathcal{B}}=\frac{1}{\det(M)}\begin{pmatrix}bc+ad&-ac&bd\\
-2ab&a^2&-b^2\\
2cd&-c^2&d^2\end{pmatrix}.$$
We set $M=\Omega$, and obtain:

\begin{subequations}\label{Gen2}
\begin{align}
H(z)&=\frac{1}{\theta_1'^2}\begin{pmatrix}
-\theta_2^2\mu_2(z)\mu_3(z)&\theta_4^2\mu_1(z)\mu_2(z)+\theta_3^2\mu_1(z)\mu_3(z)\\
\theta_4^2\mu_1(z)\mu_2(z)-\theta_3^2\mu_1(z)\mu_3(z) & \theta_2^2\mu_2(z)\mu_3(z)\end{pmatrix},\\
E(z)&=\frac{1}{2\theta_1'^3}\begin{pmatrix}
-\theta_2^4\mu_1(z) \changes{\left(\mu_2^2(z)+\pi^2\theta_3^4\right)}& (\theta_4^2\mu_2(z)+\theta_3^2\mu_3(z))\changes{(\theta_2^2\mu_2(z)\mu_3(z)-\theta_1'^2)}\\
(\theta_4^2\mu_2(z)-\theta_3^2\mu_3(z))\changes{(\theta_2^2\mu_2(z)\mu_3(z)+ \theta_1'^2)}&\theta_2^4\mu_1(z) \changes{\left(\mu_2^2(z)+\pi^2\theta_3^4\right)}
\end{pmatrix},\\
F(z)&=\frac{1}{2\theta_1'}\begin{pmatrix}
 \mu_1(z) &-\frac{\theta_3^2}{\theta_2^2} \mu_2(z)-\frac{\theta_4^2}{\theta_2^2} \mu_3(z)\\
\frac{\theta_3^2}{\theta_2^2} \mu_2(z)-\frac{\theta_4^2}{\theta_2^2} \mu_3(z) & - \mu_1(z)
\end{pmatrix},
\end{align}
\end{subequations}
where we remind the reader that $\mu_i(z)=\frac{\theta_1'(0|\tau)}{\theta_{i+1}(0|\tau)}\frac{\theta_{i+1}(2z|\tau)}{\theta_1(2z|\tau)}$.
Recall that $H(z),E(z),F(z)$ satisfy $$[H(z),E(z)]=2E(z),\quad [H(z),F(z)]=-2F(z),\quad [E(z),F(z)]=H(z),$$
for any $z\in T\setminus  D_2\cdot \{0\}$. Furthermore, we note that $H(-z)=H(z)$ and $E(-z)=-E(z)$, $F(-z)=-F(z)$, which follows from \eqref{eq:Trans3} or can be seen directly by inspection.

The generators $H(z), E(z), F(z)$ are given in terms of a uniformisation of $E_{r_1,r_2,r_3}$. We can write them intrinsically, in the sense that they only refer to the equations that define $E_{r_1,r_2,r_3}$, as follows.  
First, let us show that the matrix $\Omega(z)$ as defined in \eqref{def:Omega} has a more intrinsic, albeit possibly less transparent formulation. \changes{Let $$c=\left(\frac{\theta_3\theta_4}{\pi\theta_2^3}\right)^4=\frac{(r_1-r_3)(r_1-r_2)}{\pi^2(r_2-r_3)^3}$$}
\changes{and define }
\begin{equation}\label{def:intrinsic-omega}
\Omega(\lambda_2,\lambda_3)=\frac{1}{\sqrt{2(r_2-r_3)}}\begin{pmatrix} \left(-\frac{1}{A}\lambda_2-\frac{1}{B}\lambda_3\right)\sqrt{A\lambda_2-B\lambda_3} & \sqrt{A\lambda_2+B\lambda_3}\\  \left(\frac{1}{A}\lambda_2-\frac{1}{B}\lambda_3\right)\sqrt{A\lambda_2+B\lambda_3} & \sqrt{A\lambda_2-B\lambda_3}  \end{pmatrix}\changes{\begin{pmatrix}c^{1/8}& 0\\ 0&c^{-1/8}\end{pmatrix}},
\end{equation}
where $$A^2=\frac{r_3-r_1}{r_3-r_2}=\frac{\theta_3(0|\tau)^4}{\theta_2(0|\tau)^4}=\frac{1}{\lambda(\tau)},\quad B^2=\frac{r_2-r_1}{r_3-r_2}=\frac{\theta_4(0|\tau)^4}{\theta_2(0|\tau)^4}=1-\frac{1}{\lambda(\tau)}.$$ 
\changes{Here we have used the identities
\begin{subequations}\label{eq:Rij}
\begin{align}r_1-r_2&=\pi^2 \theta_4^4,\\ r_1-r_3&=\pi^2 \theta_3^4,\\ r_2-r_3&=\pi^2 \theta_2^4,\end{align}
\end{subequations} see \cite[Section 3.3]{MR4866321}, together with the correspondence between the $r_i$ and the $e_i$, as described in Section \ref{subsec:1}.}

Notice that $\Omega(\lambda_2,\lambda_3)$ is not uniquely defined \changes{because of the appearance of square and eighth roots.}  
The determinant of $\Omega(\lambda_2,\lambda_3)$ is unaffected by the sign choices \changes{of $A$ and $B$}; for any choice of signs, we have 
\begin{align*}
\det\Omega(\lambda_2,\lambda_3)&=\frac{1}{2(r_2-r_3)}\left[(A\lambda_2-B\lambda_3) \left(-\frac{1}{A}\lambda_2-\frac{1}{B}\lambda_3\right)- (A\lambda_2+B\lambda_3) \left(\frac{1}{A}\lambda_2-\frac{1}{B}\lambda_3\right)\right]\\
&= \frac{1}{r_2-r_3}(\lambda_3^2-\lambda_2^2)\\
&=1.
\end{align*}

To illustrate the effect of changes of signs $A\mapsto -A$ or $B\mapsto -B$, let $\Omega_{\pm A,\pm B}$ denote the matrix $\Omega(\lambda_2,\lambda_3)$ with the indicated choice of sign. Then
\begin{align*}
\Ad(\Omega_{-A,B})&= \Ad(T_2)\Ad(\Omega_{A,B}),\\
\Ad(\Omega_{A,-B})&= \Ad(T_2)\Ad(\Omega_{A,B})\Ad(T_1),\\
\Ad(\Omega_{-A,-B})&= \Ad(\Omega_{A,B})\Ad(T_1),
\end{align*}
where $T_1$ and $T_2$ are defined in \eqref{def:rho'}, remain $D_2$-equivariant, see Remark \ref{rem:intertwiner}. 

Recall the coordinate ring $\CC[E_{r_1,r_2,r_3}]=\CC[\lambda_1,\lambda_2,\lambda_3]/I_{r_1,r_2,r_3}$ and let $\tilde{H}(\lambda_1,\lambda_2,\lambda_3)$, $\tilde{E}(\lambda_1,\lambda_2,\lambda_3)$ and $\tilde{F}(\lambda_1,\lambda_2,\lambda_3)$ be the images of $h,e$ and $f$ under $\Ad(\Omega(\lambda_2,\lambda_3))\in \A(\mf{sl}(2,\CC)\otimes_{\CC}\CC[E_{r_1,r_2,r_3}])$, respectively. Set $R_{ij}=r_i-r_j$. Then 

\begin{subequations}\label{def:HEFtilde}
\begin{align}
\tilde{H}&=\begin{pmatrix}
-\frac{1}{\sqrt{R_{12}}\sqrt{R_{13}}}\lambda_2 \lambda_3 &\frac{1}{\sqrt{R_{23}}\sqrt{R_{13}}}\lambda_1 \lambda_2 +\frac{1}{\sqrt{R_{12}}\sqrt{R_{23}}}\lambda_1 \lambda_3 \\
\frac{1}{\sqrt{R_{23}}\sqrt{R_{13}}}\lambda_1 \lambda_2 -\frac{1}{\sqrt{R_{12}}\sqrt{R_{23}}}\lambda_1 \lambda_3  & \frac{1}{\sqrt{R_{12}}\sqrt{R_{13}}}\lambda_2 \lambda_3 \end{pmatrix},\\[1em]
\tilde{E}&=\frac{\changes{c^{1/4}}}{2}\begin{pmatrix}
-\frac{R_{23}}{R_{12}}\lambda_1\left(\frac{1}{R_{13}}\lambda_2^2+1\right)  &\left(-\frac{1}{\sqrt{R_{13}}}\lambda_2 -\frac{1}{\sqrt{R_{12}}}\lambda_3 \right)\tilde{\lambda}_-\\
 \left(\frac{1}{\sqrt{R_{13}}}\lambda_2 -\frac{1}{\sqrt{R_{12}}}\lambda_3 \right)\tilde{\lambda}_+ & \frac{R_{23}}{R_{12}}\lambda_1\left(\frac{1}{R_{13}}\lambda_2^2+1\right)
\end{pmatrix},\\[1em]
\tilde{F} &=\frac{1}{2\changes{c^{1/4}}R_{23}}\begin{pmatrix}
 \lambda_1  &- \sqrt{\frac{R_{13}}{R_{23}}} \lambda_2 - \sqrt{\frac{R_{12}}{R_{23}}} \lambda_3 \\
 \sqrt{\frac{R_{13}}{R_{23}}} \lambda_2 - \sqrt{\frac{R_{12}}{R_{23}}} \lambda_3  & - \lambda_1
\end{pmatrix},
\end{align}
\end{subequations}
where $\tilde{\lambda}_{\pm}= \sqrt{R_{23}}\pm \frac{ \sqrt{R_{23}}}{\sqrt{R_{12}}\changes{\sqrt{R_{13}}}}\lambda_2 \lambda_3$.
\changes{The matrices \eqref{Gen2} are obtained from  \eqref{def:HEFtilde} by substituting $\lambda_i=\mu_i(z)$ and using \eqref{eq:Rij}.}

Note that conjugating $h,e,f$ with $\Omega(\lambda_2,\lambda_3)$ results in matrices in which $\lambda_1$ appears, in addition to $\lambda_2$ and $\lambda_3$, due to the defining equations of $E_{r_1,r_2,r_3}$. 
One verifies that these matrices are equivariant with respect to the actions defined by $\rho$ and $\sigma$ as defined in  \eqref{rhoprime} and \eqref{ActionC2xC2}, respectively. Explicitly, letting $\tilde{X}\in \{\tilde{H},\tilde{E},\tilde{F}\}$, we have $$\tilde{X}(\sigma(t)(\lambda_1,\lambda_2,\lambda_3))=\rho(t)\tilde{X}(\lambda_1,\lambda_2,\lambda_3), \quad t\in D_2,$$
for all $(\lambda_1,\lambda_2,\lambda_3)\in E_{r_1,r_2,r_3}$. Thus, $\tilde{H}, \tilde{E}, \tilde{F}\in (\mf{sl}(2,\CC)\otimes_{\CC}\CC[E_{r_1,r_2,r_3}])^{D_2}$. 
Again, they of course satisfy $$[\tilde{H},\tilde{E}]=2\tilde{E},\quad [\tilde{H},\tilde{F}]=-2\tilde{F}, \quad [\tilde{E},\tilde{F}]=\tilde{H}.$$ 

\subsection{Real Lie algebras of invariants}

We now discuss real loci of elliptic curves and a real analogue of a previously considered automorphic Lie algebra based on $\mf{sl}(2,\CC)$. 

Let $I_{r_1,r_2,r_3}\subset \RR[\lambda_1,\lambda_2,\lambda_3]$ be the ideal generated by the polynomials $\lambda_1^2-\lambda_3^2-r_3+r_1$ and $\lambda_2^2-\lambda_3^2-r_3+r_2$, where $r_i\neq r_j$ for $i\neq j$. We focus on the fixed-point Lie subalgebras 
\begin{equation}\label{real-Lie-algebra}
(\mf{sl}(2,\RR)\otimes_{\RR}\RR[\lambda_1,\lambda_2,\lambda_3]/I_{r_1,r_2,r_3})^{D_2},
\end{equation}
where $D_2$ acts on $\mf{sl}(2,\RR)$ and on $\RR[\lambda_1,\lambda_2,\lambda_3]/I_{r_1,r_2,r_3}$ via the 
 restricted actions of \eqref{rhoprime} and \eqref{ActionC2xC2}, respectively.  
 
An elliptic curve $T=\CC/(\ZZ+\ZZ\tau)$ is said to admit a real structure if there exists an anti-holomorphic involution $s:T\rightarrow T$. We consider the standard real structure $s:T\rightarrow T$ given by $s(z)=\overline{z}$, where $\overline{z}$ is the complex conjugate of $z$. The set of real points of $T$ is then $$T_{\RR}:=\{z\in T: s(z)=z\}=\{z\in T: z=\overline{z}\}.$$
Let $p:\CC\rightarrow \CC/(\ZZ+\ZZ\tau)$ be the canonical projection, and identify \changes{$\RR$ with $\RR1\subset \CC$}. Invariance under $s$ restricts the modulus $\tau$ to exactly two possibilities:
\begin{enumerate}
\item $\tau\in \mathrm{i}\RR_{>0}$ and $T_{\RR}=p(\RR)\cup p(\tfrac{\tau}{2}+\RR)$;
\item $\tau\in \tfrac{1}{2}+\mathrm{i}\RR_{>0}$ and  $T_{\RR}=p(\RR)$.
\end{enumerate}

We focus on Case 1, the only case compatible with our framework, as we will explain below. Consider the real locus $(E_{r_1,r_2,r_3})_{\RR}$ of the real affine curve given by 
\begin{equation*}
  E_{r_1,r_2,r_3}:\begin{cases}
    \lambda_1^2-\lambda_3^2= r_3-r_1,\\
    \lambda_2^2-\lambda_3^2= r_3-r_2,
  \end{cases}
\end{equation*}
where $r_1,r_2,r_3\in \RR$ satisfy $\lambda(\tau)=\frac{r_3-r_2}{r_3-r_1}$ for some $\tau=\mathrm{i}c$, $c\in \RR_{> 0}$. The $D_2$-orbit $D_2\cdot\{0\}=\{0,\tfrac{1}{2},\tfrac{ic}{2},\tfrac{1+\mathrm{i}c}{2}\}$ is a subset of $T_{\RR}$ because $\tau=\mathrm{i}c$.
There is a real-analytic isomorphism $$T_{\RR}\setminus D_2\cdot\{0\}\cong (E_{r_1,r_2,r_3})_{\RR},$$ explicitly given by $z\mapsto (\mu_1(z),\mu_2(z),\mu_3(z))$, where the $\mu_i$ are defined in \eqref{mu} \changes{(we will show below that the $\mu_i$ are real-valued on $T_{\RR})$}. Note that $T_{\RR}$ is homeomorphic to the disjoint union of two copies of $S^1:=\{z\in \CC:|z|=1\}$. It follows that the curve $(E_{r_1,r_2,r_3})_{\RR}$ consists of four connected components (two copies of $S^1$ both punctured at two distinct points), \changes{see Fig. \ref{fig:RealLocus1} for an illustration}.

\begin{figure}[h!]
\centering
\begin{subfigure}{.3\textwidth}
  \centering
\begin{tikzpicture}[scale=0.5]

\draw[->] (-1,0) -- (6,0) node[right] {$\mathrm{Re}(z)$};
\draw[->] (0,-1) -- (0,6) node[above] {$\mathrm{Im}(z)$};

\pgftransformcm{1}{0}{0}{0.8}{\pgfpoint{0cm}{0cm}}

\draw[-,thick,blue] (0+0.09,5/2) -- (5/2-0.085,5/2) node[right] {};
\draw[-,thick,blue] (5/2+0.085,5/2) -- (5,5/2) node[right] {};

\coordinate (O) at (0,0);
\coordinate (Bone) at (0,5);
\coordinate (Btwo) at (5,0);
\coordinate (Bfour) at (5,5);

\draw[thick] (O) -- (Bone) node[left] {$\mathrm{i}c$};
\draw[-,thick,blue] (0+0.085,0) -- (5/2-0.085,0) node[right] {};
\draw[-,thick,blue] (5/2+0.085,0) -- (5,0) node[below] {$\textcolor{black}{1}$};
\draw[thick] (Btwo) -- (Bfour);
\draw[thick] (Bone) -- (Bfour);

\node[draw,blue,circle,inner sep=1pt] at (0,0) {};
\node[draw,blue,circle,inner sep=1pt] at (2.5,0) {};
\node[draw,blue,circle,inner sep=1pt] at (0,2.5) {};
\node[draw,blue,circle,inner sep=1pt] at (2.5,2.5) {};

\end{tikzpicture}

  \caption{}
  \label{fig:RealLocus1}
\end{subfigure}\hspace{2cm}
\begin{subfigure}{.35\textwidth}
  \centering
\begin{tikzpicture}[scale=0.5]

\draw[->] (-1,0) -- (6,0) node[right] {$\mathrm{Re}(z)$};
\draw[->] (0,-1) -- (0,6) node[above] {$\mathrm{Im}(z)$};

\pgftransformcm{1}{0}{1/2}{0.8}{\pgfpoint{0cm}{0cm}}

\coordinate (O) at (0,0);
\coordinate (Bone) at (0,5);
\coordinate (Btwo) at (5,0);
\coordinate (Bfour) at (5,5);

\draw[thick] (O) -- (Bone) node[left] {$\tfrac{1}{2}+\mathrm{i}c$};
\draw[-,thick,blue] (0+0.075,0) -- (5/2-0.075,0) node[right] {};
\draw[-,thick,blue] (5/2+0.075,0) -- (5,0) node[below] {$\textcolor{black}{1}$};
\draw[thick] (Btwo) -- (Bfour);
\draw[thick] (Bone) -- (Bfour);

\node[draw,blue,circle,inner sep=1pt] at (0,0) {};
\node[draw,blue,circle,inner sep=1pt] at (2.5,0) {};

\end{tikzpicture}
  \caption{}
  \label{fig:RealLocus2}
\end{subfigure}
\caption{\changes{\textbf{(a)} The real locus $T_{\RR}$ punctured at $D_2\cdot \{0\}$, for $\tau=\mathrm{i}c$ (blue). The open dots represent the set $D_2\cdot \{0\}$. \textbf{(b)}  The real locus $T_{\RR}$ punctured at $D_2\cdot \{0\}\cap T_{\RR}$, for $\tau=\tfrac{1}{2}+\mathrm{i}c$ (blue).  The open dots represent the set $D_2\cdot \{0\}\cap T_{\RR}=\{0,\frac{1}{2}\}$.}}
\end{figure}

There is another case of elliptic curves possessing a real structure -- the case of $\tau\in \frac{1}{2}+\mathrm{i}\RR_{>0}$.
Let us remark that for such a $\tau$ the real part $T_{\RR}$ of $\CC/(\ZZ+\ZZ\tau)$ is homeomorphic to $S^1$. In this case, the curve $(E_{r_1,r_2,r_3})_{\RR}$ consists of two connected components, \changes{see Fig. \ref{fig:RealLocus2}}. The set $D_2\cdot \{0\}=\{0,\tfrac{1}{2},\tfrac{1+\mathrm{i}c}{4},\tfrac{3+\mathrm{i}c}{4}\}$ is not contained in $T_{\RR}$; instead, we have $D_2\cdot \{0\}\cap T_{\RR}=\{0,\tfrac{1}{2}\}$, which consists of two points. The fact that $D_2\cdot \{0\}\not\subset T_{\RR}$ prevents us from considering \eqref{real-Lie-algebra} in the given context.

We now turn our attention to Lie-algebraic aspects assuming $\tau\in \mathrm{i}\RR_{>0}$. The elements $H(z),E(z)$ and $F(z)$ defined in \eqref{Gen2} behave as follows under $z\mapsto \overline{z}$: \begin{equation}\label{eq:conjugation}
X(\overline{z})=\overline{X(z)},\quad X=H,E,F.
\end{equation}
This follows from the identity $\theta_i(\overline{z}|\tau)=\overline{\theta_i(z|-\overline{\tau})}$, which holds for all $z\in \CC, \tau\in \mathbb{H}$, $i=1,\ldots,4$. Consequently, $\Omega(\overline{z})=\overline{\Omega(z)}$ for $\tau\in \mathrm{i}\RR_{>0}$, which establishes \eqref{eq:conjugation}. Note $\theta_j(z|\tau)$ is real-valued for $z \in \RR$ when $\tau \in \mathrm{i}\RR_{>0}$. In particular, it follows that $\mu_i(z+p)\in \RR$ for $z\in \RR$ and any $p\in D_2\cdot \{0\}$, using Lemma \ref{IsotypicalComponents}.

For explicit generators, consider $\tau=\mathrm{i}$. There is the well-known identity $$\theta_3(0|\mathrm{i})=\frac{\pi^{\frac{1}{4}}}{\Gamma(\tfrac{3}{4})},$$ where $
\Gamma(z)=\int_0^{\infty} t^{z-1}e^{-t}dt$, see, for example, \cite{MR2047619}. 
Moreover, $\theta_2(0|\mathrm{i})=\theta_4(0|\mathrm{i})=[2^{-\frac{1}{4}}]\theta_3(0|\mathrm{i})$ (which follow from modularity properties of the theta functions \cite{MR3350266} and \eqref{IdentityThetaConstants}). Substituting these expressions in \eqref{Gen2}, using also \eqref{IdentityThetaPrime}, we obtain 
\begin{align*}
H(z)&=\frac{\Gamma(\frac{3}{4})^4}{\pi^3}\begin{pmatrix}
-\sqrt{2}\mu_2(z) \mu_3(z) &\sqrt{2}\mu_1(z) \mu_2(z) +2\mu_1(z) \mu_3(z) \\
\sqrt{2}\mu_1(z) \mu_2(z) -2\mu_1(z) \mu_3(z)  & \sqrt{2}\mu_2(z) \mu_3(z) \end{pmatrix},\\[1em]
\changes{E(z)}&\changes{=\frac{\Gamma(\frac{3}{4})^5}{\sqrt{2}\pi^{\frac{17}{4}}}\begin{pmatrix}
-\mu_1(z)\left(\mu_2^2(z)+\frac{\pi^3}{\Gamma(\frac{3}{4})^4}\right)  & \left(\mu_2(z)+\sqrt{2}\mu_3(z) \right)\tilde{\mu}_{-}(z)\\
 \left(-\mu_2(z) + \sqrt{2}\mu_3(z)\right)\tilde{\mu}_+(z)& \mu_1(z)\left(\mu_2^2(z)+\frac{\pi^3}{\Gamma(\frac{3}{4})^4}\right) 
\end{pmatrix}},\\[1em]
F(z)&=\frac{\Gamma(\frac{3}{4})^3}{\sqrt{2}\pi^{\frac{7}{4}}}\begin{pmatrix}
 \mu_1(z)  &- \sqrt{2} \mu_2(z) - \mu_3(z) \\
\sqrt{2}\mu_2(z) - \mu_3(z)  & - \mu_1(z)
\end{pmatrix},
\end{align*}
where \changes{$\tilde{\mu}_{\pm}(z)=\pm \frac{\pi^3}{\sqrt{2}\Gamma(\frac{3}{4})^4}+\mu_2(z) \mu_3(z)$}. In particular, $H(z),E(z)$ and $F(z)$ are $\mf{sl}(2,\RR)$-valued for $z \in T_{\RR}\setminus D_2\cdot\{0\} \subset \CC/(\ZZ+\ZZ \mathrm{i})$. 

For the general case $\tau=\mathrm{i}c$, $c\in \RR_{>0}$, let $\mc{O}_{r_1,r_2,r_3}=\RR[\lambda_1,\lambda_2,\lambda_3]/I_{r_1,r_2,r_3}$. Recall that there is a natural action of $D_2$ on $\mf{sl}(2,\RR)\otimes_{\RR}\mc{O}_{r_1,r_2,r_3}$, defined by the homomorphisms $\rho: D_2\rightarrow \A(\mf{sl}(2,\RR))$ and $\sigma:D_2\rightarrow \A(E_{r_1,r_2,r_3})$ given by $$\rho(t_1)=\Ad\begin{pmatrix}
1&0\\0&-1
\end{pmatrix},\quad \rho(t_2)=\Ad\begin{pmatrix}
0&1\\1&0
\end{pmatrix},$$ and $$\sigma(t_1)(\lambda_1,\lambda_2,\lambda_3)=(\lambda_1,-\lambda_2,-\lambda_3),\quad \sigma(t_2) (\lambda_1,\lambda_2,\lambda_3)=(-\lambda_1,-\lambda_2,\lambda_3).$$
Classical invariant theory gives $$\RR[\lambda_1,\lambda_2,\lambda_3]^{D_2}=\RR[\lambda_1^2,\lambda_2^2,\lambda_3^2,\lambda_1\lambda_2\lambda_3].$$ 
Setting  $x=\lambda_1^2+r_1=\lambda_2^2+r_2=\lambda_3^2+r_3$ and $y=\lambda_1\lambda_2\lambda_3$, we have  $y^2=(x-r_1)(x-r_2)(x-r_3)$ and an isomorphism of rings $$\mc{O}_{r_1,r_2,r_3}^{D_2}\cong \RR[x,y]/(y^2-(x-r_1)(x-r_2)(x-r_3)).$$
Restricting $\Ad(\Omega)$ to $T_{\RR}\setminus D_2\cdot\{0\}$ yields a real-analytic map $T_{\RR}\setminus D_2\cdot\{0\}\rightarrow \A(\mf{sl}(2,\RR))$, which in the algebraic formulation is a map on $ (E_{r_1,r_2,r_3})_{\RR}$. 
By similar reasoning as before -- using the intrinsic interpretation $\Omega(\lambda_2,\lambda_3)$ in \eqref{def:intrinsic-omega} --  $\Ad(\Omega)$ defines an $\mathcal{O}_{r_1,r_2,r_3}$-linear automorphism of $\mf{sl}(2,\RR)\otimes_{\RR}\mathcal{O}_{r_1,r_2,r_3}$. 
Moreover, since $\Ad(\Omega)$ is still $D_2$-equivariant in the real setting, we obtain
\begin{equation*}
(\mf{sl}(2,\RR)\otimes_{\RR}\mc{O}_{r_1,r_2,r_3})^{D_2}= \mc{O}_{r_1,r_2,r_3}^{D_2}\left\langle \tilde{H},\tilde{E},\tilde{F} \right\rangle.
\end{equation*}
Thus we have an isomorphism $(\mf{sl}(2,\RR)\otimes_{\RR}\mc{O}_{r_1,r_2,r_3})^{D_2}\cong \mf{sl}(2,\RR)\otimes_{\RR}\mc{O}_{r_1,r_2,r_3}^{D_2}$ of real Lie algebras and $\mc{O}_{r_1,r_2,r_3}^{D_2}$-modules, and an isomorphism
$$(\mf{sl}(2,\RR)\otimes_{\RR}\mc{O}_{r_1,r_2,r_3})^{D_2}\cong \mf{sl}(2,\RR)\otimes_{\RR}\RR[x,y]/(y^2-(x-r_1)(x-r_2)(x-r_3))$$
of real Lie algebras.

The Lie algebras $(\mf{sl}(2,\RR)\otimes_{\RR}\mc{O}_{r_1,r_2,r_3})^{D_2}$ do not, strictly speaking, belong to the class of automorphic Lie algebras since we are working over $\RR$. 

We summarise what we have found for the $D_2$-invariant Lie algebras with base Lie algebra $\mf{sl}(2,\RR)$. Recall that the real locus of the real elliptic curve $E_{r_1,r_2,r_3}$ has four connected components precisely when $\tau$ is purely imaginary, with $\lambda(\tau)=\frac{r_3-r_2}{r_3-r_1}$.

\begin{Proposition}
Consider the real locus $(E_{r_1,r_2,r_3})_{\RR}$ of the affine curve 
\begin{equation*}
  E_{r_1,r_2,r_3}:\begin{cases}
    \lambda_1^2-\lambda_3^2= r_3-r_1,\\
    \lambda_2^2-\lambda_3^2= r_3-r_2,
  \end{cases}
\end{equation*}
in $\mathbb{C}^3$. 
 If $(E_{r_1,r_2,r_3})_{\RR}$ has four connected components,  then $$(\mf{sl}(2,\RR)\otimes_{\RR}\mathcal{O}_{r_1,r_2,r_3})^{D_2}=\mc{O}_{r_1,r_2,r_3}^{D_2}\left\langle \tilde{H},\tilde{E},\tilde{F} \right\rangle,$$
where $D_2$ acts on $\mf{sl}(2,\RR)$ and on $(E_{r_1,r_2,r_3})_{\RR}$ as above, and $\tilde{H},\tilde{E},\tilde{F}$ are defined in \eqref{def:HEFtilde}. In particular, there is an isomorphism $(\mf{sl}(2,\RR)\otimes_{\RR}\mc{O}_{r_1,r_2,r_3})^{D_2}\cong \mf{sl}(2,\RR)\otimes_{\RR}\mc{O}_{r_1,r_2,r_3}^{D_2}$ of real Lie algebras and $\mc{O}_{r_1,r_2,r_3}^{D_2}$-modules, and an isomorphism of real Lie algebras $$(\mf{sl}(2,\RR)\otimes_{\RR}\mathcal{O}_{r_1,r_2,r_3})^{D_2}\cong \mf{sl}(2,\RR)\otimes_{\RR}\RR[x,y]/(y^2-(x-r_1)(x-r_2)(x-r_3)).$$ 
\end{Proposition}

\section{Normal forms of Uglov's Lie algebra and Holod's Lie algebra}\label{sec:NormalForms}
In this section, we discuss two infinite-dimensional Lie algebras that have appeared in the context of integrable systems. The first is the Lie algebra $\mc{E}_{k,\nu^{\pm}}$, introduced by Uglov in \cite{MR1276451}, which has a known realisation in terms of elliptic automorphic Lie algebras. \changes{The Lie structure can be conveniently written in terms of the elliptic $r$-matrix of Sklyanin. Although we will not focus on $r$-matrices in the present paper, we mention that \cite{MR1276451} specialises a general construction of Cherednik \cite{Cherednik1983}. Different approaches exist \cite{Holod1984}, and were later unified by Skrypnyk in \cite{10.1063/1.4824152}}. The second one is the well-known hidden symmetry algebra of the Landau--Lifshitz equation introduced by Holod in \cite{holod1987hidden}. This Lie algebra appears frequently in the Adler--Kostant--Symes (AKS) scheme in the construction of integrable PDEs \cite{MR1852644, MR4855733}.

Using the fact that  $\mc{E}_{k,\nu^{\pm}}$ can be realised as an elliptic automorphic Lie algebra, and using results from Section \ref{sec:Basics}, we obtain a new normal form of this algebra. This normal form reveals that $\mc{E}_{k,\nu^{\pm}}$ is in fact an elliptic current algebra. More precisely, we show that there is a Lie algebra isomorphism $\mc{E}_{k,\nu^{\pm}}\cong \mf{sl}(2,\CC)\otimes_{\CC}R_{k,\nu^{\pm}}$, where $R_{k,\nu^{\pm}}$ is a ring of functions on an elliptic curve with modulus $k$. 
Furthermore, we show that Holod's algebra can be interpreted as an elliptic automorphic Lie algebra with symmetry group $D_2$ and with, generically, three orbits of punctures. We present a basis for Holod's algebra which reveals it to be isomorphic to an elliptic current algebra $\mf{sl}(2,\CC)\otimes_{\CC}R'$, where $R'$ is a ring of functions defined on an elliptic curve. 
To our knowledge, the identification of both Uglov's and Holod's algebra as elliptic current algebras has not previously been observed. 

We begin revisiting the $D_2$-automorphic Lie algebra based on $\mf{sl}(2,\CC)$ with precisely one orbit of punctures as defined in Section \ref{sec:Basics}, namely 
$$\mf{A}(\tau, \{0\}):=\mf{A}(\mf{sl}(2,\CC), \tau, \{0\},\rho)=(\mf{sl}(2,\CC)\otimes_{\CC}\mc{O}_{T\setminus D_2\cdot \{0\}})^{\rho\otimes \tilde{\sigma}(D_2)},$$
 where $T=\CC/\LL$, $\LL=\ZZ+\ZZ\tau$, and where $\rho$ is defined in \eqref{rhoprime}.  By \cite[Theorem 6.20]{MR4808734} (or Theorem \ref{Thm2}) we know that 
 $$\mf{A}(\tau, \{0\})\cong \mf{sl}(2,\CC)\otimes_{\CC}\CC[\wp_{\LL},\wp_{\LL}'].$$ 
 We now compute a set of three elements that generate $\mf{A}( \tau, \{0\})$ as a Lie algebra over $\CC$. That is, $$\mf{A}(\tau, \{0\})=\CC\langle X_1,X_2,X_3\rangle$$ for suitable elements $X_1,X_2,X_3\in \mf{A}(\tau, \{0\})$. This will be used repeatedly in the present section and in Section \ref{sec:LL}.

Consider the following basis elements of $\mf{sl}(2,\CC)$:
\begin{align}\label{vi}
v_1=\frac{1}{2}\begin{pmatrix}
\mathrm{i}&0\\0&-\mathrm{i}
\end{pmatrix},\quad v_2=\frac{1}{2}\begin{pmatrix}
0&1\\-1&0
\end{pmatrix},\quad v_3=\frac{1}{2}\begin{pmatrix}
0&\mathrm{i}\\\mathrm{i}&0
\end{pmatrix},
\end{align}
so that $[v_i,v_j]=\varepsilon_{ijk}v_k$, where $\varepsilon_{ijk}$ is the totally antisymmetric tensor. The isotypical components of the action \eqref{rhoprime} of $D_2$ on $\mf{sl}(2,\CC)$ are given by 
\begin{equation}\label{IsotypicalCompsl2}
\mf{sl}(2,\CC)^{\alpha_{00}}=\{0\},\quad \mf{sl}(2,\CC)^{\alpha_{10}}=\CC v_3,\quad \mf{sl}(2,\CC)^{\alpha_{01}}=\CC v_1,\quad \mf{sl}(2,\CC)^{\alpha_{11}}=\CC v_2,
\end{equation}
where the $\alpha_{ij}$ are the characters of $D_2$. In particular, the only element of $\mf{sl}(2,\CC)$ fixed by $D_2$ is $0$.

Introduce the elements
\begin{align}\label{gen:X}
X_i=v_{i}\otimes \mu_{i},\quad X_i'=v_{i}\otimes \mu_j\mu_k,
\end{align}
where $(i,j,k)$ is a cyclic permutation of $(1,2,3)$. 
These elements are invariants with respect to the action defined by $\rho\otimes \tilde{\sigma}$ where we recall that $\rho$ is defined in \eqref{rhoprime} and $\sigma$ in  \eqref{def:Gamma-action} (also, recall that $\tilde{\sigma}$ is the induced action of $D_2$ on $\ot$). In other words, 
\begin{equation}\label{invariance}
\rho(t_j)\otimes\tilde{\sigma}(t_j)(X_i)=X_i,\quad \rho(t_j)\otimes\tilde{\sigma}(t_j)(X_i')=X_i'
\end{equation} 
for $i=1,2,3$ and $j=1,2$. Notice that $[X_i,X_j]=\varepsilon_{ijk}X_k'.$

We now show that the automorphic Lie algebra $\mf{A}(\tau, \{0\})$ is a free module of rank 6 over $\CC[\wp_{\frac{1}{2}\LL}]$, with basis  $\{X_i, X_i': i=1,2,3\}$. Suppose that $V,W$ are (complex) representations of a finite abelian group $\Gamma$, with characters $\chi_i$, $i=1,\ldots, |\Gamma|$. Let $V^{\chi_i}$ denote the $\chi_i$-isotypical component of $V$.  Then the $\Gamma$-invariant subspace of $V\otimes_{\CC}W$ decomposes as $(V\otimes_{\CC}W)^{\Gamma}=\bigoplus_{i=1}^{|\Gamma|}V^{\chi_i}\otimes_{\CC}W^{\overline{\chi}_i}$, with $\overline{\chi}_i$ the dual character of $\chi_i$. Therefore, by Lemma \ref{IsotypicalComponents} and \eqref{IsotypicalCompsl2}, 
\begin{align*}
\mf{A}(\tau, \{0\})&=\bigoplus_{i,j=0}^1\mf{sl}(2,\CC)^{\alpha_{ij}}\otimes_{\CC}\mc{O}_{T\setminus D_2\cdot \{0\}}^{\overline{\alpha_{ij}}}\\
&=\bigoplus_{(i,j,k)\in \mathrm{Cycl}}\CC v_i\otimes (\CC[\wp_{\frac{1}{2}\LL}]\mu_i\oplus \CC[\wp_{\frac{1}{2}\LL}]\mu_j\mu_k)\\
&=\bigoplus_{i=1}^3\CC[\wp_{\frac{1}{2}\LL}]X_i \oplus \CC[\wp_{\frac{1}{2}\LL}]X_i',
\end{align*}
where $\mathrm{Cycl}$ denotes the set of cyclic permutations of $(1,2,3)$.
From this, it becomes clear that the elements $\wp_{\frac{1}{2}\LL}^kX_i$ and $\wp_{\frac{1}{2}\LL}^kX_i'$  ($k\geq 0$, $i=1,2,3$) form a $\CC$-basis of $\mf{A}(\tau, \{0\})$.
Equivalently,
\begin{equation}\label{dsum}
\mf{A}(\tau, \{0\})=\bigoplus_{(i,j,k)\in \mathrm{Cycl}}\CC[\wp_{\frac{1}{2}\LL}]X_i \oplus \CC[\wp_{\frac{1}{2}\LL}][X_j,X_k].
\end{equation}
We have established, via a direct computation, that $\mf{A}(\tau, \{0\})$ is a free module of rank 6 over the ring $\CC[\wp_{\frac{1}{2}\LL}]\subsetneq \mathcal{O}_{T\setminus D_2\cdot \{0\}}^{D_2}$. By \cite[Theorem 6.20]{MR4808734} (or Theorem \ref{Thm2}), this automorphic Lie algebra is in fact a free module of rank 3 over the full ring of invariants $\mathcal{O}_{T\setminus D_2\cdot \{0\}}^{D_2}=\CC[\wp_{\frac{1}{2}\LL},\wp_{\frac{1}{2}\LL}']$ -- a fact that is not immediately visible from the computation above. Instead, we had to construct an intertwining operator $\Ad(\Omega)$ to establish this fact, see Lemma \ref{Lemma3}.

We now show that it follows from \eqref{dsum} that the elements $X_1,X_2$ and $X_3$ in \eqref{gen:X} generate $\mf{A}(\tau, \{0\})$ as a Lie algebra over $\CC$. This alternative characterisation will be important when relating the present algebra to 
the Wahlquist--Estabrook algebra of the Landau--Lifshitz equation. 
\begin{Corollary}\label{cor:Cor}
As a Lie algebra, $\mf{A}(\tau, \{0\})=\CC\langle X_1,X_2,X_3\rangle$. The generators satisfy 
$$[X_i,[X_j,X_k]]=0, \quad [X_i,[X_i,X_k]]-[X_j,[X_j,X_k]]=(r_j-r_i)X_k,$$
for any cyclic permutation $(i,j,k)$ of $(1,2,3)$, where the $r_i$ are related to $\lambda(\tau)$ as in \eqref{eq:mod_lambda}.
\end{Corollary}
\begin{proof}
Let $\mf{A}=\CC\langle X_1,X_2,X_3\rangle$, the Lie algebra generated by $X_1,X_2,X_3$ over $\CC$, and write $\wp=\wp_{\frac{1}{2}\LL}$. Recall that $\frac{1}{4}\wp=\mu_i^2+e_i$, $i=1,2,3$, and that $[X_i,X_j]=\varepsilon_{ijk}X_k'$, where $X_k'$ is defined in \eqref{gen:X}. 

Since $$[X_i,X_j']=\varepsilon_{ijk}\mu_i^2X_k=\varepsilon_{ijk}\frac{1}{4}(\wp-e_i)X_k,$$ it follows that $\wp X_i\in \mf{A}.$ (For brevity, we drop the $\cdot$ in the $\ot^{D_2}$-module notation as in \eqref{module}.)
Similarly, $\wp^mX_i\in \mf{A}$ for all $m\in \ZZ_{\geq 2}$ and $i=1,2,3$. 
 By the same argument, $\wp^mX_i'\in \mf{A}$ for any $m\in \ZZ_{\geq 0}$ and $i=1,2,3.$ The first claim then follows from \eqref{dsum}. 

The first bracket relation holds because $[X_i,X_i']=0$ for $i=1,2,3$.
For the second relation, we compute $[X_i,[X_i,X_k]]=[X_i,X_j']= \mu_{i}^2X_k,$ and hence $$[X_i,[X_i,X_k]]-[X_j,[X_j,X_k]]=(\mu_i^2-\mu_j^2)X_k=(r_j-r_i)X_k.$$
\end{proof}

We will need the following properties of the automorphic Lie algebras $\mf{A}(\mf{sl}(2,\CC), \tau, S)$. 
\begin{Lemma}\label{lem:decomp}
Let $S,S'\subset T=\CC/(\ZZ+\ZZ\tau)$ be finite sets \changes{such that no two distinct elements lie in the same $D_2$-orbit}.
\begin{enumerate}
\item[i)]  If $S\subset S'$, then $\mf{A}(\mf{sl}(2,\CC), \tau, S)\subset \mf{A}(\mf{sl}(2,\CC), \tau, S')$.   
\item[ii)] If $S\cap S'=\emptyset$, then $\mf{A}(\mf{sl}(2,\CC), \tau, S)\cap \mf{A}(\mf{sl}(2,\CC), \tau, S')=\{0\}$. 
\item[iii)] $\mf{A}(\mf{sl}(2,\CC), \tau,S)=\bigoplus_{p\in S}\mf{A}(\mf{sl}(2,\CC), \tau, \{p\})$
as a direct sum of Lie subalgebras. 
\end{enumerate}
\end{Lemma}
\begin{proof}
{\it i)} This follows from the definition of $\mf{A}(\mf{sl}(2,\CC), \tau, S)$.

 {\it ii)} Recall that $\mf{sl}(2,\CC)^{D_2}=\mf{sl}(2,\CC)^{\alpha_{00}}=\{0\}$, see \eqref{IsotypicalCompsl2}. In particular, any nonzero element $a\in \mf{A}(\mf{sl}(2,\CC), \tau,S)$ cannot be constant and must therefore have poles contained in $S$. Hence, if $S\cap S'=\emptyset$, $a$ cannot belong to $\mf{A}(\mf{sl}(2,\CC), \tau,S')$.

 {\it iii)} We note that by part {\it i)}, we have $$\bigoplus_{p\in S}\mf{A}(\mf{sl}(2,\CC), \tau, \{p\})\subset \mf{A}(\mf{sl}(2,\CC), \tau,S).$$ 
The reverse inclusion follows from the relations 
\begin{equation*}
\mu_i(z)\mu_j(z-p)=c_1 \mu_{k}(z)+c_2 \mu_{k}(z-p), \quad p\notin D_2\cdot \{0\},
\end{equation*}
for certain $c_1,c_2\in \CC^*$ depending on $i,j$, and where $(i,j,k)$ is a permutation of $(1,2,3)$. \changes{By translation invariance, it suffices to establish these relations in this form. They can be} proved by taking the divisor $$D:=\sum_{\gamma\in D_2}(\gamma\cdot 0)+\sum_{\gamma\in D_2}(\gamma\cdot p)$$ on $T$. Since $\deg(D)=8$, it follows from Riemann--Roch that $\dim L(D)=8$. The space $L(D)$ is spanned by the eight functions $1, \mu_1(\mu_1\circ t_p), \mu_i, \mu_i\circ t_p$, for $i=1,2,3$, where we recall that $t_p(z)=z-p$. Now, for $p\neq 0$, the function $\mu_i(\mu_j\circ t_p)$ belongs to $L(D)$, and by considering how it transforms under $D_2$ and by comparing its poles, we see that it must be a $\CC$-linear combination of $\mu_k$ and $\mu_k\circ t_p$. 

Hence, any element $a\in \mf{A}(\mf{sl}(2,\CC), \tau, S)$ can be decomposed as $a=\sum_{p\in S} a^{(p)}$, where $a^{(p)}\in \mf{A}(\mf{sl}(2,\CC), \tau, \{p\})$.
\end{proof}

\subsection{\changes{Uglov's Lie algebra}}

We now focus on the automorphic Lie algebra $(\mf{sl}(2,\CC)\otimes_{\CC}\mc{O}_{\TT})^{\rho\otimes\tilde{\sigma}(D_2)}$, where $\TT$ is the torus $T=\CC/\LL=\CC/(\ZZ+\ZZ\tau)$ punctured at two distinct $D_2$-orbits of the points $\nu^+$ and $\nu^-$ in $T$. This algebra appears in the context of quantisation and elliptic $r$-matrices \cite{MR1229896,MR1276451}.

 Let $\nu^+,\nu^-\in T$ satisfy $\nu^+-\nu^-\notin D_2\cdot \{0\}$, \changes{i.e., they lie in distinct $D_2$-orbits}.  Uglov \cite{MR1276451} defines a complex Lie algebra $\mc{E}_{k,\nu^{\pm}}$ generated by six elements $\{x_i^{\pm}\}_{i=1,2,3}$, subject to the relations
\begin{align*}
[x_i^{\pm},[x_j^{\pm},x_k^{\pm}]]&=0,  \\ 
[x_i^{\pm},[x_i^{\pm},x_k^{\pm}]]-[x_j^{\pm},[x_j^{\pm},x_k^{\pm}]]&=J_{ij}x_k^{\pm},\\
[x_i^{+}, x_i^{-}] &=0,\\
[x_i^{\pm},x_j^{\mp}]&=\sqrt{-1}(w_i(\nu^{\mp}-\nu^{\pm})x_k^{\mp}-w_j(\nu^{\mp}-\nu^{\pm})x_k^{\pm}),
\end{align*}
where $(i,j,k)$ is any cyclic permutation of $(1,2,3)$, and 
\begin{equation*}
w_1(z)=\frac{1}{\mathrm{sn}(z)},\quad w_2(z)=\frac{\mathrm{dn}(z)}{\mathrm{sn}(z)}, \quad w_3(z)=\frac{\mathrm{cn}(z)}{\mathrm{sn}(z)},
\end{equation*} 
are the Jacobi elliptic functions $\mathrm{sn}$, $\mathrm{cn}$ and $\mathrm{dn} $ of modulus $k=\frac{\theta_2^2(0|\tau)}{\theta_3^2(0|\tau)}$ \cite[Chapter VII, \S 1]{chandrasekharan2012elliptic}. These functions satisfy $w_i(z)^2-w_j(z)^2=J_{ij}$, with $J_{12}=k^2$, $J_{23}=1-k^2$ and $J_{31}=-1$. \changes{They are related to the functions $\mu_i$ by a common multiplicative factor, reparametrisation, and a relabelling of indices ($1\leftrightarrow 3$). By \eqref{eq:Rij} we have $k^2=\frac{r_2-r_3}{r_1-r_3}$, and hence \begin{equation}\label{eq:Jij}
r_2-r_3=(r_1-r_3)J_{12}, \quad r_1-r_2=(r_1-r_3)J_{23}, \quad r_3-r_1=(r_1-r_3)J_{31}.\end{equation}} 
\changes{The relations defining $\mc{E}_{k,\nu^{\pm}}$ can be succinctly written in terms of Sklyanin's elliptic $r$-matrix \cite{MR1276451}. }
 
The algebra $\mc{E}_{k,\nu^{\pm}}$ admits an elliptic realisation $\tilde{\mc{E}}_{k,\nu^{\pm}}$, which coincides with the $\mf{sl}(2,\CC)$-automorphic Lie algebra on the punctured complex torus $\TT=T\setminus  D_2\cdot  \{\nu^+,\nu^-\}$. The $ D_2$-action is defined in \eqref{def:Gamma-action}. That is, $$\tilde{\mc{E}}_{k,\nu^{\pm}}=(\mf{sl}(2,\CC)\otimes_{\CC}\mc{O}_{\TT})^{ D_2}\cong \mc{E}_{k,\nu^{\pm}}\changes{,}$$
\changes{see \cite[Section 3]{MR1276451}} 

The algebra $\mc{E}_{k,\nu^{\pm}}$ decomposes as a \changes{vector space} direct sum  $$\mc{E}_{k,\nu^{\pm}}=\mc{E}^{+}\oplus\mc{E}^{-},$$ where $\mc{E}^{+}$ and $\mc{E}^-$ are \changes{Lie subalgebras} generated by $\{x_i^{+}\}_{i=1,2,3}$ and $\{x_i^{-}\}_{i=1,2,3}$, respectively \cite[Theorem]{MR1276451}. 

We will need the following result, which will also play an important role in Section \ref{sec:LL}. Recall the Lie algebra $\mf{sl}(2,\CC)\otimes_{\CC} \CC[E_{r_1,r_2,r_3}]$, where $\CC[E_{r_1,r_2,r_3}]$ denotes the coordinate ring of the elliptic curve $E_{r_1,r_2,r_3}$. Also, recall the basis $v_1,v_2,v_3$ of $\mf{sl}(2,\CC)$ as defined in \eqref{vi}.

\begin{Proposition}\label{prop:isom_pi}
The Lie subalgebra of $\mf{sl}(2,\CC)\otimes_{\CC} \CC[E_{r_1,r_2,r_3}]$ generated by $$v_1\otimes \lambda_1, \quad v_2\otimes \lambda_2, \quad v_3\otimes \lambda_3,$$
is isomorphic to the complex Lie algebra generated by $p_1,p_2,p_3$ with relations 
$$
[p_i,[p_j,p_k]]=0,\quad 
[p_i,[p_i,p_k]]-[p_j,[p_j,p_k]]=(r_j-r_i)p_k, 
$$
where $(i,j,k)$ is a cyclic permutation of $(1,2,3)$. In particular, we have 
$$\CC\langle p_1,p_2,p_3\rangle \cong (\mf{sl}(2,\CC)\otimes_{\CC}\mc{O}_{T\setminus D_2\cdot \{0\}})^{ D_2}\cong \mf{sl}(2,\CC)\otimes_{\CC}\CC[\wp_{\LL},\wp_{\LL}'].$$
\end{Proposition}
\begin{proof}
The first isomorphism follows from  \cite[Theorem 3.2]{MR1218995}. The second is obtained by identifying $v_i\otimes \lambda_i$ with $X_i=v_i\otimes \mu_i$, where $\mu_i$ uniformises $E_{r_1,r_2,r_3}$, and using Corollary \ref{cor:Cor}.
\end{proof}

\begin{Remark}\label{rem:R}
The Lie algebra generated by the $v_i\otimes \lambda_i$ is identical to the Lie algebra $\mf{R}_{r_1,r_2,r_3}$ as defined in \cite{MR4056196}, and denoted by $R$ in \cite{MR1218995, MR1947240}. This Lie algebra appears  in connection to the prolongation algebras of the fully anisotropic Landau--Lifshitz equation and the non-singular Krichever--Novikov equation, as we will explain in more detail in Section \ref{sec:LL}. It follows from Proposition \ref{prop:isom_pi} that $\mf{R}_{r_1,r_2,r_3}$ is nothing but $\mf{sl}(2,\CC)$ with the scalars replaced by functions in $\mc{O}_{T\setminus D_2\cdot \{0\}}^{ D_2}\cong \CC[\wp_{\LL},\wp_{\LL}']$.
\end{Remark}

It follows from Proposition \ref{prop:isom_pi}, \changes{using \eqref{eq:Jij} together with the index swap $1\leftrightarrow 3$ and a rescaling of the generators,} that $\CC\langle X_1,X_2,X_3\rangle\cong \mc{E}^{\pm}$ as Lie algebras. Hence, $$ \mf{A}(\tau, \{0\})\cong \mc{E}^{\pm}.$$

Note that due to the identification of $\mc{E}_{k,\nu^{\pm}}$ with $\mf{A}(\tau, \{\nu^+,\nu^-\})$, the decomposition $\mc{E}_{k,\nu^{\pm}}=\mc{E}^{+}\oplus\mc{E}^{-} $ also follows directly from Lemma \ref{lem:decomp} by taking $S=\{\nu^+,\nu^-\}$.

By Theorem \ref{Thm2}, we have $$\tilde{\mc{E}}_{k,\nu^{\pm}}\cong \mf{sl}(2,\CC)\otimes_{\CC}\CC[\tilde{\wp},\tilde{\wp}',\xi],$$
where $\tilde{\wp}:=\wp_{\frac{1}{2}\LL}\circ t_{\nu^-}$, and where $\xi(z)=\mu_1(z-\nu^+)\mu_1(z-\nu^-)$. Corollary \ref{cor:isomorphism_classes} shows that $\tilde{\mc{E}}_{k_1,\nu_{1}^{\pm}}\cong \tilde{\mc{E}}_{k_2,\nu_2^{\pm}}$ if and only if there exists an isomorphism of complex tori $\phi:T_{k_1}\rightarrow T_{k_2}$ between complex tori with moduli $k_1$ and $k_2$ such that $\changes{\phi(\{2\nu_1^-,2\nu_1^+\})=\{2\nu_2^-,2\nu_2^+\}}$. 
We summarise the above in the next theorem. 

\begin{Theorem}\label{thm:uglov}
The Lie algebra $\tilde{\mc{E}}_{k,\nu^{\pm}}$ admits the normal form 
$$\tilde{\mc{E}}_{k,\nu^{\pm}}=\CC[\twp,\twp',\xi]\left\langle H,E,F\right\rangle,$$
where $\tilde{\wp}=\wp_{\frac{1}{2}\LL}\circ t_{\nu^-}$, and $\xi(z)=\mu_1(z-\nu^+)\mu_1(z-\nu^-)$, and $$H(z)=\omega(z-\nu^-)h, \quad E(z)=\omega(z-\nu^-)e, \quad F(z)=\omega(z-\nu^-)f,$$ with $\omega(z)=\Ad(\Omega(z))$ and $\Omega(z)$ defined in \eqref{def:Omega}.
Consequently, there is an isomorphism of Lie algebras
$$\mc{E}_{k,\nu^{\pm}}\cong \tilde{\mc{E}}_{k,\nu^{\pm}}\cong \mf{sl}(2,\CC)\otimes_{\CC}\CC[\tilde{\wp},\tilde{\wp}',\xi].$$
Moreover, $\mc{E}_{k_1,\nu_{1}^{\pm}}\cong \mc{E}_{k_2,\nu_2^{\pm}}$ if and only if  there exists an isomorphism $\phi$ between the corresponding complex tori such that $\phi(\changes{\{2\nu_1^-,2\nu_1^+\}})=\changes{\{2\nu_2^-,2\nu_2^+\}}$.
\end{Theorem}

\subsection{\changes{Holod's Lie algebra}}

We now turn our attention to the hidden symmetry algebra of the Landau--Lifshitz equation, introduced by Holod in \cite{holod1987hidden}. It appears in \cite{MR2920470} in connection with elliptic automorphic Lie algebras.

Recall the basis $v_1,v_2,v_3$ of $\mf{sl}(2,\CC)$ defined in \eqref{vi} and recall the curve $E_{r_1,r_2,r_3}$ given by $\lambda_i^2-\lambda_j^2=r_j-r_i$, for $i,j=1,2,3$. Let $A_i=r_i-\frac{1}{3}(r_1+r_2+r_3)$ and $\lambda=\frac{1}{3}(\lambda_1^2+\lambda_2^2+\lambda_3^2)$ so that $\lambda=\lambda_i^2+A_i$, independent of $i$.

The Holod algebra is the complex Lie algebra $\mc{H}_{r_1,r_2,r_3}$ defined on $E_{r_1,r_2,r_3}$, with basis 
\begin{equation}\label{holod_basis}
X_{i}^{2n+2}=\lambda^n\lambda_{j}\lambda_kv_i,\quad X_{i}^{2m+1}=\lambda^m\lambda_{i}v_i, \quad (n,m\in \ZZ),
\end{equation}
where $(i,j,k)$ is a cyclic permutation of $(1,2,3)$. For clarity, we write $\lambda_iv_j$ instead of $v_i\otimes \lambda_j$.  
The Lie structure is given by 
\begin{align*}
[X_i^{2l+1}, X_j^{2s+1}]&= \varepsilon_{ijk}X_k^{2(l+s)+2}, \\
[X_i^{2l+1}, X_j^{2s}]&=\varepsilon_{ijk}\big (X_k^{2(l+s)+1}-A_iX_k^{2(l+s)-1}\big), \\
[X_i^{2l}, X_j^{2s}]&= \varepsilon_{ijk}\big(X_k^{2(l+s)}-A_kX_k^{2(l+s)-2}\big),
\end{align*}
where $l,s\in \ZZ$. \changes{$\mc{H}_{\tau}$ is an almost-graded Lie algebra \cite{MR2920470}.}

The algebra $\mc{H}_{r_1,r_2,r_3}$ decomposes as a direct sum of subalgebras 
\begin{equation}\label{def:AB}
\mc{A}:=\bigoplus_{(i,j,k)\in \mathrm{Cycl}}\CC[\lambda] \lambda_iv_i\oplus \CC[\lambda] \lambda_j\lambda_kv_i,\quad \mc{B}:=\bigoplus_{(i,j,k)\in \mathrm{Cycl}}\CC[\lambda^{-1}]\lambda^{-1} \lambda_iv_i\oplus \CC[\lambda^{-1}] \lambda^{-1}\lambda_j\lambda_k v_i,
\end{equation}
which can be seen using  $\lambda=\lambda_i^2+A_i$.
By Corollary \ref{cor:Cor} and the correspondence $\lambda_iv_i \leftrightarrow X_i$, the subalgebra $\mc{A}$ is isomorphic to the $D_2$-automorphic Lie algebra with a single orbit of punctures.

Recall that the curve $E_{r_1,r_2,r_3}$ can be uniformised using the functions $\mu_i$ on $\CC/\LL$, with $$\mu_i(z)=\frac{\theta_1'(0|\tau)}{\theta_{i+1}(0|\tau)}\frac{\theta_{i+1}(2z|\tau)}{\theta_1(2z|\tau)}, \quad \lambda(\tau)=\frac{r_3-r_2}{r_3-r_1}.$$ 
Moreover, recall that $\wp_{\LL}(2z)=\frac{1}{4}\wp_{\frac{1}{2}\LL}(z)=\mu_i(z)^2+A_i$ by \eqref{eq:rel-wp-mu} (using that $A_i=r_i-\frac{1}{3}(r_1+r_2+r_3)$). 
Using this uniformisation, and identifying $\lambda$ with $\wp_{\LL}(2z)=\frac{1}{4}\wp_{\frac{1}{2}\LL}(z)$, the Holod algebra $\mc{H}_{r_1,r_2,r_3}$ admits an analytic realisation: 
\begin{equation}\label{def:analytic-basis}
X_{i}^{2n+2}(z)=\wp_{\LL}(2z)^n\mu_{j}(z)\mu_k(z)v_i, \,\,X_{i}^{2m+1}(z)=\wp_{\LL}(2z)^m\mu_{i}(z)v_i,
\end{equation}
for $m,n\in \ZZ$ and cyclic $(i,j,k)$. 

The complex-analytic isomorphism $E_{r_1,r_2,r_3}\cong T\setminus D_2\cdot \{0\}$ identifies the abstract and analytic descriptions of the algebra. We denote the analytic realisation of  $\mc{H}_{r_1,r_2,r_3}$ by $\mc{H}_{\tau}$ and we shall interchangeably refer to the Holod algebra as either $\mc{H}_{\tau}$ or  $\mc{H}_{r_1,r_2,r_3}$.  

 We can realise $\mc{H}_{r_1,r_2,r_3}$ as an automorphic Lie algebra on $E_{r_1,r_2,r_3}$ with symmetry group $D_2$. This perspective combined with the normal forms introduced in Section \ref{sec:Basics}, allows us to establish a Lie algebra isomorphism $$\mc{H}_{r_1,r_2,r_3}\cong \mf{sl}(2,\CC)\otimes_{\CC}R,$$ where $R$ is a ring of $D_2$-invariant meromorphic functions on $E_{r_1,r_2,r_3}$.

To realise $\mc{H}_{\tau}$ as an automorphic Lie algebra, we use the homomorphisms $\rho: D_2\rightarrow \A(\mf{sl}(2,\CC))$ and $\sigma:D_2\rightarrow \A(T)$ defined \changes{in \eqref{rhoprime}--\eqref{def:Gamma-action}.} 

Recall that the $\mu_i$ have simple poles at $ D_2\cdot \{0\}$, and note that $1/\lambda=1/\wp_{\LL}(2z)$ has poles at $\{\pm z_0\}$, where \changes{$\pm z_0\in \CC$} are the zeros of $\wp_{\LL}(2z)$, which depend on the lattice. 
The function $\wp_{\LL}(2z)$ is $D_2$-invariant and has poles exactly at the orbit $D_2\cdot \{0\}$. Using  \eqref{invariance}, it follows that $\mc{H}_{\tau}$ is a $D_2$-invariant Lie algebra on $\CC/\LL$ with poles restricted to the (union of) orbits $ D_2\cdot \{0\}$, $ D_2\cdot \{z_0\}$ and $ D_2\cdot \{-z_0\}$. (If $z_0=-z_0$, there are only two distinct orbits.) 
Hence
$$\mc{H}_{\tau}\subset\mf{A}(\mf{sl}(2,\CC), \tau,\{0,\pm z_0\})=(\mf{sl}(2,\CC)\otimes_{\CC}\mathcal{O}_{T\setminus D_2\cdot \{0,\pm z_0\}})^{\rho\otimes \tilde{\sigma}(D_2)}.$$
For brevity, we drop the $\mf{sl}(2,\CC)$ in the notation and write $\mf{A}( \tau,\{0,\pm z_0\})$. 

By Theorem \ref{Thm2}, with $S=\{0,\pm z_0\}$, $$\mf{A}( \tau, S)\cong \mf{sl}(2,\CC)\otimes_{\CC}\CC[\wp_{\frac{1}{2}\LL},\wp_{\frac{1}{2}\LL}',\xi_{-z_0},\xi_{z_0}],$$
where $\xi_p(z)=\mu_1(z)\mu_1(z-p)$, cf. \eqref{xip}. If we show that $\mf{A}( \tau, S)\subset \mc{H}_{\tau}$, it follows that $\mc{H}_{\tau}=\mf{A}( \tau, S)$. Consequently, $ \mc{H}_{\tau}$ is an $\mf{sl}(2,\CC)$-elliptic current algebra. 

Before stating our main result on $\mc{H}_{\tau}$, we need the following facts about the Weierstrass $\wp$-function associated with a complex torus isomorphic to $\CC/(\ZZ+\ZZ\mathrm{i})$.

\begin{Lemma}\label{Lem:zeros-wp}
Let $\LL=\ZZ+\ZZ \tau$. The Weierstrass $\wp_{\LL}$-function has a double zero if and only if $[\tau]=[\mathrm{i}]$ in $SL(2,\ZZ)\backslash \mathbb{H}$. Furthermore,  $\wp_{\LL_{\mathrm{i}}}(2z)=4\mu_2(z)^2$, where $\mu_2(z)=\frac{\theta_1'(0|\mathrm{i})}{\theta_3(0|\mathrm{i})}\frac{\theta_3(2z|\mathrm{i})}{\theta_1(2z|\mathrm{i})}$ and $\LL_{\mathrm{i}}=\ZZ+\ZZ\mathrm{i}$.
\end{Lemma}

\begin{proof}
The first statement is proved in \cite{MR650945}{, p. 401}. For the second, note that $\mu_2^2$ defines a function on $\CC/\frac{1}{2}\LL_{\mathrm{i}}$ by Lemma \ref{IsotypicalComponents}.  Both $4\mu_2^2$ and $\wp_{\frac{1}{2}\LL_{\mathrm{i}}}$ have a unique double zero at $(1+\mathrm{i})/4$ and a unique double  pole at $0$. Their Laurent expansions at $z=0$ of both functions begin as $z^{-2}+O(1)$.  By the  Riemann--Roch theorem, there exists a meromorphic function $g$ on $\CC/\frac{1}{2}\LL_{\mathrm{i}}$ with poles of order less than two, such that   $\wp_{\frac{1}{2}\LL_{\mathrm{i}}}=4\mu_2^2+g$. This forces $g$ to be a constant, and since $\mu_2^2$ and $\wp_{\frac{1}{2}\LL_{\mathrm{i}}}$ vanish at the same location, this forces $g=0$. Using  $\frac{1}{4}\wp_{\frac{1}{2}\LL_{\mathrm{i}}}(z)=\wp_{\LL_{\mathrm{i}}}(2z)$, the claim follows.
\end{proof}

\begin{Theorem}\label{thm:Holod}
Let $\LL=\ZZ+\ZZ \tau$, and let $T=\CC/\LL$. Holod's algebra $\mc{H}_{\tau}$ on $T$ is an automorphic Lie algebra based on $\mf{sl}(2,\CC)$ with respect to the actions $\rho$ and $\sigma$ as above, where $T$ is punctured at  $D_2\cdot S$, where $S=\{0,-z_0,z_0\}$ and $\pm z_0$ are the zeros of $\wp_{\LL}(2z)$:
$$\mc{H}_{\tau}=\mf{A}(\tau, \{0, \pm z_0\}).$$
A normal form of $\mc{H}_{\tau}$ is given by 
 $$\mc{H}_{\tau}=\CC[\wp_{\frac{1}{2}\LL},(\wp_{\frac{1}{2}\LL})^{-1},\wp_{\frac{1}{2}\LL}']\left\langle H,E,F\right\rangle,$$
 where $H=\Ad(\Omega)h$,  $E=\Ad(\Omega)e$, and  $F=\Ad(\Omega)f$, with $\Omega$ defined in \eqref{def:Omega}.
For $\tau=\mathrm{i}$, we have $$\mc{H}_{\mathrm{i}} \cong \mc{E}_{k,\nu^{\pm}},$$
where $\mc{E}_{k,\nu^{\pm}}$ is Uglov's algebra with $k=\tfrac{1}{\sqrt{2}}$ and $\nu^+=0$, $\nu^-=\tfrac{1+\mathrm{i}}{4}$.
\end{Theorem}

\begin{proof}
In the discussion prior to Lemma \ref{Lem:zeros-wp}, we already argued that $\mc{H}_{\tau}\subset \mf{A}( \tau,\{0,\pm z_0\})$. We now prove the reverse inclusion $\mf{A}( \tau,\{0,\pm z_0\})\subset \mc{H}_{\tau}$. We divide the proof into two cases. Case 1 corresponds to the situation $[\tau]\neq [\mathrm{i}]$, in which $\mc{H}_{\tau}$ consists of $\mf{sl}(2,\CC)$-valued maps with poles in exactly three distinct orbits --  namely at $D_2$-orbits of $z=0, -z_0,z_0$.

\begin{figure}[h!]
  \begin{center}
  \begin{tikzpicture}[scale=0.8]

  \clip (-1/3,-1/2) rectangle (8cm,4cm);
    \coordinate (O)   at (0,0);
    \pgftransformcm{1}{0}{0.5}{0.7}{\pgfpoint{0cm}{0cm}}
 
    \coordinate (Bone) at (0,5);
    \coordinate (Btwo) at (5,0);
    \coordinate (Bthree) at (0,5);
     \coordinate (Bfour) at (5,5);
    \coordinate (eta1) at (5/3,5/3);
    \coordinate (eta2) at (10/3,10/3);
    \coordinate (eta3) at (5/2,0);
       
    \draw [thick,black] (O) 
        -- (Bone) node [above left] {$\tau$};
    \draw [thick,black] (O)
        -- (Btwo) node [above left] {};
    \draw [thick, black] (Btwo)
        -- (Bfour) node [below right] {};
    \draw [thick,black] (O)
        -- ($(Btwo)$) node [below right] {$1$};
    \draw [ thick,black] (Bthree)
        -- (Bfour) node [above left] {};                
      
    \node[draw,blue,circle,inner sep=1pt] at (O) {};
       \node[draw,black,circle,inner sep=1pt,fill] at (5/3,5/3) {};
       \node[blue,inner sep=1pt] at (10/3,10/3) {\scriptsize $\times$};
       \node[draw,blue,circle,inner sep=1pt] at (2.5,0) {};
       \node[draw,blue,circle,inner sep=1pt] at (0,2.5) {};
       \node[draw,blue,circle,inner sep=1pt] at (2.5,2.5) {};
       \node[draw,black,circle,inner sep=1pt,fill] at (5/3,5/3+5/2) {};
        \node[draw,black,circle,inner sep=1pt,fill] at (5/3+5/2,5/3+5/2) {};
         \node[draw,black,circle,inner sep=1pt,fill] at (5/3+5/2,5/3) {};
         \node[blue,inner sep=1pt] at (-5/3+5/2,10/3) {\scriptsize $\times$};
         \node[blue, inner sep=1pt] at (10/3,-5/3+5/2) {\scriptsize $\times$};
         \node[blue, inner sep=1pt] at (-5/3+5/2,-5/3+5/2) {\scriptsize $\times$};
         
         \draw[thick] (-5/3+5/2,-5/3+5/2) node[right] {$z_0$};
         \draw[thick] (5/3+5/2,5/3+5/2) node[left] {$-z_0$};
         \draw[thick] (O) node[below right] {$0$};
    \end{tikzpicture}
  \end{center}
  \caption{\changes{The set of punctures $D_2\cdot \{0, \pm z_0\}$. The open (blue) dots represent the orbit $D_2\cdot \{0\}$, the filled (black) dots the orbit $D_2\cdot \{-z_0\}$, and the (blue) crosses $D_2\cdot \{z_0\}$.}}
    \label{figure:3}
\end{figure}

 The strategy is to decompose $\mf{A}( \tau,\{0,\pm z_0\})$ as a direct sum of three suitably constructed subalgebras, and then to show that each lies inside $\mc{H}_{\tau}$. Case 2 treats the special situation $[\tau]=[\mathrm{i}]$. The proof follows along similar lines, except that now $\wp_{\LL_{\mathrm{i}}}$ has a single double zero, and we decompose $\mf{A}( \tau,\{0,\pm z_0\})$ as a direct sum of only two subalgebras.  \\

\noindent \textbf{Case 1. $[\tau]\neq [\mathrm{i}]$}\\

Suppose that $\mc{H}_{\tau}$ is defined on a complex torus $\CC/\LL$ in an isomorphism class distinct from $[\tau]=[\mathrm{i}]$. By Lemma \ref{Lem:zeros-wp}, $\wp_{\LL}(2z)$ has two distinct sets of zeros at $D_2\cdot \{z_0\}$ and $D_2\cdot \{-z_0\}$ in $\CC/\LL$, both disjoint with $D_2\cdot \{0\}$\changes{, see Fig. \ref{figure:3}}. 
Consider the following elements of $\mc{H}_{\tau}$: $$Y_i^m(z)=\frac{1}{\wp_{\LL}(2z)^m}\mu_{i}(z)v_i,\quad Z_i^m(z)=\frac{1}{\wp_{\LL}(2z)^m}\mu_{j}(z)\mu_k(z)v_i,$$
where $i=1,2,3$, and $(i,j,k)$ \changes{is} a cyclic permutation of $(1,2,3)$, and $m>0$. 
The elements $Y_i^m$ and $Z_i^m$ are $D_2$-invariant with respect to action defined by $\rho\otimes \tilde{\sigma}$ as defined in \eqref{rhoprime}, \eqref{def:Gamma-action}. Recall that $\tilde{\sigma}$ is the induced action of $D_2$ on meromorphic functions on $T$. They have order $m$ poles precisely at $ D_2\cdot \{z_0,-z_0\}$. Indeed, the order $2m$ zeros of $1/\wp_{\LL}(2z)^m$ at $D_2\cdot \{0\}$ cancel out the poles at $D_2\cdot \{0\}$ of both $\mu_{i}v_i$ and $\mu_{j}\mu_kv_i$ (recall that the $\mu_i$ have poles exactly at $D_2\cdot \{0\}$). Thus only the zeros of $\wp_{\LL}(2z)$ contribute to the pole set of $Y_i^m$ and $Z_i^m$. 

Set $m=1$. The elements $Y_i^1$ and $Z_i^1$ have order 1 poles at $ D_2\cdot \{z_0,-z_0\}$. By the Riemann--Roch theorem, applied to the divisor $D:= \sum_{\gamma\in D_2 }(-\gamma\cdot z_0)+ \sum_{\gamma\in D_2 }(\gamma\cdot z_0)$ on $T$, we can rewrite
\begin{align*}
Y_i^1(z)&=\changes{c_1^{(i)}}\mu_i(z-z_0)v_i+\changes{c_2^{(i)}}\mu_i(z+z_0)v_i,\\
Z_i^1(z)&=\changes{d_1^{(i)}}\mu_i(z-z_0)v_i+\changes{d_2^{(i)}}\mu_i(z+z_0)v_i,
\end{align*}
with $i=1,2,3$ and for some constants $\changes{c_j^{(i)},d_j^{(i)}}$ which satisfy $\changes{\beta_i:=d_2^{(i)}c_1^{(i)}-c_2^{(i)}d_1^{(i)}\neq 0}$. This follows because $Y_i^1$ and $Z_i^1$ are $\CC$-linearly independent, for each $i$, since they vanish at distinct points (the $\mu_i$ have disjoint sets of zeros). The elements 
$$
W_i^+:=\changes{d_2^{(i)}Y_i^1-c_2^{(i)}Z_i^1,\quad W_i^-:=d_1^{(i)}Y_i^1-c_1^{(i)}Z_i^1},
$$
only have order 1 poles in $ D_2\cdot \{z_0\}$ and order 1 poles in $ D_2\cdot \{-z_0\}$, respectively. More explicitly, 
\begin{align*}
W_i^+(z)&=\changes{\beta_i}\mu_i(z-z_0)v_i=\changes{\beta_i}X_i(z-z_0),\\
W_i^-(z)&=-\changes{\beta_i}\mu_i(z+z_0)v_i=-\changes{\beta_i} X_i(z+z_0),
\end{align*}
 where the $X_i$ are defined in \eqref{gen:X}.

Consider the two Lie subalgebras $\mc{L}_{\pm}\subset \mc{H}_{\tau}$ generated as Lie algebras by the $W_i^{\pm}$, that is, $$\mc{L}_{\pm}=\CC\langle W_1^{\pm} , W_2^{\pm}, W_3^{\pm}  \rangle.$$
The $\mc{L}_{\pm}$ coincide with the automorphic Lie algebras $\mf{A}( \tau,\{\pm z_0\})$, which follows from Corollary \ref{cor:Cor} together with \eqref{eq:isom-one-orbit}. We have therefore established that $\mf{A}( \tau,\{\pm z_0\})\subset \mc{H}_{\tau}$. The inclusion $\mf{A}( \tau,\{0\})\subset \mc{H}_{\tau}$ holds as well since $\mc{A}$, as defined in \eqref{def:AB}, can be identified with $\mf{A}( \tau,\{0\})$. 

The automorphic Lie algebra $\mf{A}( \tau,\{0,\pm z_0\})$ decomposes as
\begin{equation*}
\mf{A}( \tau,\{0,\pm z_0\})=\mf{A}( \tau,\{0\})\oplus \mf{A}( \tau,\{z_0\})\oplus \mf{A}( \tau,\{-z_0\})
\end{equation*}
by Lemma \ref{lem:decomp}.
We had already established that $\mc{H}_{\tau}\subset \mf{A}( \tau,\{0,\pm z_0\})$, and thus we obtain $\mf{A}( \tau,\{0,\pm z_0\})= \mc{H}_{\tau}$ as Lie subalgebras. Alternatively, one can use  $\mc{H}_{\tau}\cong \mc{A}\oplus \mc{B}$ by \eqref{def:AB}, and note that $\mc{B}$ can be identified with $\mf{A}( \tau,\{z_0\})\oplus \mf{A}( \tau,\{-z_0\})$. Hence, the Holod algebra $\mc{H}_{\tau}$ coincides with the automorphic Lie algebra $\mf{A}( \tau,\{0,\pm z_0\})$.

Theorem \ref{Thm2} establishes that $$ \mc{H}_{\tau}=\mf{A}( \tau,\{0,\pm z_0\})=\CC[\wp_{\frac{1}{2}\LL},\wp_{\frac{1}{2}\LL}',\xi_{-z_0},\xi_{z_0}]\left\langle H,E,F\right\rangle.$$
The above description of $\mc{H}_{\tau}$ may be simplified in the following way. 
The divisor of $\wp_{\frac{1}{2}\LL}'/\wp_{\frac{1}{2}\LL}$ (defined on $\CC/\frac{1}{2}\LL$) equals
\begin{equation}\label{eq:div}
\mathrm{div}\left(\frac{\wp_{\frac{1}{2}\LL}'}{\wp_{\frac{1}{2}\LL}}\right)=-(-z_0)-(z_0)-(0)+(1/4)+(\tau/4)+((1+\tau)/4).
\end{equation}
We may take a $\CC$-linear combination of $1/\wp_{\frac{1}{2}\LL}$ and $\wp_{\frac{1}{2}\LL}'/\wp_{\frac{1}{2}\LL}$ to eliminate the pole at $z_0$ or $-z_0$ (but not at the same time). 

Consider the divisor $D_{\pm}=(0)+(\pm z_0)$ on $\CC/\frac{1}{2}\LL$. The spaces $L(D_{\pm})$ are two-dimensional and are spanned by $1$ and $\xi_{\pm z_0}$, respectively, where we recall that the functions $\xi_{p}$ are defined as $\xi_{p}(z)=\mu_1(z)\mu_1(z-p)$. It follows that the functions $\xi_{\pm z_0}$
lie in the ring generated by $1/\wp_{\frac{1}{2}\LL}$ and $\wp_{\frac{1}{2}\LL}'/\wp_{\frac{1}{2}\LL}$. We may conclude that $\CC[\wp_{\frac{1}{2}\LL},\wp_{\frac{1}{2}\LL}',\xi_{-z_0},\xi_{z_0}]=\CC[\wp_{\frac{1}{2}\LL}, (\wp_{\frac{1}{2}\LL})^{-1}, \wp_{\frac{1}{2}\LL}']$, thereby proving the first claim. \\

\noindent \textbf{Case 2. $[\tau]= [\mathrm{i}]$}\\

We will now consider the case $[\tau]=[\mathrm{i}]$, where the two zeros of $\wp_{\LL}$ coincide. For convenience, we take $\tau=\mathrm{i}$, so that $T=\CC/\LL_{\mathrm{i}}$, with $\LL_{\mathrm{i}}=\ZZ+\ZZ\mathrm{i}$. The general argument is obtained by replacing $\LL_{\mathrm{i}}$ by $\alpha\LL_{\mathrm{i}}$ for some $\alpha\in \CC^*$, and scaling the zero $z_0=\frac{1+\mathrm{i}}{4}$ of $\wp_{\frac{1}{2}\LL_{\mathrm{i}}}$ to $\alpha z_0$.

Reasoning as in the previous case, we have $\mc{H}_{\mathrm{i}}\subset \mf{A}( \mathrm{i}, \{0,\tfrac{1+\mathrm{i}}{4}\})$. To prove the reverse inclusion, we take the basis elements $Y_i^m(z)$ and $Z_i^m(z)$ and consider $$Y_2^1(z)=\frac{1}{\mu_2(z)}v_2, \quad Z_1^1(z)=\frac{1}{\mu_2(z)}\mu_3(z)v_1, \quad Z_3^1(z)=\frac{1}{\mu_2(z)}\mu_1(z)v_3,$$
where $\mu_2^2(z)=\frac{1}{4}\wp_{\frac{1}{2}\LL_{\mathrm{i}}}(z)=\wp_{\LL}(2z)$ by Lemma \ref{Lem:zeros-wp}. These elements have poles precisely at $D_2\cdot \{\frac{1+\mathrm{i}}{4}\}$ since $\mu_2$ vanishes exactly at this set. A similar argument as before shows that, up to scalar multiples, $Y_2^1(z)=X_2(z-z_0)$, $Z_1^1(z)=X_1(z-z_0)$ and $Z_3^1(z)=X_3(z-z_0)$. Consequently, both the Lie algebras 
$$
\mf{A}( \mathrm{i}, \{0\})=\CC\langle X_1(z), X_2(z), X_3(z) \rangle,\quad \mf{A}( \mathrm{i}, \{z_0\})=\CC\langle X_1(z-z_0),X_2(z-z_0), X_3(z-z_0) \rangle
$$
are contained in $\mc{H}_{\mathrm{i}}$, and hence also $\mf{A}( \mathrm{i}, \{0,z_0\})=\mf{A}( \mathrm{i}, \{0\})\oplus \mf{A}( \mathrm{i}, \{z_0\})\subset \mc{H}_{\mathrm{i}}$. Together with Theorem \ref{Thm2}, this establishes 
$$\mc{H}_{\mathrm{i}}=\CC[\wp_{\frac{1}{2}\LL_{\mathrm{i}}},\wp_{\frac{1}{2}\LL_{\mathrm{i}}}',\xi_{z_0}]\left\langle H,E,F\right\rangle.$$ 
Hence 
$$\mc{H}_{\mathrm{i}}\cong \mf{sl}(2,R) \cong \mf{sl}(2,\CC)\otimes_{\CC}R,$$ where $R=\CC[\wp_{\frac{1}{2}\LL_{\mathrm{i}}},\wp_{\frac{1}{2}\LL_{\mathrm{i}}}',\xi_{z_0}]$. 
For $\tau=\mathrm{i}$, the divisor \eqref{eq:div} becomes $$\mathrm{div}\left(\frac{\wp_{\frac{1}{2}\LL_{\mathrm{i}}}'}{\wp_{\frac{1}{2}\LL_{\mathrm{i}}}}\right)=-((1+\mathrm{i})/4)-(0)+(1/4)+(\mathrm{i}/4),$$ and since $\xi_{z_0}$ also has simple poles precisely at $0$ and $z_0=(1+\mathrm{i})/4$, we may replace $\xi_{z_0}$ by $\wp_{\frac{1}{2}\LL_{\mathrm{i}}}'/\wp_{\frac{1}{2}\LL_{\mathrm{i}}}$ in the description of $R$. Therefore, we have $R=\CC[\wp_{\frac{1}{2}\LL_{\mathrm{i}}},\wp_{\frac{1}{2}\LL_{\mathrm{i}}}'/\wp_{\frac{1}{2}\LL_{\mathrm{i}}}]$. Moreover, using that $(\wp_{\LL_{\mathrm{i}}}')^2=4\wp_{\LL_{\mathrm{i}}}^3-g_2(\mathrm{i})\wp_{\LL_{\mathrm{i}}}$ and $g_2(\mathrm{i})\neq 0$, we have $1/\wp_{\frac{1}{2}\LL_{\mathrm{i}}}\in R$, so that $R=\CC[\wp_{\frac{1}{2}\LL_{\mathrm{i}}},(\wp_{\frac{1}{2}\LL_{\mathrm{i}}})^{-1}, \wp_{\frac{1}{2}\LL_{\mathrm{i}}}']$. The general case of $[\tau]=[\mathrm{i}]$ now follows as well, as we have remarked above. 

By Theorem \ref{thm:uglov}, $\mc{H}_{\mathrm{i}}\cong \mc{E}_{\frac{1}{\sqrt{2}}, 0, \frac{1+\mathrm{i}}{4}}$, using that for $\tau=\mathrm{i}$, we have $k=\frac{\theta_2^2(0|\mathrm{i})}{\theta_3^2(0|\mathrm{i})}=\frac{1}{\sqrt{2}}$.
\end{proof}

\changes{\begin{Remark}
Both Lie algebras $\mc{H}_{\tau}$ and  $\mf{sl}(2,\CC)\otimes_{\CC}\CC[\wp,\wp^{-1},\wp']$ (with $\wp:=\wp_{\frac{1}{2}\LL}$) are almost-graded, but with different almost-gradings \cite{10.1063/1.4824152}. In particular, the isomorphism $\mc{H}_{\tau}\cong \mf{sl}(2,\CC)\otimes_{\CC}\CC[\wp,\wp^{-1},\wp']$ obtained in Theorem \ref{thm:Holod} does not preserve the almost-graded structures.
\end{Remark}}

The normal form of $\mc{H}_{\tau}$ in Theorem \ref{thm:Holod} is given in terms of the modulus $\tau$, i.e., in terms of a uniformisation of the underlying elliptic curve $E_{r_1,r_2,r_3}$. We can also give a more intrinsic description of $\mc{H}_{\tau}$ as follows. Recall that $\lambda=\lambda_i^2+A_i$ with $A_i=r_i-\frac{1}{3}(r_1+r_2+r_3)$. Reformulating Theorem \ref{thm:Holod}, we obtain
$$\mc{H}_{r_1,r_2,r_3}=\CC[\lambda,\lambda^{-1}, \lambda_1\lambda_2\lambda_3]\left\langle \tilde{H},\tilde{E},\tilde{F}\right\rangle\cong \mf{sl}(2,R),$$
where $\tilde{H},\tilde{E},\tilde{F}$ (expressed in terms of $\lambda_1,\lambda_2,\lambda_3$) are defined in \eqref{def:HEFtilde}, and where $$R=\CC[x,x^{-1},y]/\left(y^2-(x-A_1)(x-A_2)(x-A_3)\right).$$ These generators satisfy the standard $\mf{sl}(2,\CC)$-relations. 
The normal form of $\mc{H}_{\tau}$ allows us to determine its $\CC$-isomorphism classes, as stated in the next corollary.

\begin{Corollary}
For the Holod algebra $\mc{H}_{\tau}$, we have $\mc{H}_{\tau}\cong \mc{H}_{\tau'}$ as Lie algebras if and only if $[\tau]=[\tau']$.
\end{Corollary}
\begin{proof}
\changes{We will use Corollary \ref{cor:isomorphism_classes} to prove the claim. Let $T=\CC/\LL$ and $T'=\CC/\LL'$, where $\LL=\ZZ+\ZZ\tau$ and $\LL'=\ZZ+\ZZ\tau'$ for some $\tau,\tau'\in \mathbb{H}.$ Recall that $\mc{H}_{\tau}=(\mf{sl}(2,\CC)\otimes_{\CC}\mathcal{O}_{T\setminus D_2\cdot \{0,\pm z_0\}})^{\rho\otimes \tilde{\sigma}(D_2)}.$} 

\changes{The complex tori $T$ and $T'$ are isomorphic}
if and only if there exists $\alpha\in \CC^*$ such that $\LL'=\alpha\LL$. Equivalently, $[\tau]=[\tau']$.
Now, the points $\pm z_0\in \CC/\LL$ are the zeros of $\wp_{\LL}(2z)$ if and only if $\pm \alpha z_0\in \CC/\alpha\LL$ are the zeros of $\wp_{\alpha\LL}(2z)$. 
An isomorphism $\CC/\LL\cong \CC/\LL'$ is given by $\phi(z)=\alpha z$,
which maps the set of punctures $ \changes{S:=\{0,\pm z_0\}}\subset \CC/\LL$ 
to $\changes{S':=\{0, \pm \alpha z_0\}}\subset \CC/\alpha\LL$. \changes{Consequently, $\phi(2S)=2S'$.} Since $\mc{H}_{\tau}=\mf{A}(\tau, S)$ and $\mc{H}_{\tau'}=\mf{A}(\tau', S')$, we may apply Corollary \ref{cor:isomorphism_classes}, which establishes that $\mc{H}_{\tau}\cong \mc{H}_{\tau'}$ if and only if $[\tau]=[\tau']$.
\end{proof}

\section{Multicomponent Landau--Lifshitz systems}\label{sec:LL}
In this section, we consider the multicomponent Landau--Lifshitz system introduced in \cite{MR1821313},
\begin{equation*}
S_t=\left(S_{xx}+\frac{3}{2}\langle S_x,S_x\rangle S\right)_x+\frac{3}{2}\langle S,RS\rangle S_x,\quad \langle S,S\rangle=1,
\end{equation*}
where $R=\diag(r_1,\ldots,r_n)$ and $S=(s^1(x,t),\ldots, s^n(x,t))^T\in \mathbb{K}^n$, with $n\geq 2$. We assume that the parameters $r_1,\ldots, r_n\in \mathbb{K}$ are pairwise distinct, i.e., $r_i\neq r_j$ for all $i\neq j$. 

Although the system is defined more generally, we will focus on the case $\mathbb{K}=\CC$ and $n=3$, for which the system coincides with the higher symmetry of third order for the fully anisotropic Landau--Lifshitz equation. We retain the general $n$ notation in some places to indicate that the methodology can, in principle, be applied to general $n$.

 We show that the zero-curvature representation (ZCR) for this equation found in \cite{MR3035112}, admits a natural interpretation in terms of the automorphic Lie algebra $$(\mf{so}_{3,1}\otimes_{\RR}\CC[E_{r_1,r_2,r_3}])^{D_2},$$ where $\mf{so}_{3,1}$ is the Lie algebra of the matrix Lie group $O(3,1)$, and $\CC[E_{r_1,r_2,r_3}]$ is the coordinate ring of the elliptic curve $E_{r_1,r_2,r_3}$.  The action of $D_2$ on $\mf{so}_{3,1}\otimes_{\RR}\CC[E_{r_1,r_2,r_3}]$ will be specified below. 

We show that the above automorphic Lie algebra is isomorphic, as a complex Lie algebra, to the direct sum of two copies of the current algebra $\mf{sl}(2,\CC)\otimes_{\CC}\mathcal{O}_{T\setminus D_2\cdot \{0\}}^{D_2}$. Furthermore, we shall make the connection between the WE (Wahlquist--Estabrook) algebra of the multicomponent Landau--Lifshitz system $\eqref{eq:PDE}$ for $n=3$ and the above automorphic Lie algebra.  

Finally, we show that the realisation of $\mf{R}_{r_1,r_2,r_3}$ as an automorphic Lie algebra (see Remark \ref{rem:R}) leads to a transparent description of the prolongation algebras of both the fully anisotropic Landau--Lifshitz equation and the non-singular Krichever--Novikov equation. In particular, it follows from this that both are isomorphic to the Lie algebra $\mf{S}_{\LL}$ defined as $$\mf{S}_{\LL}:=\mf{sl}(2,\CC)\otimes_{\CC}\CC[\wp_{\LL},\wp_{\LL}'] \oplus \CC^2,$$ where the sum is a direct sum of Lie algebras and $\LL$ is a suitable lattice. 
 
The notation $\TT$ will be used throughout this section to denote a complex torus $T$ punctured at a single orbit, namely $\TT=T\setminus D_2\cdot \{0\}$. We largely follow the approach and notations of \cite{MR3035112}.

Consider the real Lie algebra $\mf{so}_{n,1}\subset \mf{gl}_{n+1}(\mathbb{R})$, which is the Lie algebra of the matrix Lie group $O(n,1)$  of linear transformations preserving the standard bilinear form of signature $(n,1)$. Explicitly, $$\mf{so}_{n,1}=\{X\in \mathrm{Mat}_{n+1}(\mathbb{R}): X^TI_{n,1}=-I_{n,1}X\},$$
where $I_{n,1}=\diag(1,1,\ldots, 1,-1)\in GL_{n+1}(\mathbb{R})$ and $(\cdot)^T$ denotes matrix transpose. Denote by $E_{ij}$ the matrix with a 1 in position $(i,j)$ and zeros everywhere else. The following elements form a basis for $\mf{so}_{n,1}$: 
\begin{equation}\label{basis-so3,1}
A_{ij}:=E_{ij}-E_{ji},\quad i<j\leq n, \quad B_{l,n+1}:=E_{l,n+1}+E_{n+1,l},\quad l=1,\ldots, n.
\end{equation}

Recall the algebraic curve \eqref{eq:algcurve} in $\CC^n$ defined by 
\begin{equation*}
E_{r_1,\ldots, r_n}: \lambda_i^2-\lambda_j^2=r_j-r_i,\quad i,j=1,\ldots, n.
\end{equation*}
The genus of the (compactified) curve is given by $1+(n-3)2^{n-2}$; see \cite{MR1821313}.

The Lie algebra $\mf{so}_{3,1}$ is isomorphic to $\mf{sl}(2,\CC)$ viewed as a real Lie algebra, and 
its complexification $(\mf{so}_{3,1})_{\CC}:=\mf{so}_{3,1}\otimes_{\RR}\CC$  is isomorphic to  $\mf{so}(4,\CC)\cong\mf{sl}(2,\CC)\oplus \mf{sl}(2,\CC)$ as a complex Lie algebra. More generally, one has $(\mf{so}_{n,1})_{\CC}\cong \mf{so}(n+1,\CC)$.

Consequently, viewing $\mf{so}_{3,1}\otimes_{\RR}\CC[E_{r_1,r_2,r_3}]$ as a complex Lie algebra, we obtain a natural isomorphism  $$\mf{so}_{3,1}\otimes_{\RR}\CC[E_{r_1,r_2,r_3}]\cong \left(\mf{sl}(2,\CC)\otimes_{\CC}\CC[E_{r_1,r_2,r_3}] \right)\oplus \left(\mf{sl}(2,\CC)\otimes_{\CC}\CC[E_{r_1,r_2,r_3}]\right),$$ 
where the sum is a direct sum of complex Lie algebras. 

Following \cite{MR3035112}, we now introduce a $\mf{so}_{n,1}$-valued ZCR for the multicomponent Landau--Lifshitz system \eqref{eq:PDE}. Let $D_x$ and $D_t$ denote the total derivative operators with respect to $x$ and $t$ corresponding to \eqref{eq:PDE}, respectively. The ZCR is given by 
\begin{subequations}\label{eq:ZCR}
\begin{gather}
 M(\lambda_1,\ldots, \lambda_n)=\sum_{i=1}^n(E_{i,n+1} +E_{n+1,i})\otimes s^i\lambda_i,\\
N(\lambda_1,\ldots, \lambda_n)=D_x^2(M)+[D_x(M),M]+\left(r_1+\lambda_1^2+\frac{1}{2}\langle S,RS\rangle  +\frac{3}{2}\langle S_x,S_x\rangle\right)M,\\
D_x(N)-D_t(M)+[M,N]=0,
\end{gather}
\end{subequations}
where $S(x,t)=(s^1(x,t),\ldots, s^n(x,t))^T$\changes{, $\langle S,S\rangle=1$,} and $\lambda_1,\ldots, \lambda_n\in \mathbb{C}$ are parameters satisfying \eqref{eq:algcurve}.

The Lax matrices $M$ and $N$ in \eqref{eq:ZCR} can be interpreted as regular maps $E_{r_1,\ldots, r_n}\rightarrow \mf{so}_{n,1}$. 
For $n=3$ (the elliptic case), we construct an action of $D_2$ on $$\mf{so}_{3,1}\otimes_{\RR}\ot\cong\mf{so}_{3,1}\otimes_{\RR}\CC[E_{r_1,r_2,r_3}]$$  and show that $M$ and $N$ may be interpreted as $D_2$-equivariant maps on the elliptic curve $E_{r_1,r_2,r_3}$. \changes{That is, they are $D_2$-invariant matrices with respect to the action $\rho\otimes \sigma.$} In this way, they naturally define elements of an elliptic automorphic Lie algebra. 

We now describe the ingredients needed to make the automorphic Lie algebra framework available. Define a representation $\tilde{\rho}:D_2\rightarrow \A(\mf{so}_{3,1})$ by
\begin{equation}\label{rhotilde}
\tilde{\rho}(t_1)=\Ad\begin{pmatrix} 1&&&\\ &-1&&\\ &&-1&\\&&&1  \end{pmatrix},\quad \tilde{\rho}(t_2)=\Ad \begin{pmatrix} -1&&&\\ &-1&&\\ &&1&\\&&&1  \end{pmatrix}.
\end{equation} 

Recall the homomorphism $\sigma:D_2\rightarrow \A(E_{r_1,r_2,r_3})$ defined by
\begin{equation*}
\sigma(t_1)(\lambda_1,\lambda_2,\lambda_3)=(\lambda_1,-\lambda_2,-\lambda_3),\quad \sigma(t_2)(\lambda_1,\lambda_2,\lambda_3)=(-\lambda_1,-\lambda_2,\lambda_3),
\end{equation*}
or equivalently, in the analytic setting, by the transformations given in \eqref{def:Gamma-action}. 

We now show that the elements $$Q_i=(E_{i,4}+E_{4,i})\otimes \lambda_{i},\quad i=1,2,3$$ are invariant under the action of $D_2$ defined by $\tilde{\rho}\otimes \sigma$ (or equivalently by $\tilde{\rho}\otimes \tilde{\sigma}$ in the analytic formulation, where $\tilde{\sigma}$ denotes the induced action on $\ot$). 

\begin{Lemma}\label{lem:Qi}
Let  $D_2=\langle t_1,t_2\rangle $ and let $\tilde{\rho}(t_i)=\Ad(M_i)$ for $i=1,2$, where $M_i$ are as in \eqref{rhotilde}.  Then the elements $$Q_i=(E_{i,4}+E_{4,i})\otimes \lambda_{i}, \quad i=1,2,3,$$ are invariant under the action of $D_2$ defined by $\tilde{\rho}\otimes\sigma.$
\end{Lemma}

\begin{proof}
For $\bold{k}=(k_1,k_2,k_3,k_4)$ with $k_i\in \{0,1\}$, define $$\Delta_\bold{k}:=\Delta_{(k_1,k_2,k_3,k_4)}:=\diag((-1)^{k_1},\ldots, (-1)^{k_4}).$$ 
Then $\Ad(\Delta_\bold{k})E_{ij}=(-1)^{k_i-k_j}E_{ij}$. The automorphisms $\Ad(M_1),\Ad(M_2)$ defined in \eqref{rhotilde} generate the group $D_2$ inside $\A(\mf{so}_{3,1})$. Let $\bold{k}^{(1)}=(0,1,1,0)$ and $\bold{k}^{(2)}=(1,1,0,0)$. Note that $M_1=\Delta_{\bold{k}^{(1)}}$ and $M_2=\Delta_{\bold{k}^{(2)}}$. It follows that $$\Ad(M_{\ell})(E_{i,4}+E_{4,i})=(-1)^{k_i^{(\ell)}}(E_{i,4}+E_{4,i}), \quad \ell=1,2,$$
for $i=1,2,3$.
The action of $D_2$ on $E_{r_1,r_2,r_3}$ can be written as $$\lambda_i \xmapsto[]{\, t_{\ell}\,}(-1)^{k_i^{(\ell)}}\lambda_i, \quad i=1,2,3.$$ 
Consequently,
$$\tilde{\rho}(t_{\ell})\otimes \sigma(t_{\ell})(Q_i)= (-1)^{k_i^{(\ell)}}(E_{i,4}+E_{4,i})\otimes (-1)^{k_i^{(\ell)}}\lambda_i = Q_i$$ for $i=1,2,3$. Hence each $Q_i$ is invariant under the action of $D_2$. 
\end{proof}

The invariance of the $Q_i$ implies that $M$ and $N$  are invariants as well. Consequently, we may interpret $M$ and $N$ as elements of the automorphic Lie algebra 
\begin{align*}
(\mf{so}_{3,1}\otimes_{\mathbb{R}}\CC[E_{r_1,r_2,r_3}])^{\tilde{\rho}\otimes \sigma(D_2)},
\end{align*}
where $\tilde{\rho}$ and $\sigma$ are defined in \eqref{rhotilde} and \eqref{ActionC2xC2}, respectively.
In particular, this establishes that for $n=3$, the Lax matrices in \eqref{eq:ZCR} for \eqref{eq:PDE} can be interpreted as elements of the automorphic Lie algebra $(\mf{so}_{3,1}\otimes_{\RR}\CC[E_{r_1,r_2,r_3}])^{\tilde{\rho}\otimes \sigma(D_2)}\cong (\mf{so}_{3,1}\otimes_{\RR}\ot)^{\tilde{\rho}\otimes \tilde{\sigma}(D_2)}$. 

We now relate this automorphic Lie algebra to the WE algebra of the multicomponent Landau--Lifshitz equation for $n=3$ and focussing on $\mathbb{K}=\CC$. In \cite{MR3035112}, the subalgebra $$L(n)\subset (\mf{so}_{n,1})_{\CC}\otimes_{\CC}\CC[\lambda_1,\ldots,\lambda_n]/I,$$ generated by $Q_1,\ldots, Q_n$, is studied. Here $I$ is the ideal generated by $\lambda_i^2-\lambda_j^2-r_j+r_i$ for $ i,j=1,\ldots, n$ and  $$Q_i=(E_{i,n+1}+E_{n+1,i})\otimes \lambda_i \in(\mf{so}_{n,1})_{\CC}\otimes_{\CC}\CC[\lambda_1,\ldots,\lambda_n]/I.$$
\changes{Consider the complex Lie algebra} $\g(n)$ generated by $p_1,\ldots,p_n$ subject to the relations 
\begin{subequations}\label{Relations}
\begin{align}
[p_i,[p_j,p_k]]&=0, \quad i\neq j\neq k\neq i, \\ 
[p_i,[p_i,p_k]]-[p_j,[p_j,p_k]]&=(r_j-r_i)p_k,\quad i\neq k,\quad j\neq k,\quad i,j,k=1,\ldots,n. 
\end{align}
\end{subequations} 
\changes{There is an isomorphism $L(n)\cong \g(n)$ for $n\geq 3$, given by $Q_i\mapsto p_i$ \cite[Theorem 4]{MR3035112}.}
The WE algebra of system \eqref{eq:PDE} for $\mathbb{K}=\CC$ is isomorphic to $\g(n)\oplus \CC^2$  \changes{\cite[Theorem 3]{MR3035112}}.
To emphasise the dependence on the parameters $r_i$, we shall now write $\g(r_1,r_2,r_3)$ and $L(r_1,r_2,r_3)$ instead of $\g(3)$ and $L(3)$.  

It follows from Proposition \ref{prop:isom_pi} \changes{and Remark \ref{rem:R}} that $$\g(r_1,r_2,r_3)\cong \mf{R}_{r_1,r_2,r_3}\cong  (\mf{sl}(2,\CC)\otimes_{\CC}\ot)^{D_2}.$$ Moreover, $L(r_1,r_2,r_3)\cong \g(r_1,r_2,r_3)$ according to \cite[Theorem 4]{MR3035112}.  
By Theorem \ref{Thm2}
we obtain $$\g(r_1,r_2,r_3)\cong  \mf{sl}(2,\CC)\otimes_{\CC}\CC[x,y]/(y^2-4x^3+g_2(\LL)x+g_3(\LL)).$$
 
We now show that this implies an isomorphism of complex Lie algebras
\begin{equation}\label{liealgiso}
(\mf{so}_{3,1}\otimes_{\RR}\ot)^{\tilde{\rho}\otimes \tilde{\sigma}(D_2)}\cong \g(r_1,r_2,r_3)\oplus \g(r_1,r_2,r_3),
\end{equation}
where the sum is a direct sum of Lie algebras. 

Recall the basis elements of $\mf{so}_{3,1}$ defined in \eqref{basis-so3,1}. Consider the following linear combinations of the $A_{ij}$ and $B_{ij}$, which form a basis of $(\mf{so}_{3,1})_{\CC}$:
$$
K_i  = \frac{1}{2}(A_{jk}-\mathrm{i}B_{jk}), \quad 
L_i  = \frac{1}{2}(A_{jk}+\mathrm{i}B_{jk}),
$$
where $(i,j,k)$ is a cyclic permutation of $(1,2,3)$. We complexify the representation $\tilde{\rho}$ defined in \eqref{rhotilde}, and we use the same notation for the result.

The Lie subalgebras $\CC\langle K_1, K_2,  K_3\rangle$ and  $\CC\langle L_1, L_2,  L_3\rangle$ are $D_2$-submodules of $(\mf{so}_{3,1})_{\CC}$ and satisfy $$[K_i,K_j]=\varepsilon_{ijk}K_k,\quad [L_i,L_j]=\varepsilon_{ijk}L_k,\quad [K_i,L_j]=0.$$
Thus each subalgebra is isomorphic to $\mf{sl}(2,\CC)$, and there is a $D_2$-equivariant Lie algebra isomorphism $(\mf{so}_{3,1})_{\CC}\cong \mf{sl}(2,\CC)\oplus \mf{sl}(2,\CC)$, where the sum is of Lie algebras. This establishes \eqref{liealgiso}.

We summarise these results as follows.
\begin{Proposition}\label{prop:iso_g(3)}
There is an isomorphism of complex Lie algebras 
\begin{equation*}
(\mf{so}_{3,1}\otimes_{\RR}\CC[E_{r_1,r_2,r_3}])^{\tilde{\rho}\otimes \sigma(D_2)}\cong \g(r_1,r_2,r_3)\oplus \g(r_1,r_2,r_3),
\end{equation*}
where the sum is of Lie algebras.
Moreover, for $$L(r_1,r_2,r_3)=\CC\langle Q_1,Q_2,Q_3\rangle\subset (\mf{so}_{3,1}\otimes_{\RR}\CC[E_{r_1,r_2,r_3}])^{\tilde{\rho}\otimes \sigma(D_2)},$$ there are isomorphisms of Lie algebras $$ L(r_1,r_2,r_3)\cong \g(r_1,r_2,r_3)\cong \mf{sl}(2,\CC)\otimes_{\CC} \CC[x,y]/(y^2-4x^3+g_2(\LL)x+g_3(\LL)).$$
\end{Proposition}

The WE algebra of the Landau--Lifshitz equation is closely related to the Lie algebra $\g(r_1,r_2,r_3)$; see \cite{MR1218995, MR3035112}. More precisely, this prolongation algebra is isomorphic to a direct sum of Lie algebras $$\g(r_1,r_2,r_3)\oplus \CC^2,$$ where $\CC^2$ is the two-dimensional complex Lie algebra. This algebra is not usually realised as an automorphic Lie algebra with symmetry group $D_2$. 

Similarly, the generalised prolongation algebra of the non-singular Krichever--Novikov equation, as described in \cite{MR1947240},  is also isomorphic to $\g(r_1,r_2,r_3)\oplus \CC^2$ \cite[Theorem 2]{MR1947240}.  
We summarise these observations in the following corollary.

\begin{Corollary}
The prolongation algebra in the sense of Wahlquist--Estabrook of the Landau--Lifshitz equation are isomorphic, as Lie algebras, to $$\mf{S}_{\LL}=\mf{sl}(2,\CC)\otimes_{\CC}\CC[\wp_{\LL},\wp_{\LL}'] \oplus \CC^2,$$
where $\LL$ corresponds to the pairwise distinct parameters $r_1,r_2,r_3$ in the Landau--Lifshitz equation \eqref{eq:LL}. The prolongation algebra of the non-singular Krichever--Novikov equation, as discussed in \cite{MR1947240}, is also isomorphic to $\mf{S}_{\LL}$ for a suitable choice of $r_1,r_2,r_3$.  
\end{Corollary}

\begin{Remark}
The Lie algebra generated by $X_1,X_2,X_3$ is often denoted in the literature \cite{MR1218995, MR3035112} as $\mf{R}(e_1,e_2,e_3)$ or $\mf{R}_{e_1,e_2,e_3}$, where the $e_i$ correspond to the $r_i$ in the definition of $E_{r_1,r_2,r_3}$ \eqref{eq:algcurve}. This Lie algebra is well known and appears in a variety of contexts, most notably in connection to the Landau--Lifshitz equation. It arises as a subalgebra of the hidden symmetry algebra of the Landau--Lifshitz equation \cite{holod1987hidden}, as we have shown in \eqref{def:AB}. To the best of our knowledge, it has not been previously identified as an elliptic current algebra. This perspective provides a natural explanation for properties of $\mf{R}_{e_1,e_2,e_3}$, such as the quasi-finiteness established in  \cite[Theorem 21]{MR2225105}.
\end{Remark}

\section{Conclusion and outlook}
We have shown that the automorphic Lie algebras $(\g\otimes_{\CC}\ot)^{D_2}$ based on a complex reductive Lie algebra $\g$, are current algebras for any punctured complex torus $\TT$ (cf. Theorem \ref{Thm2}). In this construction, the group $D_2$ acts by inner automorphisms on $\g$ (factoring through $PGL(2,\CC)$) and acts faithfully and fixed-point freely on $T$. 

The special case of $\g=\mf{sl}(2,\CC)$ has notable applications in the theory of integrable systems. When there is a single $D_2$-orbit of punctures, the corresponding algebra arises in the context of the Wahlquist--Estabrook (WE) algebra of the Landau--Lifshitz equation and the prolongation algebra of the Krichever--Novikov equation. Two or three orbits of punctures correspond to Lie algebras introduced by Uglov and Holod, respectively, again in the context of integrable systems.  In particular, the normal forms of these algebras show that they -- originally defined either via generators and relations (Uglov) or a basis over $\CC$ (Holod) -- are isomorphic to $\mf{sl}(2,R)$ for a suitable ring of elliptic functions $R$ (cf. Theorems \ref{thm:uglov} and \ref{thm:Holod}). 

We have explicitly implemented the automorphic Lie algebra framework in the context of an $n$-component generalisation of the Landau--Lifshitz equation by Golubchik and Sokolov \eqref{eq:PDE} in the case $n=3$ (Section \ref{sec:LL}). In particular, we showed that the Lax matrices $M$ and $N$ \eqref{eq:ZCR} of this equation for $n=3$,  also admit a natural interpretation in terms of an automorphic Lie algebra based on $\mf{so}_{3,1}$. A natural question is whether this perspective extends to $n>3$: can the WE algebra of \eqref{eq:PDE} be realised as an automorphic Lie algebra with base Lie algebra $\mf{so}_{n,1}$ (or a different Lie algebra) and on a higher-genus algebraic curve?

The above WE algebra is closely related to a Lie algebra $\g(n)$ \eqref{Relations}. The realisation of $\g(3)$ as an automorphic Lie algebra reveals a transparent structure: it is isomorphic to $\mf{sl}(2,\CC[\wp,\wp'])$. It would be interesting to explore whether analogous results hold for $\g(n)$ with $n>3$, and whether similar normal forms can be found. Specifically, it would be interesting to investigate whether the Lie algebras $\g(n)$ for $n>3$ are also isomorphic to current algebras on higher-genus curves. 
Realising general WE algebras as automorphic Lie algebras could provide a deeper understanding of their structural properties, including representation-theoretic aspects, and may offer new perspectives in the classification of integrable PDEs.

Finally, given that the automorphic Lie algebras $(\mf{sl}(2,\CC)\otimes_{\CC}\ot)^{D_2}$, for different numbers of orbits of punctures, appear in (quantum) integrable systems, it is natural to ask whether other infinite-dimensional Lie algebras in this context also admit such a realisation. Specifically, one could ask whether they can be realised as a $D_2$-automorphic Lie algebra with additional punctures. This would provide new and transparent characterisations of these algebras.\\

\noindent \textbf{Acknowledgements} 
We would like to thank Alexander Veselov for helpful suggestions related to the construction of the intertwiner $\Omega(z)$, and Vincent Knibbeler for encouraging feedback. We also thank Alec Cooper for many stimulating discussions, \changes{and finally the anonymous referees for their helpful comments}.\\

\noindent \textbf{Data Availability}  Data sharing not applicable to this article as no datasets were generated or analysed in this work.\\

\noindent \textbf{Conflict of interest}  The authors declare they have no financial or conflict of interest.

\bibliographystyle{abbrvurl}

\begin{thebibliography}{10}

\bibitem{babelon2003introduction}
O.~Babelon, D.~Bernard, and M.~Talon.
\newblock {\em Introduction to classical integrable systems}.
\newblock Cambridge University Press, 2003.

\bibitem{MR2189890}
M.~J. Balakhnev.
\newblock The vector generalization of the {L}andau-{L}ifshitz equation:
  {B}\"acklund transformation and solutions.
\newblock {\em Appl. Math. Lett.}, 18(12):1363--1372, 2005.
\newblock \href {https://doi.org/10.1016/j.aml.2005.02.021}
  {\path{https://doi:10.1016/j.aml.2005.02.021}}.

\bibitem{10.1063/1.530700}
M.~Bremner.
\newblock Universal central extensions of elliptic affine Lie algebras.
\newblock {\em J. Math. Phys.}, 35(12):6685--6692, 1994.
\newblock \href{https://doi.org/10.1063/1.530700}
  {\path{https://doi.org/10.1063/1.530700}}.

\bibitem{MR1220945}
A.~L. Carey, K.~C. Hannabuss, L.~J. Mason, and M.~A. Singer.
\newblock The {L}andau-{L}ifshitz equation, elliptic curves and the {W}ard
  transform.
\newblock {\em Comm. Math. Phys.}, 154(1):25--47, 1993.
\newblock \href {https://doi.org/10.1007/BF02096830}
  {\path{https://doi.org/10.1007/BF02096830}}.

\bibitem{chandrasekharan2012elliptic}
K.~Chandrasekharan.
\newblock {\em Elliptic functions}, volume 281.
\newblock Springer, 1985.

\bibitem{Cherednik1983}
I.~V. Cherednik.
\newblock B\"acklund--Darboux transformations for classical Yang--Baxter bundles.
\newblock {\em Funct. Anal. Appl.}, 17(2):199--211, 1983.

\bibitem{collingwood1993nilpotent}
D.~H. Collingwood and W.~M. McGovern.
\newblock {\em Nilpotent orbits in semisimple {L}ie algebra: an introduction}.
\newblock CRC Press, 1993.

\bibitem{MR4712421}
D.~D. Duffield, V.~Knibbeler, and S.~Lombardo.
\newblock Wild local structures of automorphic {L}ie algebras.
\newblock {\em Algebr. Represent. Theory}, 27(1):305--331, 2024.
\newblock \href {https://doi.org/10.1007/s10468-023-10208-y}
  {\path{https://doi:10.1007/s10468-023-10208-y}}.

\bibitem{MR650945}
M.~Eichler and D.~Zagier.
\newblock On the zeros of the {W}eierstrass {$\wp$}-function.
\newblock {\em Math. Ann.}, 258(4):399--407, 1981/82.
\newblock \href {https://doi.org/10.1007/BF01453974}
  {\path{https://doi:10.1007/BF01453974}}.

\bibitem{MR2183958}
A.~Fialowski and M.~Schlichenmaier.
\newblock Global geometric deformations of current algebras as
  {K}richever-{N}ovikov type algebras.
\newblock {\em Comm. Math. Phys.}, 260(3):579--612, 2005.
\newblock \href {https://doi.org/10.1007/s00220-005-1423-5}
  {\path{https://doi:10.1007/s00220-005-1423-5}}.

\bibitem{fulton2013representation}
W.~Fulton and J.~Harris.
\newblock {\em Representation theory: a first course}, volume 129.
\newblock Springer Science \& Business Media, 2013.

\bibitem{MR1821313}
I.~Z. Golubchik and V.~V. Sokolov.
\newblock A multicomponent generalization of the hierarchy of the
  {L}andau-{L}ifshits equation.
\newblock {\em Teoret. Mat. Fiz.}, 124(1):62--71, 2000.
\newblock \href {https://doi.org/10.1007/BF02551067}
  {\path{https://doi:10.1007/BF02551067}}.

\bibitem{hartshorne1977algebraic}
R.~Hartshorne.
\newblock {\em Algebraic geometry}, volume~52.
\newblock Springer Science \& Business Media, 1977.

\bibitem{Holod1984}
P. Holod, \textit{Hamiltonian systems connected with anisotropic affine Lie algebras and higher Landau-Lifshitz equations}, Reports of the Academy of Sciences of the Ukrainian SSR, Series A, No. 5 (1984), pp. 5--8.

\bibitem{holod1987hidden}
P.~I. Holod.
\newblock Hidden symmetry of the {L}andau--{L}ifshitz equation, hierarchy of
  higher equations, and the dual equation for an asymmetric chiral field.
\newblock {\em Teoreticheskaya i Matematicheskaya Fizika}, 70(1):18--29, 1987.

\bibitem{humphreys1972introduction}
J.~E. Humphreys.
\newblock {\em Introduction to Lie algebras and representation theory},
  volume~9.
\newblock Springer-Verlag, 1972.

\bibitem{husemoller2004elliptic}
D.~Husem{\"o}ller.
\newblock {\em Elliptic curves}.
\newblock Springer, 2004.

\bibitem{MR2225105}
S.~Igonin.
\newblock Coverings and fundamental algebras for partial differential
  equations.
\newblock {\em J. Geom. Phys.}, 56(6):939--998, 2006.
\newblock \href {https://doi.org/10.1016/j.geomphys.2005.06.001}
  {\path{https://doi:10.1016/j.geomphys.2005.06.001}}.

\bibitem{MR4056196}
S.~Igonin and G.~Manno.
\newblock On {L}ie algebras responsible for integrability of
  {$(1+1)$}-dimensional scalar evolution {PDE}s.
\newblock {\em J. Geom. Phys.}, 150:103596, 24, 2020.
\newblock \href {https://doi.org/10.1016/j.geomphys.2020.103596}
  {\path{https://doi:10.1016/j.geomphys.2020.103596}}.

\bibitem{MR1947240}
S.~Igonin and R.~Martini.
\newblock Prolongation structure of the {K}richever-{N}ovikov equation.
\newblock {\em J. Phys. A}, 35(46):9801--9810, 2002.
\newblock \href {https://doi.org/10.1088/0305-4470/35/46/306}
  {\path{https://doi:10.1088/0305-4470/35/46/306}}.

\bibitem{MR3035112}
S.~Igonin, J.~van~de Leur, G.~Manno, and V.~Trushkov.
\newblock Infinite-dimensional prolongation {L}ie algebras and multicomponent
  {L}andau-{L}ifshitz systems associated with higher genus curves.
\newblock {\em J. Geom. Phys.}, 68:1--26, 2013.
\newblock \href {https://doi.org/10.1016/j.geomphys.2013.02.002}
  {\path{https://doi:10.1016/j.geomphys.2013.02.002}}.

\bibitem{MR3350266}
S.~Kharchev and A.~Zabrodin.
\newblock Theta vocabulary {I}.
\newblock {\em J. Geom. Phys.}, 94:19--31, 2015.
\newblock \href {https://doi.org/10.1016/j.geomphys.2015.03.010}
  {\path{https://doi:10.1016/j.geomphys.2015.03.010}}.

\bibitem{knibbeler2025uniform}
V.~Knibbeler.
\newblock A uniform construction of {C}hevalley normal forms for automorphic
  {L}ie algebras on the {R}iemann sphere.
\newblock {\em arXiv preprint arXiv:2503.17801}, 2025.
\newblock {\url{https://arxiv.org/abs/2503.17801}}

\bibitem{MR4808734}
V.~Knibbeler, S.~Lombardo, and C.~Oelen.
\newblock A classification of automorphic {L}ie algebras on complex tori.
\newblock {\em Proc. Edinb. Math. Soc. (2)}, 67(3):947--989, 2024.
\newblock \href {https://doi.org/10.1017/S0013091524000324}
  {\path{https://doi:10.1017/S0013091524000324}}.

\bibitem{MR3254844}
V.~Knibbeler, S.~Lombardo, and J.~A. Sanders.
\newblock Automorphic {L}ie algebras with dihedral symmetry.
\newblock {\em J. Phys. A}, 47(36):365201, 19, 2014.
\newblock \href {https://doi.org/10.1088/1751-8113/47/36/365201}
  {\path{https://doi:10.1088/1751-8113/47/36/365201}}.

\bibitem{MR3682219}
V.~Knibbeler, S.~Lombardo, and J.~A. Sanders.
\newblock Higher-dimensional automorphic {L}ie algebras.
\newblock {\em Found. Comput. Math.}, 17(4):987--1035, 2017.
\newblock \href {https://doi.org/10.1007/s10208-016-9312-1}
  {\path{https://doi:10.1007/s10208-016-9312-1}}.

\bibitem{MR4072405}
V.~Knibbeler, S.~Lombardo, and J.~A. Sanders.
\newblock Automorphic {L}ie algebras and cohomology of root systems.
\newblock {\em J. Lie Theory}, 30(1):59--84, 2020.

\bibitem{MR4142334}
V.~Knibbeler, S.~Lombardo, and J.~A. Sanders.
\newblock Hereditary automorphic {L}ie algebras.
\newblock {\em Commun. Contemp. Math.}, 22(8):1950076, 32, 2020.
\newblock \href {https://doi.org/10.1142/S0219199719500767}
  {\path{https://doi:10.1142/S0219199719500767}}.

\bibitem{MR4565687}
V.~Knibbeler, S.~Lombardo, and A.~P. Veselov.
\newblock Automorphic {L}ie algebras and modular forms.
\newblock {\em Int. Math. Res. Not. IMRN}, (6):5209--5262, 2023.
\newblock \href {https://doi.org/10.1093/imrn/rnab376}
  {\path{https://doi:10.1093/imrn/rnab376}}.

\bibitem{MR902293}
I.~M. Krichever and S.~P. Novikov.
\newblock Algebras of {V}irasoro type, {R}iemann surfaces and the structures of soliton theory.
\newblock {\em Funktsional. Anal. i Prilozhen.}, 21(2):46--63, 1987.

\bibitem{MR2090010}
S.~Lombardo and A.~V. Mikhailov.
\newblock Reductions of integrable equations: dihedral group.
\newblock {\em J. Phys. A}, 37(31):7727--7742, 2004.
\newblock \href {https://doi.org/10.1088/0305-4470/37/31/006}
  {\path{https://doi:10.1088/0305-4470/37/31/006}}.

\bibitem{MR2166845}
S.~Lombardo and A.~V. Mikhailov.
\newblock Reduction groups and automorphic {L}ie algebras.
\newblock {\em Comm. Math. Phys.}, 258(1):179--202, 2005.
\newblock \href {https://doi.org/10.1007/s00220-005-1334-5}
  {\path{https://doi:10.1007/s00220-005-1334-5}}.

\bibitem{MR2718933}
S.~Lombardo and J.~A. Sanders.
\newblock On the classification of automorphic {L}ie algebras.
\newblock {\em Comm. Math. Phys.}, 299(3):793--824, 2010.
\newblock \href {https://doi.org/10.1007/s00220-010-1092-x}
  {\path{https://doi:10.1007/s00220-010-1092-x}}.

\bibitem{MR4866321}
H.~McKean and V.~Moll.
\newblock {\em Elliptic curves---function theory, geometry, arithmetic}.
\newblock Cambridge Mathematical Library. Cambridge University Press,
  Cambridge, 2025.
  
\bibitem{mikhailov1979integrability}
A.~V. Mikhailov.
\newblock Integrability of a two-dimensional generalization of the {T}oda
  chain.
\newblock {\em JETP Lett}, 30(7):414--418, 1979.

\bibitem{mikhailov1980reduction}
A.~V. Mikhailov.
\newblock Reduction in integrable systems. {T}he reduction group.
\newblock {\em JETP Lett}, 32(2):187--192, 1980.

\bibitem{miranda1995algebraic}
R.~Miranda.
\newblock {\em Algebraic curves and {R}iemann surfaces}, volume~5.
\newblock American Mathematical Soc., 1995.

\bibitem{oelen2022automorphic}
C.~Oelen.
\newblock {\em Automorphic {L}ie algebras on complex tori}.
\newblock PhD thesis, Loughborough University, 2022.

\bibitem{MR861265}
A.~G. Reiman and M.~A. Semenov-Tyan-Shanskii.
\newblock Lie algebras and {L}ax equations with a spectral parameter on an
  elliptic curve.
\newblock {\em Zap. Nauchn. Sem. Leningrad. Otdel. Mat. Inst. Steklov. (LOMI)},
  150:104--118, 221, 1986.
\newblock \href {https://doi.org/10.1007/BF01099195}
  {\path{https://doi:10.1007/BF01099195}}.

\bibitem{MR1218995}
G.~H.~M. Roelofs and R.~Martini.
\newblock Prolongation structure of the {L}andau-{L}ifshitz equation.
\newblock {\em J. Math. Phys.}, 34(6):2394--2399, 1993.
\newblock \href {https://doi.org/10.1063/1.530124}
  {\path{https://doi:10.1063/1.530124}}.

\bibitem{MR3186310}
W.~Schlag.
\newblock {\em A course in complex analysis and {R}iemann surfaces}, volume 154
  of {\em Graduate Studies in Mathematics}.
\newblock American Mathematical Society, Providence, RI, 2014.
\newblock \href {https://doi.org/10.1090/gsm/154} {\path{https://doi:10.1090/gsm/154}}.

\bibitem{MR1082031}
O.~K. She\u{\i}nman.
\newblock Elliptic affine {L}ie algebras.
\newblock {\em Funktsional. Anal. i Prilozhen.}, 24(3):51--61, 96, 1990.
\newblock \href{https://doi.org/10.1007/BF01077962}
  {\path{https://doi.org/10.1007/BF01077962}}.

\bibitem{sklyanin1979complete}
E.~K. Sklyanin.
\newblock On complete integrability of the {L}andau--{L}ifshitz equation.
\newblock Preprint LOMI-E-79-3, 1979.


\bibitem{MR1852644}
T.~Skrypnyk.
\newblock Quasigraded {L}ie algebras on hyperelliptic curves and classical
  integrable systems.
\newblock {\em J. Math. Phys.}, 42(9):4570--4581, 2001.
\newblock \href {https://doi.org/10.1063/1.1379066}
  {\path{https://doi:10.1063/1.1379066}}.
  
\bibitem{10.1063/1.4824152}
T.~Skrypnyk.
\newblock Infinite-dimensional Lie algebras, classical r-matrices, and Lax operators: Two approaches.
\newblock {\em J. Math. Phys.}, 54(10):103507, 2013.
\newblock \href{https://doi.org/10.1063/1.4824152}
  {\path{https://doi.org/10.1063/1.4824152}}.  

\bibitem{MR2105209}
T.~Skrypnyk.
\newblock Deformations of loop algebras and integrable systems: hierarchies of
  integrable equations.
\newblock {\em J. Math. Phys.}, 45(12):4578--4595, 2004.
\newblock \href {https://doi.org/10.1063/1.1804229}
  {\path{https://doi:10.1063/1.1804229}}.
  
\bibitem{SKRYPNYK2023116102}
T.~Skrypnyk.
\newblock Elliptic Gaudin-type model in an external magnetic field and modified algebraic Bethe ansatz.
\newblock {\em Nuclear Phys. B}, 988:116102, 2023.
\newblock \href{https://doi.org/10.1016/j.nuclphysb.2023.116102}
  {\path{https://doi.org/10.1016/j.nuclphysb.2023.116102}}.

\bibitem{MR2920470}
T.~Skrypnyk.
\newblock Quasi-periodic functions on the torus and {$sl(n)$}-elliptic {L}ie
  algebra.
\newblock {\em J. Math. Phys.}, 53(2):023502, 19, 2012.
\newblock \href {https://doi.org/10.1063/1.3681211}
  {\path{https://doi:10.1063/1.3681211}}.

\bibitem{MR4855733}
T.~Skrypnyk.
\newblock Reduction in soliton hierarchies, elliptic {L}ie algebras of {H}olod
  and {K}richever-{N}ovikov-type equations.
\newblock {\em Nonlinearity}, 38(2):Paper No. 025007, 27, 2025.
\newblock \href {https://doi.org/10.1088/1361-6544/ada67e}
  {\path{https://doi:10.1088/1361-6544/ada67e}}.

\bibitem{MR1229896}
D.~B. Uglov.
\newblock The quantum bialgebra associated with the eight-vertex {$R$}-matrix.
\newblock {\em Lett. Math. Phys.}, 28(2):139--141, 1993.
\newblock \href {https://doi.org/10.1007/BF00750306}
  {\path{https://doi:10.1007/BF00750306}}.

\bibitem{MR1276451}
D.~B. Uglov.
\newblock The {L}ie algebra of the {$sl(2,{\mathbb{C} })$}-valued automorphic
  functions on a torus.
\newblock {\em Lett. Math. Phys.}, 31(1):65--76, 1994.
\newblock \href {https://doi.org/10.1007/BF00751172}
  {\path{https://doi:10.1007/BF00751172}}.

\bibitem{MR358111}
H.~D. Wahlquist and F.~B. Estabrook.
\newblock Prolongation structures of nonlinear evolution equations.
\newblock {\em J. Math. Phys.}, 16:1--7, 1975.
\newblock \href {https://doi.org/10.1063/1.522396}
  {\path{https://doi:10.1063/1.522396}}.

\bibitem{MR2047619}
J.~Yi.
\newblock Theta-function identities and the explicit formulas for
  theta-function and their applications.
\newblock {\em J. Math. Anal. Appl.}, 292(2):381--400, 2004.
\newblock \href {https://doi.org/10.1016/j.jmaa.2003.12.009}
  {\path{https://doi:10.1016/j.jmaa.2003.12.009}}.

\bibitem{MR725411}
V.~E. Zakharov and A.~V. Mikhailov.
\newblock The method of the inverse scattering problem with spectral parameter
  on an algebraic curve.
\newblock {\em Funktsional. Anal. i Prilozhen.}, 17(4):1--6, 1983.
\newblock \href {https://doi.org/10.1007/BF01076715}
  {\path{https://doi.org/10.1007/BF01076715}}.


\end{thebibliography}

\end{document}